\documentclass[11pt, reqno]{amsart}

\usepackage{style}

\title[Instability of an anisotropic micropolar fluid]{Anisotropic micropolar fluids subject to a uniform microtorque: the unstable case}

\author{Antoine Remond-Tiedrez}
\address{
Department of Mathematical Sciences\\
Carnegie Mellon University\\
Pittsburgh, PA 15213, USA
}
\email[A. Remond-Tiedrez]{aremondt@andrew.cmu.edu}

\author{Ian Tice}
\address{
Department of Mathematical Sciences\\
Carnegie Mellon University\\
Pittsburgh, PA 15213, USA
}
\email[I. Tice]{iantice@andrew.cmu.edu}
\thanks{I. Tice was supported by an NSF CAREER Grant (DMS \#1653161).}



\begin{document}

	\begin{abstract}
		We study a three-dimensional, incompressible, viscous, micropolar fluid with anisotropic microstructure on a periodic domain.
		Subject to a uniform microtorque, this system admits a unique nontrivial equilibrium.
		We prove that this equilibrium is nonlinearly unstable. Our proof relies on a nonlinear bootstrap instability argument
		which uses control of higher-order norms to identify the instability at the $L^2$ level.
	\end{abstract}

	\maketitle


\section{Introduction}
\label{sec:intro}

\vspace{0.75em}
\subsection{Brief discussion of the model}
\label{sec:pre_intro}
	Micropolar fluids were introduced by Eringen in \cite{eringen_first} as part of an effort to describe microcontinuum mechanics, which extend classical continuum mechanics by
	taking into account the effects of microstructure present in the medium.
	For viscous, incompressible continua, this results in a model in which the incompressible Navier-Stokes equations are coupled to an evolution equation for the rigid microstructure present at every point of the continuum.
	This theory can be used to describe aerosols and colloidal suspensions such as those appearing in biological fluids \cite{maurya_biological_fluids},
	blood flow \cite{ramkissoon_air_bubble_blood_flow, beg_al_blood_flow, mekheimer_kot_blood_flow},
	lubrication \cite{allen_kline_lubrication, bayada_lukaszewicz_lubrication, rajasekhar_nicodemus_sharma_lubrication}
	and in particular the lubrication of human joints \cite{sinha_singh_prasad_human_joints},
	liquid crystals \cite{eringen_first, lhuillier_rey_liquid_crystals, gay_balmaz_ratiu_tronci_liquid_crystals}, and ferromagnetic fluids \cite{nochetto_salgado_tomas_ferrohydrodynamics}.
	
	We postpone a more thorough discussion of the model until Section \ref{sec:discussion} and here provide only a brief overview sufficient to state the main result.  The variables needed to describe the state of a micropolar fluid at a point in three-space and time are as follows: the fluid velocity is a vector $u \in \R^3$, the fluid pressure is a scalar $p \in \R$, the microstructure's angular velocity is a vector $\omega \in \R^3$, and the microstructure's inertia tensor is a positive definite symmetric matrix $J \in \R^{3\times 3}$.  We study \emph{homogeneous} micropolar fluids, which means that the microstructures at any two points of the fluid are equal up to a proper rotation.  In turn, this means that the microinertia tensors at any two points of the fluid are equal up to conjugation.  Note that the shape of the microstructure determines the inertia tensor, but the converse fails in the sense that the same inertia tensor may be achieved by differently shaped microstructure.  
	
	We restrict our attention to problems in which the microinertia plays a significant role, and so in this paper we only consider \emph{anisotropic} micropolar fluids for which the microinertia tensor is not isotropic, i.e. $J$ has at least two distinct eigenvalues.  In fact,  we study micropolar fluids whose microstructure has an \emph{inertial axis of symmetry}, which means that the microinertia $J$ has a repeated eigenvalue.  More concretely: there are some physical constants $\lambda, \nu > 0$ which depend on the microstructure such that, at every point, $J$ is a symmetric matrix with spectrum $\{\lambda, \lambda, \nu\}$.  This is in some sense the intermediate case between the case of isotropic microstructure where the microinertia has a repeated eigenvalue of multiplicity three and the ``fully'' anisotropic case where the microstructure has three distinct eigenvalues.
	
	The equations of motion related to these quantities in the periodic spatial domain $\T^3 = \R^3 / {\brac{2\pi\Z}}^3$, subject to an external microtorque $\tau e_3$, read:
		\begin{subnumcases}{}
		\pdt u + \brac{u \cdot \nabla} u = \tilde{\mu} \Delta u + \kappa \nabla\times\omega - \nabla p
		&on $\brac{0, T} \times \T^3$,
		\label{eq:stat_prob_lin_mom}\\
		\nabla\cdot u = 0
		&on $\brac{0, T} \times \T^3$,
		\label{eq:stat_prob_incompress}\\
		J\brac{ \pdt\omega + \brac{u\cdot\nabla}\omega } + \omega\times J\omega
			= \kappa\nabla\times u - 2\kappa\omega + \brac{\tilde{\alpha} - \tilde{\gamma}} \nabla\brac{\nabla\cdot\omega} + \tilde{\gamma} \Delta\omega + \tau e_3
		\hspace{-1.2em}&on $\brac{0, T} \times \T^3$, { \hspace{1.5em} }
		\label{eq:stat_prob_ang_mom}\\
		\pdt J + \brac{u\cdot \nabla} J = \sbrac{\Omega, J}
		&on $\brac{0, T} \times \T^3$,
		\label{eq:stat_prob_microinertia}
	\end{subnumcases}
	where $\sbrac{\,\cdot\,,\,\cdot\,}$ denotes the matrix commutator, $\tilde{\mu}$, $\kappa$, $\tilde{\alpha}$, and $\tilde{\gamma}$ are physical constants related to viscosity, $\tau$ denotes the magnitude of the microtorque,
	and $\Omega$ is the 3-by-3 antisymmetric matrix identified with $\omega$ via the identity $\Omega v = \omega\times v$ for every $v\in\R^3$. 

	We have chosen to consider the situation in which external forces are absent and the external microtorque is constant, namely equal to $\tau e_3$ for some fixed $\tau > 0$.
	Note that the choice of $e_3$ as the direction of the microtorque may be made without loss of generality since the equations are equivariant under proper rotations,
	in the sense that if $\brac{u, p, \omega, J}$ is a solution of \eqref{eq:stat_prob_lin_mom}--\eqref{eq:stat_prob_microinertia} then, for any $\mathcal{R}\in SO\brac{3}$,
	$\brac{u, p, \mathcal{R}\omega, \mathcal{R}J\mathcal{R}^T}$ is a solution of \eqref{eq:stat_prob_lin_mom}--\eqref{eq:stat_prob_microinertia} \emph{provided} that the external torque
	$\tau e_3$ is replaced by $\tau \mathcal{R} e_3$.

	There are two ways to motivate our choice to have no external forces and a constant microtorque.
	On one hand, it is reminiscent of certain chiral active fluids constituted of self-spinning particles which continually pump energy into the system
	\cite{banerjee_souslov_et_at_chiral_active_fluids}, as our constant microtorque does.
	On the other hand, this choice of an external force -- external microtorque pair is motivated by the dearth of analytical results on anisotropic micropolar fluids.
	It is indeed natural, as a first step in the study of non-trivial equilibria of anisotropic micropolar fluids, to consider a simple external force -- external microtorque pair yielding
	non-trivial equilibria for the angular velocity $\omega$ and the microinertia $J$. The simplest nonzero such pair is precisely our choice of $\brac{0, \tau e_3}$.

	Let us now turn to the aforementioned equilibrium.  Due to the uniform microtorque, the system admits a  nontrivial equilibrium.  At equilibrium the fluid velocity is quiescent ($u_{eq}=0$), the pressure is null ($p_{eq} = 0$), the angular velocity is aligned with the microtorque ($\omega_{eq} = \frac{\tau}{2\kappa} e_3$),
	and the inertial axis of symmetry of the microstructure is aligned with the microtorque such that the microinertia is $J_{eq} = \diag (\lambda, \lambda, \nu)$.
	
	\begin{figure}[h!]
		\centering
		\begin{subfigure}[h!]{0.45\textwidth}
			\centering
			\includegraphics{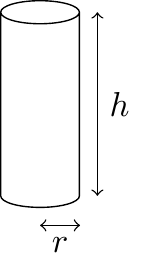}
			\caption{This rigid body is inertially oblong if $h^2 > 6 r^2$.}
			\label{fig:oblong}
		\end{subfigure}
		\begin{subfigure}[h!]{0.45\textwidth}
			\centering
			\includegraphics{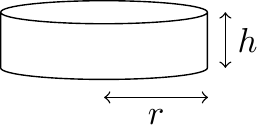}
			\caption{This rigid body is inertially oblate if $h^2 < 6 r^2$.}
			\label{fig:oblate}
		\end{subfigure}
		\caption{Two rigid bodies with uniform density which possess an inertial axis of symmetry.}
		\label{fig:rigid_bodies}
	\end{figure}

	Physically-motivated heuristics (which again we postpone until Section \ref{sec:discussion}) suggest that the stability of this equilibrium depends on the `shape' of the microstructure.  The heuristics suggest that  if the microinertia is inertially oblong, i.e. if $\lambda > \nu$, then the equilibrium is unstable,
	and that if the microinertia is inertially oblate, i.e. if $\nu > \lambda$, then the equilibrium is stable.
	This nomenclature is justified by the fact that for rigid bodies with an axis of symmetry and a uniform mass density, the notions of being oblong (or oblate),
	which essentially means that the body is longer (respectively shorter) along its axis of symmetry than it is wide across it, and being inertially oblong (respectivelly inertially oblate) coincide.
	Examples of inertially oblong and oblate rigid bodies are provided in \fref{Figure}{fig:rigid_bodies}.
	This paper deals with the instability of inertially oblong microstructure.
	In future work we will study the stability of inertially oblate microstructure.

\vspace{0.75em}
\subsection{Statement of the main result}
\label{sec:statement_problem_main_result}
	The main thrust of this paper is to prove that if the microstructure is inertially oblong, then the equilibrium is nonlinearly unstable in $L^2$.  A precise statement of the theorem may be found in  \fref{Theorem}{thm:bootstrap_instability}, but an informal statement of the result is the following. 

	\begin{thm}[$L^2$ instability of the equilibrium]
	\label{thm:main_result}
		Suppose that the microstructure is inertially oblong, i.e. suppose that $\lambda > \nu$,
		and let $X_{eq} = (u_{eq}, \omega_{eq}, J_{eq} ) = (0, \frac{\tau}{2\kappa} e_3, \diag(\lambda, \lambda, \nu))$ be the equilibrium solution
		of \eqref{eq:stat_prob_lin_mom}--\eqref{eq:stat_prob_microinertia}.
		Then $X_{eq}$ is nonlinearly unstable in $L^2$.
	\end{thm}
	Here the notion of nonlinear instability is the familiar one from dynamical systems: there exists a radius $\delta >0$ and a sequence of initial data $\{X_n^0\}_{n=0}^\infty$, converging to $X_{eq}$ in $L^2$, such that the solutions to \eqref{eq:stat_prob_lin_mom}--\eqref{eq:stat_prob_microinertia} starting from $X_n^0$ exit the ball $B(X_{eq},\delta)$ in finite time, depending on $n$.  

	Note that in \fref{Theorem}{thm:main_result} the pressure has disappeared from consideration. This is because the pressure plays only an auxiliary role in the equations
	and may be eliminated from \eqref{eq:stat_prob_lin_mom} by projecting onto the space of divergence-free vector fields.


\section{Background, preliminaries, and discussion}
\label{sec:discussion}

\subsection{Micropolar fluids}
\label{sec:micropolar}
	To the best of our knowledge, the anisotropic micropolar fluid model has not been studied in the PDE literature, so our aim in this subsection is to provide the reader with a brief overview of the model and its features.  We emphasize that it is a natural extension of the Navier-Stokes model, as it follows from the same principles of rational continuum mechanics. We refer to \cite{erigen_vol_1, erigen_vol_2} for a complete continuum mechanics derivation of the micropolar fluid model, and we refer to \cite{lukaszewicz_book} for a thorough discussion of the mathematical analysis of \emph{isotropic} micropolar fluids.  Throughout this discussion we will take the domain under consideration to be the (normalized) torus $\T^3 = \R^3 / {\brac{2\pi\Z}}^3$ and we will let $T\in\ocbrac{0,+\infty}$ denote our time horizon.  For the sake of brevity, in this subsection we will commit the usual crime of assuming all quantities are ``sufficiently regular'' to justify the written assertions. 

	Just as rational continuum mechanics begins with the postulation that there exists some flow map $\eta : \brac{0,T} \times \T^3 \to \T^3$ which describes the motion of the continuum,
	the micropolar theory posits the existence of an additional (Lagrangian) microrotation map $Q : \brac{0,T} \times \T^3 \to SO\brac{3}$
	which describes the rotation of the microstructure present at every point in the continuum.
	The pair $\brac{\eta, Q}$ thus provides a complete kinematic description of a micropolar continuum as illustrated in \fref{Figure}{fig:micro_structure}.

	\begin{figure}
		\centering
		\includegraphics{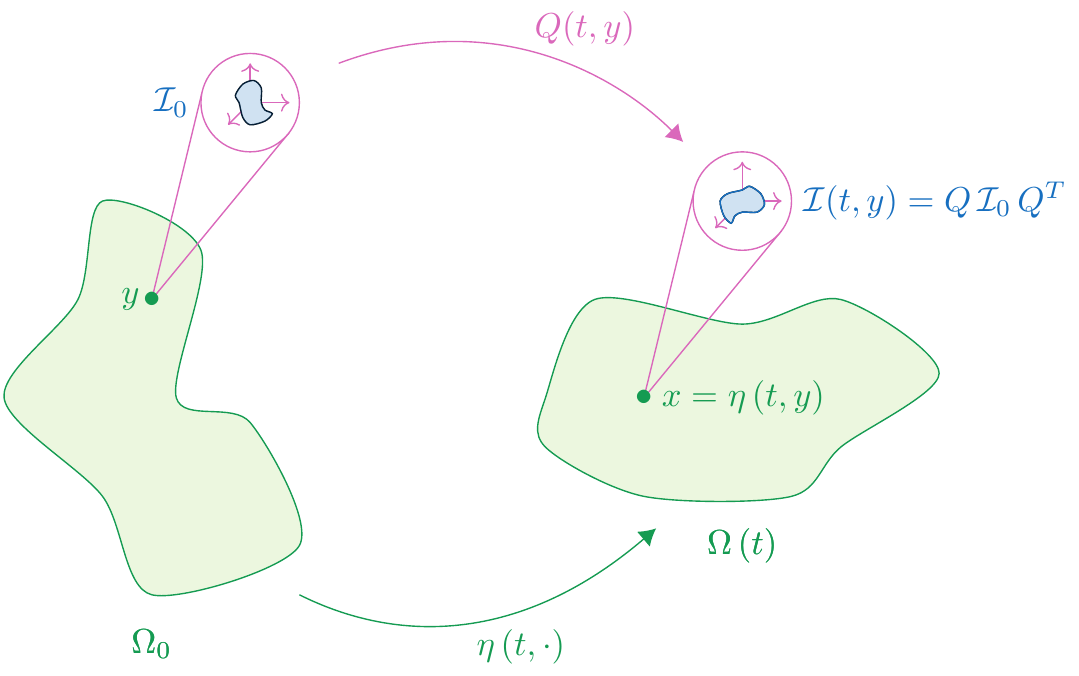}
		\caption[
			A depiction of how a subset of the micropolar continuum behaves under the joint flow of the flow map and the microrotation map.
		]{
			A depiction of how a subset $\Omega_0 \subseteq \T^3$ of the micropolar continuum behaves under the flow of $\eta$ and $Q$.
			$\Omega\brac{t} = \eta\brac{t, \Omega_0}$ is the image of $\Omega_0$ under the flow of $\eta$ and
			$y\in\Omega_0$ is a point in $\Omega_0$ at which the micropolar continuum has microinertia $\mathcal{I}_0$.
			At the point $x = \eta\brac{t, y}$ the microinertia is $\mathcal{I}\brac{t,y} = Q\brac{t} \,\mathcal{I}_0\, Q^t \brac{t}$
			since the microinertia transforms as a 2-tensor under the flow of the microrotation $Q$.
		}
		\label{fig:micro_structure}
	\end{figure}

	A word of warning: there are two ways to define the microrotation map and we have chosen here the convention that $Q$ is \emph{absolute}.
	Indeed, one may either define $Q$ to be the rotation of the microstructure with respect to its immediate environment,
	in which case $Q$ would be equal to the identity when the micropolar continuum undergoes rigid motions such as rotations,
	or one may define $Q$ to be the identity at time $t = 0$ and to be the absolute rotation underwent by the micropolar continuum thereafter.
	We choose the latter convention.
	In order to illustrate the physical interpretation of the microrotation map $Q$,
	\fref{Table}{tab:micro_rot_abs} contrasts the motions obtained for various simple expressions of $\eta$ and $Q$ 

	\begin{table}
		\centering
		\begin{tabular}[c]{cccc}
			\toprule
			Configuration at time $t=0$		&\multicolumn{3}{c}{\begin{tabular}{c}\includegraphics[width=0.2\textwidth]{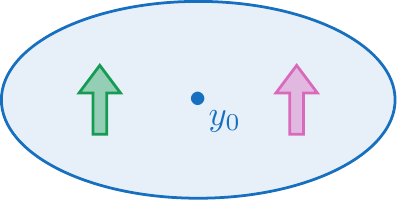}\end{tabular}}\\
			\toprule
								& Case 1			& Case 2			& Case 3		\\
			\toprule
			$\eta\brac{t,\cdot}$			& $e^{tR}\brac{\cdot - y_0}$	& $e^{tR}\brac{\cdot-y_0}$	& $I$			\\
			$Q\brac{t,y}$				& $e^{tR}$			& $I$				& $e^{tR}$		\\
			$u\brac{t,x}$				& $Rx$				& $Rx$				& $0$			\\
			$\frac{1}{2} \nabla\times u\brac{t,x}$	& $e_3$				& $e_3$				& $0$			\\
			$\omega\brac{t,x}$			& $e_3$				& $0$				& $e_3$			\\
			$\frac{1}{2}\nabla\times u - \omega$	& $0$				& $e_3$				& $-e_3$		\\
			\midrule
			Configuration at time $t = \frac{\pi}{2}$
				&\begin{tabular}{c}\includegraphics[width=0.1\textwidth]{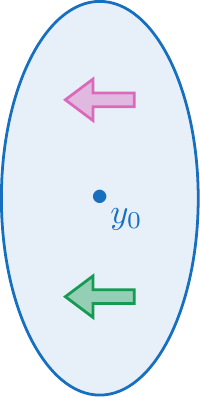}\end{tabular}
				&\begin{tabular}{c}\includegraphics[width=0.1\textwidth]{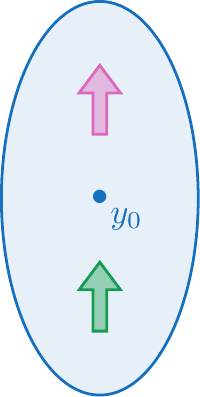}\end{tabular}
				&\begin{tabular}{c}\includegraphics[width=0.2\textwidth]{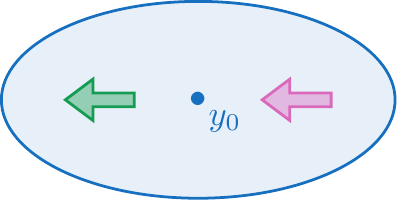}\end{tabular}\\
			\bottomrule
		\end{tabular}
		\caption[
			Three explicit examples of the motion of a micropolar continuum with the same initial configuration.
			These motions are chosen to be similar to emphasize that the microrotation $Q$ is an \emph{absolute} rotation.
		]{
			Three explicit examples of the motion of a micropolar continuum with the same initial configuration.
			These motions are chosen to be similar to emphasize that the microrotation $Q$ is an \emph{absolute} rotation.
			The figures shown correspond to cross-sections perpendicular to $e_3$,
			each colored arrow is a depiction of the orientation of the microstructure at that point,
			$y_0$ is some point in the micropolar continuum,
			and $R = e_2 \otimes e_1 - e_1 \otimes e_2$ corresponds to a (counter-clockwise) rotation by $\pi/2$ in the plane perpendicular to $e_3$.
		}
		\label{tab:micro_rot_abs}
	\end{table}

	Analogously to how the flow map $\eta$ is more conveniently characterized by its Eulerian velocity $u$,
	the microrotation map $Q$ is characterized by its Eulerian angular velocity $\omega$ where
	\begin{equation*}
		u(t,\,\cdot\,) = \pdt\eta(t,\,\cdot\,) \circ {\eta(t,\,\cdot\,)}\inv
		\text{ and }
		\omega(t,\,\cdot\,) = \vc \pdt Q(t,\,\cdot\,) Q^T (t,\,\cdot\,) \circ {\eta(t,\,\cdot\,)}\inv
	\end{equation*}
	and
	\begin{equation}
	\label{eq:def_vec}
		{\brac{\vc M}}_i = \frac{1}{2} \varepsilon_{aib} M_{ab} \text{ for any 3-by-3 matrix } M.
	\end{equation}
	Recall that the Levi-Civita symbol $\varepsilon_{ijk}$ is defined to be the sign of the permutation which maps $1 \mapsto i$, $2 \mapsto j$, and $3 \mapsto k$.
	Note that here, since $Q\in SO(3)$ we know that $\Omega(t) = \pdt Q(t) Q^T (t) \circ {\eta(t)}\inv$ is antisymmetric,
	and hence we may use the standard identification of a 3-by-3 antisymmetric matrix $A$ with a vector $a = \vc A$
	via $Av = a \times v$ for any $v\in\R^3$, where $\times$ denotes the usual cross product in $\R^3$.

	The derivation of the equations of motion for a micropolar continuum begins by postulating the conservation of mass, the balance of linear momentum, and the balance of angular momentum.  For micropolar continua the angular momentum is the sum of the macroscopic angular momentum, obtained from the fluid velocity $u$ and a choice of reference point in space, and the microscopic angular momentum $J \omega$.  Additionally, micropolar fluids conserve microinertia, which means that the Lagrangian microinertia $\mathcal{I}(t,\,\cdot\,) = J\circ\eta(t,\,\cdot\,)$ satisfies
	$\mathcal{I}(t,\,\cdot\,) = Q(t,\,\cdot\,) \, \mathcal{I}(0,\,\cdot\,) \, Q^T(t,\,\cdot\,)$.
	Differentiating in time yields $\pdt\mathcal{I} = \sbrac{\pdt Q Q^T, \mathcal{I}}$, where $\sbrac{\cdot,\cdot}$ denotes the matrix commutator.  We may rewrite this in Eulerian coordinates as
	\begin{equation}
	\label{eq:cons_microinertia}
		\pdt J + \brac{u\cdot\nabla} J = \sbrac{\Omega, J}.
	\end{equation}

	Note that a microinertia is \emph{physical} when its spectrum $\cbrac{\lambda_1, \lambda_2, \lambda_3}$ satisfies $\lambda_i \leqslant \frac{1}{2} \sum_{j=1}^3 \lambda_j = \frac{1}{2} \tr J$
	for $i=1,2,3$.
	This comes from the fact that we may compute the microinertia tensor of a rigid body of mass $M$ from the covariance matrix $V$ of its mass distribution via $J = M\brac{\brac{\tr V}I - V}$.
	The condition above on the eigenvalues of $J$ is then equivalent to requiring the physical condition that $V$ be positive semi-definite.

	For incompressible continua with constant density the conservation of mass reduces to the divergence-free condition
	\begin{equation}
	\label{eq:incompressibility}
		\nabla\cdot u = 0.
	\end{equation}
	Using \eqref{eq:cons_microinertia},
	the conservation of linear and angular momentum then respectively take the form
	\begin{equation}
	\label{eq:cons_lin_mom}
		\pdt u + \brac{u\cdot\nabla}u
		= \nabla\cdot T + f
	\end{equation}
	and
	\begin{equation}
	\label{eq:cons_ang_mom}
		\brac{\pdt + u\cdot\nabla} \brac{J\omega}
		= J\brac{\pdt \omega + \brac{u\cdot\nabla} \omega} + \omega\times J\omega
		= 2 \vc T + \nabla\cdot M + g
	\end{equation}
	where $T$ is the Cauchy stress tensor which expresses the internal forces exerted by the continuum on itself,
	$M$ is the couple stress tensor which expresses the internal microtorques exerted by the continuum on itself (and on its microstructure),
	and where $f$ and $g$ are the external forces and microtorques acting on the continuum, respectively.

	To close the system we continue along the path of rational mechanics which produces Navier-Stokes
	and postulate that some constitutive equations hold which determine the stresses $T$ and $M$
	in terms of the velocity $u$, the angular velocity $\omega$, and the pressure $p$.
	Analogously to how a Newtonian fluid is defined as a continuum for which the stress tensor is given by $T = \mu\symgrad u - pI$,
	a micropolar fluid is defined as a micropolar continuum for which
	\begin{equation}
	\label{eq:constitutive_equations}
		T = \mu\symgrad u - pI + \kappa\ten\brac{ \frac{1}{2} \nabla\times u - \omega}
		\text{ and }
		M = \alpha\brac{\nabla\cdot\omega}I + \beta\symgrad^0 \omega + \gamma\ten\nabla\times\omega,
	\end{equation}
	where:	$\symgrad$ denotes the symmetrized gradient defined by $\symgrad v = \nabla v + {\nabla v}^T$,
		$\ten$ is the inverse of $\vc$ introduced in \eqref{eq:def_vec} such that $\ten\brac{v}w = v\times w$ for every $v,w\in\R^3$,
		$\symgrad^0$ is the trace-free part of the symmetrized gradient defined by $\symgrad^0 v = \symgrad v - \frac{2}{3} \brac{\nabla\cdot v}I$, and
		$\mu$, $\kappa$, $\alpha$, $\beta$, $\gamma$ are physical constants commonly referred to as \emph{fluid viscosities}.  
	Note that, by contrast with classical fluids, the stress tensor $T$ is \emph{not} symmetric.

	The terms in $M$ are analogous to the terms one finds in the viscous stress tensor for a compressible fluid and have a similar physical interpretation.
	The most interesting novelty in the micropolar model is the coupling term $\kappa\ten\brac{ \frac{1}{2} \nabla\times u - \omega}$.
	It serves to induce a stress when there is a mismatch between the local rotation induced by the flow map and the rotation of the microstructure:
	see \fref{Table}{tab:micro_rot_abs} for some examples.
	Note that the coupling term  is not symmetric, and so it spoils the usual symmetry enjoyed by the stress tensor in standard continuum models.
	
	Finally, thermodynamical considerations, and in particular the Clausius-Duhem inequality, tell us that the quadratic form given by the dissipation
	\begin{equation*}
		T:\brac{\nabla u - \Omega} + M:\nabla\omega =
			\frac{\mu}{2} \abs{\symgrad u}^2
			+ 2 \kappa {\vbrac{\half\nabla\times u - \omega}}^2
			+ \alpha\abs{\nabla\cdot\omega}^2
			+ \frac{\beta}{2} \abs{\symgrad^0 \omega}^2
			+ 2 \gamma \abs{\nabla\times\omega}^2
	\end{equation*}
	must be positive-semidefinite, from which it follows that $\mu, \kappa, \alpha, \beta, \gamma \geqslant 0$.
	Note that in this paper we require that
	\begin{equation}
	\label{eq:positivity_requirement_on_viscosity_constants}
		\mu,\, \kappa,\, \alpha + \frac{4\beta}{3},\, \beta + \gamma > 0.
	\end{equation}
	In particular $\mu$ and $\kappa$ must be strictly positive but some of $\alpha$, $\beta$, and $\gamma$ may vanish.
	More precisely: if $\beta > 0$ then we allow $\alpha = \gamma = 0$ and if $\alpha, \gamma > 0$ then we allow $\beta = 0$.
	This requirement comes from the fact that
	\begin{equation*}
		\nabla\cdot M = \brac{\alpha + 4\beta/3} \nabla\brac{\nabla\cdot\omega} + \brac{\beta + \gamma} \brac{\Delta\omega - \nabla\brac{\nabla\cdot\omega}}
	\end{equation*}
	where the symbol of $\nabla\nabla\,\cdot\,$ is $-\abs{k}^2\proj_k$ and the symbol of $\Delta - \nabla\nabla\,\cdot\,$ is $-\abs{k}^2\proj_{k^\perp}$,
	therefore the contribution of the dissipation coming from $M$ is
	\begin{equation*}
		\int_{\T^3} \brac{\nabla\cdot M}\cdot\omega
		= \sum_{k\in\Z^3} -\abs{k}^2 \brac{
			\brac{\alpha + 4\beta/3} \abs{\proj_k \hat{\omega}}^2
			+ \brac{\beta+\gamma} \abs{\proj_{k^\perp} \hat{\omega}}^2
		}.
	\end{equation*}
	This dissipative term will then control $\norm{\nabla\omega}{L^2}$ precisely when $\alpha + 4\beta/3,\, \beta+\gamma > 0$.

	Putting \eqref{eq:cons_microinertia}, \eqref{eq:incompressibility}, \eqref{eq:cons_lin_mom}, and \eqref{eq:cons_ang_mom} together with \eqref{eq:constitutive_equations}
	yields \eqref{eq:stat_prob_lin_mom}--\eqref{eq:stat_prob_microinertia} when the external forces are taken to vanish and when the external microtorques are taken to be constant,
	namely $g = \tau e_3$ for some fixed $\tau > 0$.
	Note also that, for simplicity, we have defined $\tilde{\mu} = \mu + \kappa/2$, $\tilde{\alpha} = \alpha + 4\beta/3$, and $\tilde{\gamma} = \beta + \gamma$
	in \eqref{eq:stat_prob_lin_mom}--\eqref{eq:stat_prob_microinertia}.

	It is worth noting that this system is equivariant under Galilean transformations.
	More precisely: if $\brac{u, p, \omega, J}$ is a sufficiently regular solution of \eqref{eq:stat_prob_lin_mom}--\eqref{eq:stat_prob_microinertia}
	then $u_\text{avg} \defeq \fint_{\T^3} u$ is constant in time and
	\begin{equation*}
		\brac{0,T} \times \T^3 \ni \brac{t, y} \mapsto \brac{u - u_\text{avg}, p, \omega, J} \brac{t, y + u_\text{avg}}
	\end{equation*}
	also satisfies \eqref{eq:stat_prob_lin_mom}--\eqref{eq:stat_prob_microinertia}.
	We may therefore assume without loss of generality that $u$ has average zero at all times.  Similarly, since the pressure only appears in the equations with a gradient, we are free to posit that $p$ has average zero for all times.

\subsection{Previous work}
\label{sec:prev_work}
	
	Micropolar fluids have been extensively studied by the continuum mechanics community over the last fifty years and an exhaustive literature review is beyond the scope of this paper.  We restrict our attention to the mathematics literature here, in which case, to the best of our knowledge all results relate to  \emph{isotropic} microstructure, where the microinertia $J$ is a scalar multiple of the identity. In that case the precession term $\omega\times J\omega$ from \eqref{eq:stat_prob_ang_mom} vanishes
	and the entire equation \eqref{eq:stat_prob_microinertia} trivializes.  
	Note that in two dimensions the micro-inertia is a scalar, and therefore all micropolar fluids are isotropic.
	
	In two dimensions the problem is globally well-posed, as per \cite{lukaszewiscz_01} where global well-posedness and qualitative results on the long-time behaviour are obtained.  Some quantitative information on long-time behaviour is also known in two dimensions: for example, decay rates are obtained in \cite{dong_chen}.  The situation is more delicate in three dimensions, which is an unsurprising assertion in the setting of viscous fluids.  The first discussion of well-posedness in three dimensions is due to Galdi and Rionero \cite{galdi_rionero}. {\L}ukaszewicz then obtained weak solutions in \cite{lukaszewicz_90}
	and uniqueness of strong solutions in \cite{lukaszewicz_89}.
	More recent work has established global well-posedness for small data in critical Besov spaces \cite{chen_miao},
	in Besov-Morrey spaces \cite{ferreira_precioso}, and
	in the space of pseudomeasures \cite{ferreira_villamizar_roa},
	as well as derived blow-up criteria  \cite{yuan}.
	There is also an industry devoted to the study of micropolar fluids when one or more of the viscosity coefficients vanishes:  we refer to \cite{dong_zhang} for an illustrative example.

	Various extensions of the incompressible micropolar fluid model considered here have been studied.  For example, compressible models \cite{liu_zhang}, models coupled to heat transfer \cite{tarasinska, kalita_langa_lukaszewicz}, and models with coupled magnetic fields \cite{ahmadi_shahinpoor, rojas_medar_marko} have all been studied.  Again, to the best of our knowledge all of these works consider \emph{isotropic} micropolar fluids.

\subsection{Equilibria}
\label{sec:equilibria}

	In this section we describe the two classes of equilibria which arise as particular solutions of
	\eqref{eq:stat_prob_lin_mom}--\eqref{eq:stat_prob_microinertia}.
	A critical piece of this description is the following energy-dissipation relation:
	\begin{align}
		\Dt \int_{\T^3} \frac{1}{2} \abs{u}^2 &+ \frac{1}{2} J \brac{\omega - \omega_{eq}}\cdot\brac{\omega - \omega_{eq}} - \frac{1}{2} J \omega_{eq}\cdot\omega_{eq}
		\label{eq:relative_energy_dissipation_relation}
		\\
		&= - \int_{\T^3}
			\frac{\mu}{2} \abs{\symgrad u}^2
			+ 2 \kappa {\vbrac{\half\nabla\times u - \brac{\omega-\omega_{eq}}}}^2
			+ \alpha\abs{\nabla\cdot\omega}^2
			+ \frac{\beta}{2} \abs{\symgrad^0 \omega}^2
			+ 2 \gamma \abs{\nabla\times\omega}^2,
		\nonumber
	\end{align}
	where recall that $\omega_{eq} = \frac{\tau}{2\kappa} e_3$.
	This energy-dissipation relation is obtained by testing \eqref{eq:stat_prob_lin_mom} and \eqref{eq:stat_prob_ang_mom} against $u$ and $\omega-\omega_{eq}$ respectively and integrating by parts.
	For a full derivation, see \fref{Appendix}{sec:deriv_relative_energy_dissipation_relation}.
	With the relation \eqref{eq:relative_energy_dissipation_relation} in hand we may define two classes of equilibria.
	\begin{definition}
	\label{def:classes_of_equilibria}
		We say that a solution $\brac{u, p, \omega, J}$ of \eqref{eq:stat_prob_lin_mom}--\eqref{eq:stat_prob_microinertia} is an \emph{equilibrium} if $\pdt\brac{u, p, \omega, J} = 0$
		and we say that it is an \emph{energetic equilibrium} if $\Dt \mathcal{E}_\text{rel} = 0$ where the relative energy $\mathcal{E}_\text{rel}$
		is given as in \eqref{eq:relative_energy_dissipation_relation} by
		\begin{equation}
		\label{eq:relative_energy}
			\mathcal{E}_\text{rel} \brac{u, p, \omega, J}
			= \int_{\T^3} \frac{1}{2} \abs{u}^2 + \frac{1}{2} J \brac{\omega - \omega_{eq}} \cdot \brac{\omega - \omega_{eq}} - \frac{1}{2} J \omega_{eq} \cdot \omega_{eq}.
		\end{equation}
	\end{definition}

	There are two reasons why one might study the energetic equilibria introduced in \fref{Definition}{def:classes_of_equilibria}:
	(1)~they arise naturally as the stationary points of a Lyapunov functional and
	(2)~we believe that they play an essential role in characterizing the long-time behaviour of the system.

	We justify (1) now and postpone the justification of (2) until after the identification of the various equilibria is carried out in \fref{Proposition}{prop:identifying_eq}.
	Since the relative energy $\mathcal{E}_\text{rel}$ is both non-increasing in time and bounded below we may indeed view it as a Lyapunov functional.
	The observation that $\Dt\mathcal{E}_\text{rel} \leqslant 0$ follows immediately from \eqref{eq:relative_energy_dissipation_relation}
	and the boundedness from below of $\mathcal{E}_\text{rel}$ follows from the fact that the spectrum of the microinertia $J$ is invariant over time.

	More precisely: as described in \fref{Section}{sec:micropolar}, the conservation of microinertia for a \emph{homogeneous} micropolar fluid means that
	there exists some reference microinertia $J_\text{ref}$ to which $J(t,x)$ is similar at all times $0 \leqslant t < T$ and at every point $x\in\T^3$.
	Denoting by $\lambda_\text{max}$ the largest eigenvalue of $J_\text{ref}$ it follows that the only non-positive term in $\mathcal{E}_\text{rel}$ is bounded below:
	$-J\omega_{eq}\cdot\omega_{eq} \geqslant - \lambda_\text{max} \abs{\omega_{eq}}^2$, and hence $\mathcal{E}_\text{rel}$ itself is bounded below.

	We now identify all of the (sufficiently regular) equilibria which belong to each class as defined in \fref{Definition}{def:classes_of_equilibria}.
	Recall that we are considering a homogeneous micropolar fluid whose microstructure has an inertial axis of symmetry,
	which means that there are physical constants $\lambda, \nu > 0$ such that the microinertia has spectrum $\{\lambda, \lambda, \nu\}$.
	In particular this microinertia tensor is \emph{physical} precisely when $2\lambda \geqslant \nu \geqslant 0$.
	We will assume thereafter that strict inequalities hold, i.e. $2\lambda > \nu > 0$.
	This assumptions means that the microstructure is not degenerate, in the sense that it corresponds to a genuinely three-dimensional rigid body
	(as opposed to a degenerate rigid body which would be lower-dimensional, e.g. because it is flat in one or more directions).

	\begin{prop}
	\label{prop:identifying_eq}
		Let $\brac{u, p, \omega, J}$ be a sufficiently regular solution of \eqref{eq:stat_prob_lin_mom}--\eqref{eq:stat_prob_microinertia} where $u$ has average zero.
		\begin{enumerate}
			\item	If $\brac{u, p, \omega, J}$ is an equilibrium then $u=0$, $p=0$, $\omega = \omega_{eq} = \frac{\tau}{2\kappa}e_3$, and
				$J = \diag ( \lambda, \lambda, \nu ) = \lambda I_2 \oplus \nu$.
			\item	If $\brac{u, p, \omega, J}$ is an energetic equilibrium then either it is an equilibrium or $u=0$, $p=0$, $\omega = \omega_{eq}$, and
				$J = e^{t\frac{\tau}{2\kappa}R} \bar{J}(0) e^{-t\frac{\tau}{2\kappa}R} \oplus \lambda$
				where $R = \begin{pmatrix} 0 & -1 \\ 1 & 0 \end{pmatrix}$ and where the spectrum of $\bar{J}(0)$ is $\{\lambda, \nu\}$.
				Here `$\oplus\!$' denotes the direct sum of two linear operators, see Section \ref{sec:notation} to recall the precise definition.
		\end{enumerate}
	\end{prop}
	In simpler words \fref{Proposition}{prop:identifying_eq} says that for both equilibria and energetic equilibria the microstructure rotates in the direction of the imposed microtorque,
	with one crucial difference: the unique equilibrium corresponds to the inertial axis of symmetry of the microstructure being \emph{aligned} with the microtorque,
	giving rise to a \emph{constant} microinertia, whilst the energetic equilibria consist of an orbit where the inertial axis of symmetry rotates in the plane \emph{perpendicular} to the microtorque,
	giving rise to a \emph{periodic} microinertia (with period $4\pi\kappa / \tau$).

	\begin{proof}[Proof of \fref{Proposition}{prop:identifying_eq}]
			Since equilibria are energetic equilibria we begin by supposing that $\brac{u, p, \omega, J}$ is an energetic equilibrium.
			It follows from the energy-dissipation relation \eqref{eq:relative_energy_dissipation_relation} that the dissipation vanishes, i.e.
			\begin{equation*}
				\int_{\T^3}
					\frac{\mu}{2} \abs{\symgrad u}^2
					+ 2 \kappa {\vbrac{\half\nabla\times u - \brac{\omega-\omega_{eq}}}}^2
					+ \alpha\abs{\nabla\cdot\omega}^2
					+ \frac{\beta}{2} \abs{\symgrad^0 \omega}^2
					+ 2 \gamma \abs{\nabla\times\omega}^2
				= 0.
			\end{equation*}
			In particular: $\omega$ is constant and $u$ has constant curl.
			Coupling this with the fact that $u$ is divergence-free we deduce that $u$ is harmonic.
			Since $u$ has average zero, it follows that $u=0$, and hence that $p=0$ (recall that we require $p$ to have average zero) and $\omega = \omega_{eq}$.

			So now we know from \eqref{eq:stat_prob_ang_mom} that the precession term $\omega\times J\omega = {\brac{ \frac{\tau}{2\kappa} }}^2 e_3 \times J e_3$ vanishes,
			and hence $J$ has the block form $J = \bar{J} \oplus J_{33}$ for some 2-by-2 matrix $\bar{J}$.
			The conservation of microinertia \eqref{eq:stat_prob_microinertia} now becomes the ODE $\pdt J = \sbrac{\ten\omega_{eq}, J} = \frac{\tau}{2\kappa} \sbrac{R,\bar{J}\,} \oplus 0$
			which may be solved explicitly to yield $\bar{J}(t) = e^{t\frac{\tau}{2\kappa}R} \bar{J}(0) e^{-t\frac{\tau}{2\kappa}R}$ and $J_{33} (t) = J_{33} (0)$.

			There are now two cases to consider: either $\bar{J}$ has a repeated eigenvalue $\lambda$ or $\bar{J}$ has distinct eigenvalues $\lambda$ and $\nu$.
			Since $e^{t\frac{\tau}{2\kappa}R} \bar{J}(0) e^{-t\frac{\tau}{2\kappa}R}$ is constant in time if and only if $\bar{J}(0)$, and hence $\bar{J}(t)$, has a repeated eigenvalue,
			the result follows.
	\end{proof}

	As the next section suggests, we believe that the global attractors of \eqref{eq:stat_prob_lin_mom}--\eqref{eq:stat_prob_microinertia}
	may be characterized in terms of the equilibrium and the orbit of energetic equilibria. This is summarized in the conjecture below,
	which is the second reason why energetic equilibria are worthy of attention.
	\begin{conjecture}
	\label{conjecture} $\text{}$
	\begin{enumerate}
		\item	If the microstructure is inertially oblong, i.e. $\lambda > \nu$, then the orbit of energetic equilibria identified in \fref{Proposition}{prop:identifying_eq}
			is the global attractor of the system \eqref{eq:stat_prob_lin_mom}--\eqref{eq:stat_prob_microinertia}.
		\item	If the microstructure is inertially oblate, i.e. $\lambda < \nu$, then the equilibrium identified in \fref{Proposition}{prop:identifying_eq}
			is the global attractor of the system \eqref{eq:stat_prob_lin_mom}--\eqref{eq:stat_prob_microinertia}.
	\end{enumerate}	
	\end{conjecture}

	We note that attractors have been obtained in previous works in the context of two-dimensional isotropic micropolar fluids \cite{chen_chen_dong, lukaszewicz_tarasinska}.

	A depiction of the equilibrium and the energetic equilibria configurations of the microstructure can be found in \fref{Figure}{fig:conjecture},
	where we also label each configuration with its relevant conjectured long-time behaviour.

	\begin{figure}
		\centering
		\begin{subfigure}{0.24\textwidth}
			\centering
			\includegraphics{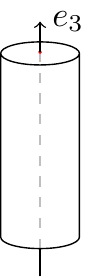}
			\caption{Unstable.}
			\label{fig:tube_rot_axis}
		\end{subfigure}
		\begin{subfigure}{0.24\textwidth}
			\centering
			\includegraphics{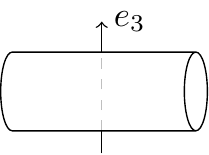}
			\caption{Globally attracting?}
			\label{fig:tube_rot_perp}
		\end{subfigure}
		\begin{subfigure}{0.24\textwidth}
			\centering
			\includegraphics{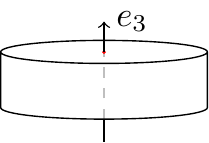}
			\caption{Globally attracting?}
			\label{fig:disk_rot_axis}
		\end{subfigure}
		\begin{subfigure}{0.24\textwidth}
			\centering
			\includegraphics{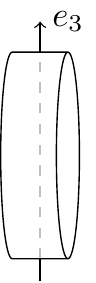}
			\caption{Unstable?}
			\label{fig:disk_rot_perp}
		\end{subfigure}
		\caption{
			Depictions of the microstructure for the equilibrium (A, C) and an energetic equilibrium (B, D)
			corresponding to both the oblong (A, B) and oblate cases (C, D).
			B and C are conjectured to be globally attracting for the oblong and oblate cases respectively,
			D is conjectured to be be unstable for the oblate case,
			and we prove in \fref{Theorem}{thm:main_result} that A is unstable.
		}
		\label{fig:conjecture}
	\end{figure}

\subsection{Heuristics for the long-time behaviour}
\label{sec:heuristics_long_time_behaviour}
	In this section we briefly discuss heuristics for the long-term behaviour of the system \eqref{eq:stat_prob_lin_mom}--\eqref{eq:stat_prob_microinertia}.
	The central element of the reasoning that follows is the energy-dissipation relation \eqref{eq:relative_energy_dissipation_relation}.
	As remarked in \fref{Section}{sec:equilibria}, this relation tells us that the relative energy $\mathcal{E}_\text{rel}$ defined in \eqref{eq:relative_energy}
	is non-increasing in time and bounded below. Let us therefore, for the sake of this discussion, assume that $\mathcal{E}_\text{rel}$ approaches its absolute minimum as time approaches $+\infty$.
	In particular this means that each term in $\mathcal{E}_\text{rel}$ approaches its absolute minimum, from which we deduce that $u$ approaches zero, $\omega$ approaches $\omega_{eq}$
	(since $J$ is strictly positive-definite at time $t=0$ and hence strictly positive-definite for all time), and $-J_{33}$ approaches $-\lambda_\text{max}$
	for $\lambda_\text{max}$ denoting the maximum eigenvalue of $J$, i.e. $\lambda_\text{max} = \max (\lambda, \nu)$.

	This last observation is precisely where the dichotomy between inertially oblong and inertially oblate microstructure comes in.
	If the microstructure is inertially oblong, i.e. $\lambda > \nu$, then $J_{33}$ approaches $\lambda$ which means that $\bar{J}$ must consist of the \emph{distinct} eigenvalues
	$\lambda$, $\nu$, and hence the global attractor is conjectured to be the orbit of energetic equilibria.
	If the microstructure is inertially oblate, i.e. $\nu > \lambda$, then $J_{33}$ approaches $\nu$ and hence $\bar{J}$ has repeated eigenvalues equal to $\lambda$,
	such that the global attractor is conjectured to be the equilibrium.

\subsection{Heuristics for the origin of the instability}
\label{sec:heuristics_instability}
	In this section we discuss heuristics for the origin of the instability of the system \eqref{eq:stat_prob_lin_mom}--\eqref{eq:stat_prob_microinertia}.
	Beyond being helpful heuristics that physically motivate the instability of the system, the ideas presented below actually form the core of our proof of the nonlinear instability.

	We begin with another energy-dissipation relation, which is associated with the linearization of the problem \eqref{eq:stat_prob_lin_mom}--\eqref{eq:stat_prob_microinertia} about its equilibrium.
	This relation is
	\begin{equation}
	\label{eq:energy_dissipation_of_the_linearization}
		\Dt \mathcal{E}_\text{lin}
		\defeq \Dt \int_{\T^3} \brac{
			\frac{1}{2} \abs{u}^2 + \frac{1}{2} J_{eq} \omega\cdot\omega - \frac{1}{2} \frac{1}{\lambda-\nu} {\brac{ \frac{\tau}{2\kappa} }}^2 \abs{a}^2
		}
		= - \mathcal{D} \brac{u, \omega - \omega_{eq}}
	\end{equation}
	where $a = \brac{J_{31}, J_{32}} = \brac{J_{13}, J_{23}}$ and where the dissipation $\mathcal{D}$ is given as in \eqref{eq:relative_energy_dissipation_relation} by
	\begin{equation*}
		\mathcal{D} \brac{u,\omega} = \int_{\T^3} 
			\frac{\mu}{2} \abs{\symgrad u}^2
			+ 2 \kappa {\vbrac{\half\nabla\times u - \omega}}^2
			+ \alpha\abs{\nabla\cdot\omega}^2
			+ \frac{\beta}{2} \abs{\symgrad^0 \omega}^2
			+ 2 \gamma \abs{\nabla\times\omega}^2.
	\end{equation*}
	Note that only part of the micro-inertia $J$ appears in \eqref{eq:energy_dissipation_of_the_linearization},
	namely $a = \brac{J_{31}, J_{32}}$ which corresponds to the products of inertia which describe the moment of inertia about the $e_1$-axis and $e_2$-axis,
	respectively, when the microstructure rotates about the $e_3$-axis.
	This is due to the fact that, as explained in detail in \fref{Section}{sec:block_structure}, the linearized problem can de decomposed into blocks which do not interact with one another.
	In particular the block governing the dynamics of $u$, $\omega$, and $a$ is the only block which produces non-trivial dynamics, and it is this block which gives rise to
	\eqref{eq:energy_dissipation_of_the_linearization}.

	Since the integrand of $\mathcal{E}_\text{lin}$ in \eqref{eq:energy_dissipation_of_the_linearization}, viewed as a quadratic form on $\brac{u, \omega, a}$,
	has negative directions precisely when the microstructure is inertially oblong, i.e. when $\lambda > \nu$,
	this suggests that the equilibrium is unstable in that case.

	We actually know a little bit more about the instability mechanism. If we denote by $M\brac{k}$, where $k\in\Z^3$, the symbol of the linearized operator about the equilibrium,
	then we can compute the spectrum of $M(0)$ explicitly and see that is has exactly two unstable eigenvalues, which come as a conjugate pair.
	An important point to note here is that the only nonzero components of the eigenvectors corresponding to this conjugate pair are the components corresponding to
	$a$ and $\bar{\omega}$, which denotes the horizontal components of $\omega$, i.e. $\bar{\omega} = \brac{\omega_1, \omega_2}$.
	It is thus precisely $a$ and $\bar{\omega}$ that are at the origin of the instability.

	This is particularly interesting since $M(0)$ is precisely (up to neglecting its components depending on $u$) the linearization of the ODE
	\begin{equation*}
		\left\{
		\begin{aligned}
			&J\, \frac{d\omega}{dt} + \omega\times J\omega = \tau e_3 - 2\kappa\omega\\
			&\frac{dJ}{dt} = \sbrac{\Omega, J}
		\end{aligned}
		\right.
	\end{equation*}
	about its equilibrium $\brac{\omega_{eq}, J_{eq}} = \brac{\frac{\tau}{2\kappa}e_3, \diag\brac{\lambda, \lambda, \nu}}$,
	where here $\omega$ and $J$ are only time-dependent.
	This ODE describes the rotation of a damped rigid body subject to a uniform torque, which tells us that instability of the system \eqref{eq:stat_prob_lin_mom}--\eqref{eq:stat_prob_microinertia}
	stems precisely from the instability of this ODE.

	Finally note that, although this ODE plays a key role in explaining the instability mechanism, it does not fully characterize it.
	To understand what we mean by this, recall that the linearization of the ODE about its equilibrium describes the evolution of the zero Fourier mode of the linearized PDE.
	However, the nonzero Fourier modes play a nontrivial role in the instability mechanism.
	Indeed numerics show that, depending on the physical regime, the most unstable mode (i.e. that giving rise to the eigenvalue with the largest positive real part)
	may or may not be the zero mode. This is shown in \fref{Figure}{fig:insta}.
	
	\begin{figure}
		\centering
		\begin{subfigure}{0.45\textwidth}
			\centering
			\includegraphics{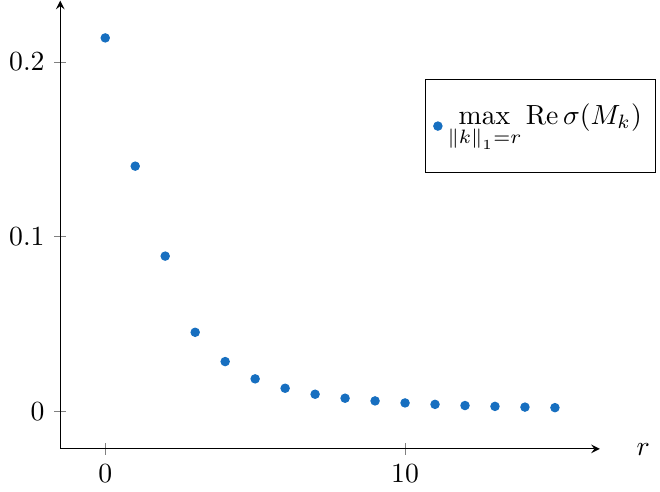}
			\caption{Physical parameters: $\lambda = 3.2$, $\nu = 0.6$, $\mu = 4.3$, $\kappa = 3.3$, $\alpha = 0.9$, $\beta = 6.8$, $\gamma = 0.4$, $\tau = 4.4$.}
			\label{fig:instaZero}
		\end{subfigure}
		\begin{subfigure}[h!]{0.45\textwidth}
			\centering
			\includegraphics{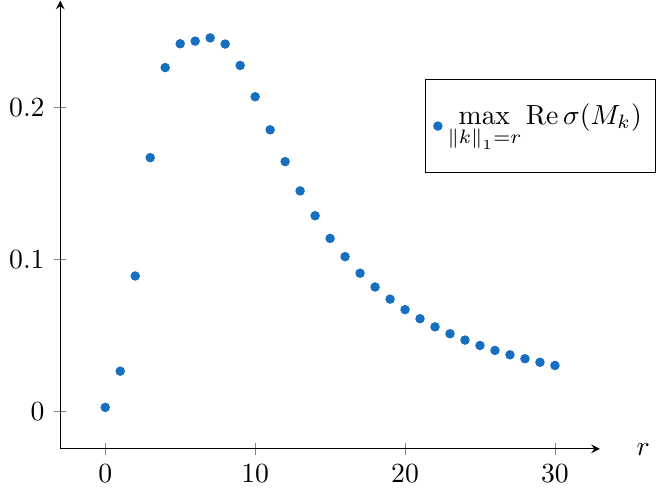}
			\caption{Physical parameters: $\lambda = 3.6$, $\nu = 1.2$, $\mu = 2.4$, $\kappa = 0.4$, $\alpha = 5.3$, $\beta = 3.1$, $\gamma = 1.7$, $\tau = 20$.}
			\label{fig:instaNonZero}
		\end{subfigure}
		\caption{
			An illustration of the fact that the instability is not exclusively due to the zero mode:
			depending on the physical parameter regime the eigenvalue with largest real part may or may not occur when $k=0$.
			Here $M_k$ denotes the symbol of the linearization of \eqref{eq:stat_prob_lin_mom}--\eqref{eq:stat_prob_microinertia} about the equilibrium.
		}
		\label{fig:insta}
	\end{figure}

\subsection{Summary of techniques and plan of paper}
\label{sec:difficulties}

	Our technique for proving \fref{Theorem}{thm:main_result} is to employ the nonlinear bootstrap instability framework first introduced by Guo-Strauss \cite{guo_strauss_bgk},
	which is not so much a black-box theorem as it is a strategy for proving instability.
	In broad strokes, the idea is to construct a maximally unstable solution to the linearized equations and then employ a nonlinear energy method to prove that this solution is nonlinearly stable,
	i.e. the nonlinear dynamics stay close to the linear growing mode, which then leads to instability.

	An essential feature of the Guo-Strauss bootstrap instability framework is that it does \emph{not} require the presence of a spectral gap,
	as is required for other standard methods used to prove nonlinear instability (see for example \cite{friedlander_strauss_vishik}).
	This is crucial for us since it is quite delicate to obtain spectral information about the problem at hand, as discussed in more detail below.
	In particular, note that \fref{Proposition}{prop:trajectories_evals} tells us that
	a pair of conjugate eigenvalues of the linearized operator approach the imaginary axis as the wavenumber approaches infinity.
	As an immediate consequence, we may thus deduce that there is no spectral gap.
	
	In order to implement the bootstrap instability strategy we need four ingredients.  The first is the maximally unstable linear growing mode.  This is a solution to the linearized equations (linearized around the equilibrium) that grows exponentially in time (when measured in various Sobolev norms) at a rate that is maximal in the sense that no other solution to the linearized equations grows more rapidly.  The second is a scheme of nonlinear energy estimates that allows us to obtain control of high-regularity norms of solutions to the nonlinear problems in terms of certain low-regularity norms.  This is the bootstrap portion of the argument.  The third is a low-regularity estimate of the nonlinearity in terms of the square of the high-regularity energy, valid at least in a \emph{small energy regime}.  Finally, we need a local existence theory for the nonlinear problem that is capable of producing solutions to which the bootstrap estimates apply.  With these ingredients in hand, we can then prove that the nonlinear solution stays sufficiently close to the growing linearized solution that it must leave a ball of fixed radius within a timescale computed in terms of the data.
	
	In Section \ref{sec:ana_lin_prob} we construct the maximally unstable solution to the linearized equations.  A principal difficulty is encountered immediately upon linearizing: the resulting (spatial) differential operator is not self-adjoint.  This is due entirely to the anisotropy of the microstructure, and in particular to the term $\omega \times J \omega$ in \eqref{eq:stat_prob_ang_mom}; indeed, in the case of isotropic microstructure this term vanishes and the linearized operator becomes self-adjoint.  The lack of self-adjointness means we have far fewer tools at our disposal, and in particular it means that we cannot employ variational methods to find the maximal growing mode.
	
	Since we work on the torus and the linearization is a constant coefficient problem, we are naturally led to seek the maximal solution in the form of a growing Fourier mode solution.  This leads to an ODE in $\C^8$ of the form $\partial_t \hat{X}_k = \hat{\B}_k \hat{X}_k$, where $k \in \Z^3$ is the wavenumber and $\hat{\B}_k \in \C^{8\times 8}$ is not Hermitian.  Without the precision tools associated to Hermitian matrices, we are forced to naively study the degree eight characteristic polynomial of $\hat{\B}_k$, which, due to the appearance of the physical parameters $\alpha,$ $\beta,$ $\gamma,$ $\kappa$, $\mu$, $\tau$, $\lambda,$ $\nu,$ in addition to the wave number $k$, is an unmitigated mess. 	Numerics (see Figure \ref{fig:insta}) suggest that for any $k \in \Z^3$ the spectrum consists of a conjugate pair of unstable eigenvalues, a zero eigenvalue (coming from the incompressibility condition), and five stable eigenvalues.  However, due to the inherent complexity of $\hat{\B}_k$ and its characteristic polynomial, we were unable to prove this, except in the case $k=0$.
	
	Failing at the direct approach of simply factoring the characteristic polynomial of $\hat{\B}_k$, we instead employ an indirect approach based on isolating the highest order (in terms of the wavenumber $k$) part of the characteristic polynomial and deriving its asymptotic form as $\abs{k} \to \infty$.  For this it's convenient to parameterize the matrices in terms of $k \in \R^3$ rather than $\Z^3$.  Using this idea, the special form of the highest-order term, and the implicit function theorem, we are then able to prove the existence of an unstable conjugate pair of eigenvalues, smoothly parameterized by $k \in \R^3$ in a neighborhood of infinity.  Remarkably, since the neighborhood of infinity contains all but finitely many lattice points from $\Z^3$, we conclude from this argument that for \emph{all but finitely many} wavenumbers $\hat{\B}_k$ is unstable.  Combining this with a number of delicate spectral estimates and an application of Rouch\'{e}'s theorem, we are then able to find $k_\ast \in \Z^3$ with the largest growth rate.  From this and a Fourier synthesis we then construct the desired maximal growing mode.  
	
	The lack of self-adjointness is also an issue when we seek to use spectral information about $\hat{\B}_k$ to obtain bounds on the corresponding matrix exponential $e^{t \hat{\B}_k}$.  These bounds are required to obtain the bounds on the semigroup generated by the linearization that verify that our growing mode is actually maximal among all linear solutions.  We only know that $e^{t \hat{\B}_k}$ is similar to its diagonal matrix up to a change of basis matrix whose norm
	\emph{depends on $k$}. Circumventing this issue requires a good understanding of the decay of the spectrum of the symmetric part of $\hat{\B}_k$ as $k$ becomes large, and the precise workaround is discussed at the beginning of the proof of \fref{Proposition}{prop:unif_bound_mat_exp}.

	In \fref{Section}{sec:nonlinear_estimates} we derive the nonlinear bootstrap energy estimates and the nonlinearity estimate.  Here the primary difficulty is related to rewriting the problem in a way that prevents time derivatives from entering the nonlinearity.  If we were to naively rewrite \eqref{eq:stat_prob_ang_mom} by writing $J\pdt\omega = J_{eq}\pdt\omega + \brac{J-J_{eq}}\pdt\omega$ and considering the term $\brac{J-J_{eq}}\pdt\omega$ as a remainder term, then we would then not be able to close the estimates due to this time derivative being present as part of the nonlinear remainder. Instead we must multiply \eqref{eq:stat_prob_ang_mom} by $J_{eq}J^{-1}$, which solves the time derivative problem but significantly worsens the form of the remaining terms in the nonlinearity.  In spite of this, we are able to derive the appropriate estimates needed for the bootstrap argument.  
	
	We delay the development of the final ingredient, the local existence theory, until \fref{Appendix}{sec:lwp}.  Our local existence theory is built on a nonlinear Galerkin scheme that employs the Fourier basis for the finite dimensional approximations.  To solve the resulting nonlinear, but finite dimensional, ODE we borrow many of the nonlinear estimates from \fref{Section}{sec:nonlinear_estimates}.
	
	\fref{Section}{sec:bootstrap} combines the four ingredients to prove our instability result.  This culminates in \fref{Theorem}{thm:bootstrap_instability}, the main result of the paper.  Finally, in \fref{Appendix}{sec:auxiliary} we record a number of auxiliary results that are used throughout the main body of the paper.
	
\subsection{Notation}
\label{sec:notation}

	We say a constant $C$ is universal if it only depends on the various parameters of the problem, the dimension, etc., but not on the solution or the data.
	The notation $\alpha \lesssim \beta$ will be used to mean that there exists a universal constant $C>0$ such that $\alpha \leqslant C \beta$. 

	Let us also record here some basic notation for linear algebraic operations.
	For any $w\in\R^n$ we denote by $P_\parallel\brac{w}$ and $P_\perp\brac{w}$ the orthogonal projections onto the span of $w$ and its orthogonal complement, respectively.
	More precisely: for any nonzero $w$, $P_\parallel\brac{w} = \frac{w\otimes w}{\abs{w}^2}$ and $P_\perp\brac{w} = I - \frac{w\otimes w}{\abs{w}^2}$,
	whilst $P_\parallel \brac{0} = 0$ and $P_\perp \brac{0} = I$.
	For any $v\in\R^2$ and $w\in\R^3$ we write $\bar{w} = \brac{w_1,w_2}$, $\bar{w}^\perp = \brac{-w_2, w_1}$, $\tilde{v} = \brac{v_1, v_2, 0}$, and $\tilde{v}^\perp = \brac{-v_2, v_1, 0}$.
	Finally, let $X_1$, $X_2$, $Y_1$, and $Y_2$ be normed vector spaces, let $L_1 \in \Leb\brac{X_1,\,Y_1}$, and let $L_2\in\Leb\brac{X_2,\,Y_2}$.
	The \emph{direct sum} of $L_1$ and $L_2$, denoted $L_1 \oplus L_2$, is the bounded linear operator from $X_1 \times X_2$ to $Y_1 \times Y_2$ defined via,
	for every $\brac{f_1,f_2}\in X_1 \times X_2$, $\brac{L_1 \oplus L_2}\brac{f_1,f_2} \defeq \brac{L_1 f_1, L_2 f_2}$.


\section{Analysis of the linearization}
\label{sec:ana_lin_prob}

	To begin we record the precise form of the linearization of \eqref{eq:stat_prob_lin_mom}--\eqref{eq:stat_prob_microinertia} about the equilibrium solution
	$\brac{u_{eq}, p_{eq}, \omega_{eq}, J_{eq}} = (0, 0, \frac{\tau}{2\kappa} e_3, \diag (\lambda, \lambda, \nu))$
	and introduce notation which allows us to write the linearized problem in a compact form.
	Then in \fref{Section}{sec:block_structure} we note that the linearized operator has a natural block structure with only \emph{one} block which gives rise to non-trivial dynamics.
	It is this component whose spectrum we study in detail in \fref{Section}{sec:spec_ana}.
	The results from \fref{Section}{sec:spec_ana} are then used to construct the semigroup associated with the linearization in \fref{Section}{sec:semigroup}
	and to construct a maximally unstable solution to the linearized problem in \fref{Section}{sec:max_unstable_sol}.

	The linearization is 
	\begin{subnumcases}{}
		\pdt u = \brac{\mu + \kappa/2} \Delta u + \kappa \nabla\times\omega - \nabla p,
		\label{eq:lin_prob_lin_mom}\\
		J_{eq} \pdt \omega
			=   -\brac{\omega\times J_{eq}\omega_{eq} + \omega_{eq}\times J\omega_{eq} + \omega_{eq} \times J_{eq} \omega}
		\nonumber\\\hspace{1.65cm}
			+ \kappa \nabla\times u - 2\kappa\omega + \brac{\alpha + \beta/3 - \gamma}\nabla\nabla\cdot\omega + \brac{\beta + \gamma}\Delta\omega, \text{ and }
		\label{eq:lin_prob_ang_mom}\\
		\pdt J = \sbrac{\Omega_{eq}, J} + \sbrac{\Omega, J_{eq}}
		\label{eq:lin_prob_microinertia}
	\end{subnumcases}
	subject to $\nabla\cdot u = 0$
	which, for $X = \brac{u,\omega,J}$, $D = I_3 \oplus J_{eq} \oplus \ISymThreeByThree$ (where $\ISymThreeByThree$ denotes the identity function on the space of 3-by-3 matrices),
	$\Lambda\brac{p} = \brac{-\nabla p, 0, 0}$,
	and an appropriate linear operator $\widetilde{\Leb}$
	can be written more succintly as
	\begin{equation}
	\label{eq:PDE_compact_not_lin}
		\pdt DX = \widetilde{\Leb} X + \Lambda\brac{p} \text{ subject to } \nabla\cdot u=0.
	\end{equation}

\subsection{The block structure}
\label{sec:block_structure}

	The linearization \eqref{eq:lin_prob_lin_mom}--\eqref{eq:lin_prob_microinertia} can be decomposed into blocks which do not interact with one another.
	Notably, only one of these blocks gives rise to non-trivial dynamics, so we will identify this block before studying its spectrum in \fref{Section}{sec:spec_ana}.
	More precisely: writing
	$$
		J = \begin{pmatrix} \JO & a \\ a^T & J_{33} \end{pmatrix},
	$$
	the linearization becomes
	\begin{subnumcases}{}
		\pdt u = \brac{\mu + \kappa/2} \Delta u + \kappa \nabla\times\omega - \nabla p
		\label{eq:lin_prob_block_lin_mom},\\
		J_{eq} \pdt \omega = \kappa\nabla\times u - 2\kappa\omega + \brac{\alpha + \beta/3 - \gamma}\nabla\nabla\cdot\omega + \brac{\beta + \gamma}\Delta\omega
			-\brac{\lambda-\nu}\frac{\tor}{2\kappa}\tilde{\omega}^\perp - {\brac{\frac{\tor}{2\kappa}}}^2 \tilde{a}^\perp
		\label{eq:lin_prob_block_ang_mom},\\
		\pdt a = \brac{\lambda - \nu} \bar{\omega}^\perp + \frac{\tor}{2\kappa} a^\perp
		\label{eq:lin_prob_block_microinertia_a}, \\
		\pdt \JO = \frac{\tau}{2\kappa} \sbrac{R, \JO}
		\label{eq:lin_prob_block_microinertia_barJ}, \text{ and }\\
		\pdt J_{33} = 0
		\label{eq:lin_prob_block_microinertia_J33}
	\end{subnumcases}
	subject to $\nabla\cdot u = 0$, where $R$ is the 2-by-2 matrix given by $R = e_2\otimes e_1 - e_1\otimes e_2$.
	In particular, if we write $Y = \brac{u,\omega,a}$ and $\bar{D} = I_3 \oplus J_{eq} \oplus I_2$
	then \eqref{eq:lin_prob_block_lin_mom}, \eqref{eq:lin_prob_block_ang_mom}, and \eqref{eq:lin_prob_block_microinertia_a} can be written as $\pdt \bar{D} Y = \widetilde{\M} Y +\Lambda\brac{p}$
	subject to $\nabla\cdot u = 0$ for an appropriate operator $\widetilde{\M}$.
	In particular, since $\widetilde{\M}$ commutes with the application of the Leray projector to $u$
	it suffices to study $\pdt \bar{D} Y = \widetilde{\M} \,\bar{\mathbb{P}}\, Y$, where $\bar{\mathbb{P}} \defeq \mathbb{P}_L \oplus I_3 \oplus I_2$ for $\mathbb{P}_L$ denoting the Leray projector.
	Recall that the Leray projector is the projection onto divergence-free vector fields, which on the 3-torus can be written explicitly as $\mathbb{P}_L = - \nabla\times \Delta\inv\, \nabla\times$
	(see \fref{Lemma}{lemma:formula_complement_Leray_projector}).

	So finally, for $\B \defeq \bar{D}\inv \widetilde{\M} \,\bar{\mathbb{P}}$ we have that
	$\Leb \defeq D\inv \widetilde{\Leb} \,\mathbb{P}$, where $\mathbb{P} \defeq \mathbb{P}_L \oplus I_3 \oplus \ISymThreeByThree$,
	can be written as $\Leb = \B\oplus\frac{\tau}{2\kappa}\sbrac{R,\,\cdot\,}\oplus 0$.
	Note that using this notation we may write the linearized problem \eqref{eq:PDE_compact_not_lin}, after Leray projection, as
	\begin{equation}
	\label{eq:PDE_compact_not_lin_semigroup}
		\pdt X = \Leb X.
	\end{equation}
	This is a particularly convenient formulation since it is amenable to attack via semigroup theory.

	What matters for the purpose of the spectral analysis carried out in the following section is that
	the equations governing the non-trivial dynamics of the problem can be written as $\pdt Y = \B Y$.
	The punchline is that it suffices to study the spectrum of $\B$, which is precisely what we do in \fref{Section}{sec:spec_ana} below.

\subsection{Spectral analysis}
\label{sec:spec_ana}

	In this subsection we study the spectrum of the operator $\B$ introduced in the preceding section.
	Since our domain is the torus it is natural to consider the symbol $\hat{\B}$ of this operator, which gives a matrix in $\mathbb{C}^{8 \times 8}$ for each wavenumber $k \in \mathbb{Z}^3$.  However, it will be more convenient for us to parameterize these with a continuous wavenumber $k \in \R^3$; for each such $k$ we define $\hat{\B}_k \in \mathbb{C}^{8 \times 8}$ 	according to   
	\begin{multline}
	\label{eq:def_B}
	\hat{\B}_k := \\
		\begin{pmatrix}
				-\brac{\mu + \frac{\kappa}{2}} \abs{k}^2 P_\perp\brac{k}
			&
				i\kappa k\times
			&
				0
			\\
				J_{eq}\inv \brac{i\kappa k\times} P_\perp\brac{k}
			&
				-2\kappa J_{eq}\inv - \tilde{\alpha} \abs{k}^2 J_{eq}\inv P_\parallel \brac{k} - \tilde{\gamma} \abs{k}^2 J_{eq}\inv P_\perp \brac{k}
				- \brac{1 - \frac{\nu}{\lambda}} \frac{\tau}{2\kappa} R_{33}
			&
				- \frac{1}{\lambda} {\brac{\frac{\tau}{2\kappa}}}^2 R_{32}
			\\
				0
			&
				\brac{\lambda - \nu} R_{23}
			&
				\frac{\tau}{2\kappa} R_{22}
		\end{pmatrix},
	\end{multline}
	where $P_\parallel$ and  $P_\perp$ are as defined in Section \ref{sec:notation}, and 
	\begin{equation*}
		R_{22} = R = \begin{pmatrix}
			0 & -1\\
			1 & 0
		\end{pmatrix},\,
		R_{23} = \begin{pmatrix}
			0 & -1 & 0\\
			1 &  0 & 0
		\end{pmatrix},\,
		R_{32} = \begin{pmatrix}
			0 & -1\\
			1 &  0\\
			0 &  0
		\end{pmatrix},\text{ and }
		R_{33} = \begin{pmatrix}
			0 & -1 & 0\\
			1 &  0 & 0\\
			0 &  0 & 0
		\end{pmatrix}.
	\end{equation*}
	Note here that we have abused notation by writing $i \kappa k \times$ as a place-holder to indicate the matrix corresponding to the linear map $z \mapsto i \kappa k \times z$.  
	
	It is somewhat tricky to extract useful spectral information from $\hat{\B}_k$ directly.
	Instead, we introduce a sort of similarity transformation $M_k \defeq Q_k \hat{\B}_k \bar{Q}_k$
	in such a way that $M_k$ is a real matrix, i.e. $M_k \in \R^{8 \times 8}$ for each $k\in\R^3$, which carries the spectral information of $\hat{\B}_k$.
	Here the matrices $Q_k,\bar{Q}_k \in \mathbb{C}^{8 \times 8}$ are defined by 
	\begin{equation*}
	 Q_k \defeq T\brac{k} \oplus J_{eq}^{1/2} \oplus sR_{22} \text{ and }\bar{Q}_k \defeq T\brac{k} \oplus J_{eq}^{-1/2} \oplus \brac{-s\inv}R_{22},
	\end{equation*}
 	where $T\brac{k} \defeq \frac{ik\times}{\abs{k}}$ if $k\neq 0$, $T\brac{0} \defeq 0$, and $s\defeq \frac{-1}{\sqrt{\lambda-\nu}} \frac{\tor}{2\kappa}$.
	Unfortunately,  $Q_k$ and $\bar{Q}_k$ are not quite invertible, so this isn't exactly a similarity transformation.
	When $k \neq 0$, this is due to the fact that $(k,0,0)$ belongs to the kernels of both operators, a fact that is ultimately related to the divergence-free condition for $u$, which reads $k\cdot \hat{u}_k =0$ on the Fourier side.
	In principle we could remove the kernel and restore invertibility, but the resulting 7-by-7 matrices are less convenient to work with.
	As such, we will stick with the 8-by-8 setup and find a work-around for the invertibility issue.
	Ultimately we will prove in \fref{Propositions}{prop:max_unstable_evals} and \ref{prop:unif_bound_mat_exp} that we can gain good spectral information about $M_k$,
	and it will follow from \fref{Definition}{def:lin_map_act_on_quot_space} and \fref{Lemmas}{lemma:B_Q_act_on_V_k} and \ref{lemma:sim_mat_act_on_quot_spaces}
	that the spectrum of $\hat{\B}_k$ coincides with that of $M_k$.
	Note that for all these $k$-dependent matrices we will write equivalently $M_k$ or $M\brac{k}$.

	An important observation is that the matrix $M_k \in \R^{8 \times 8}$ may be decomposed into its symmetric part $S_k \in \R^{8\times 8}$ and its antisymmetric part $A\in \R^{8\times 8}$ such that $A$ is \emph{independent of $k$}.   More precisely
	\begin{equation}
	\label{eq:def_S}
		S_k = \begin{pmatrix}
				-\brac{\mu+\frac{\kappa}{2}}\abs{k}^2 P_\perp \brac{k}
			&
				\kappa \abs{k} P_\perp \brac{k} J_{eq}^{-1/2}
			&
				0
			\\
				\kappa \abs{k} J_{eq}^{-1/2} P_\perp \brac{k}
			&
				-2\kappa J_{eq}\inv - \tilde{\alpha} \abs{k}^2 J_{eq}^{-1/2} P_\parallel \brac{k} J_{eq}^{-1/2} - \tilde{\gamma} \abs{k}^2 J_{eq}^{-1/2} P_\perp \brac{k} J_{eq}^{-1/2}
			&
				\phi I_{32}
			\\
				0
			&
				\phi I_{23}
			&
				0
		\end{pmatrix}
	\end{equation}
	and
	\begin{equation}
	\label{eq:def_A}
		A = 0 \oplus c R_{33} \oplus dR_{22},
	\end{equation}
	where 
	\begin{equation}\label{phi_c_d_def}
		\phi = \sqrt{1 - \frac{\nu}{\lambda}} \, \frac{\tor}{2\kappa}, \; c = \brac{\frac{\nu}{\lambda} - 1}\frac{\tor}{2\kappa}, \text{ and }d = \frac{\tor}{2\kappa}.	 
	\end{equation}
	Note that $M_k$ is written out explicitly in all its gory details in \fref{Appendix}{sec:M_k}.

	We now turn to the issue of proving that the spectra of $\hat{\B}_k$ and $M_k$ coincide.
	To do this we will need to use the notion of linear maps acting on quotient spaces.
	Here we quotient out by the spaces $V_k$ defined as $V_0 \defeq \Span\setdef{\brac{v,0,0}}{v\in\R^3}$ as well as, for any nonzero $k\in\R^3$, $V_k \defeq \Span\brac{k,0,0}$.

	\begin{definition}[Linear maps acting on quotient spaces]
	\label{def:lin_map_act_on_quot_space}
		Let $A \in \C^{n \times n}$ and let $V$ be a subspace of $\C^n$. We say that \emph{$A$ acts on $\C^n /\, V$} if and only if $\ker A = V$ and $\im A \subseteq V^\perp$, where $V^\perp$ is the orthogonal complement relative to the standard Hermitian structure on $\C^n$.
	\end{definition}

	We refer to \fref{Lemma}{lemma:sim_mat_act_on_quot_spaces} for the key property of linear maps acting on quotient spaces which we will use in the sequel, namely conditions under which two
	matrix representations of such maps are equivalent, even when the `change of basis' matrices involved are not invertible.
	We now prove that the matrices we are dealing with here do satisfy the hypotheses of \fref{Lemma}{lemma:sim_mat_act_on_quot_spaces}.

	\begin{lemma}
	\label{lemma:B_Q_act_on_V_k}
		For any $k\in\R^3$, $\hat{\B}_k$, $Q_k$, and $\bar{Q}_k$ act on $\C^8 /\, V_k$ and $Q_k \bar{Q}_k = \bar{Q}_k Q_k = \proj_{V_k^\perp}$.
	\end{lemma}

	\begin{proof}
		First we consider $\hat{\B}_k$ for $k\neq 0$.
		Since $\hat{\B}_k^\dagger \brac{k,0,0} = \hat{\B}_k \brac{k,0,0} = 0$, where $\dagger$ denotes the conjugate transpose,
		we know that $\im\hat{\B}_k \subseteq V_k$ and that $V_k \subseteq \ker\hat{\B}_k$,
		so we only have to show that $\ker\hat{\B}_k \subseteq V_k$.
		Let $y=\brac{v,\theta,b}\in\ker\hat{\B}_k$.
		The third row of \eqref{eq:def_B} tells us that $b = \frac{2\kappa\brac{\lambda-\nu}}{\tor} \bar{\theta}$ and hence
		$$
			0 = \bar{D} \hat{\B}_k y \cdot y
			= -\mu\abs{k}^2 {\vbrac{v_\perp}}^2
			- 2\kappa {\vbrac{ \frac{1}{2} ik\times v - \theta }}^2
			- \tilde{\alpha} \abs{k}^2 {\vbrac{\theta_\parallel}}^2
			- \tilde{\gamma} \abs{k}^2 {\vbrac{\theta_\perp}}^2.
		$$
		Therefore $\theta = v_\perp = 0$, and hence also $b=0$, such that indeed $y = \brac{v_\parallel,0,0}\in V_k$.
		So indeed $\hat{\B}_k$ acts on $\C^8/\,V_k$.

		Now we consider $\hat{\B}_0$, proceeding essentially as we did above for the case $k\neq 0$.
		Since $\hat{\B}_0^\dagger \brac{v,0,0} = \hat{\B}_0 \brac{v,0,0} = 0$ for any $v\in\R^3$ it follows that $\im\hat{B}_0 \subseteq V_0$ and that $V_0 \subseteq \ker \hat{\B}_0$.
		Now let $y = \brac{v,\theta,b}\in\ker\hat{\B}_0$ and observe that, as above, $b=\frac{2\kappa\brac{\lambda-\nu}}{\tor} \bar{\theta}$
		and that hence $0 = \bar{D}\hat{\B}_0 y\cdot y = -2\kappa\abs{\theta}^2$.
		Therefore $\theta=0$ and $b=0$ such that indeed $y=\brac{v,0,0}\in V_0$.
		So $\ker\hat{\B}_0 \subseteq V_0$ and thus indeed $\hat{\B}_0$ acts on $\C^8/\,V_0$.

		We now turn our attention to $Q_k$ and $\bar{Q}_k$.
		Since ${\brac{k,0,0}}^T Q_k = {\brac{k,0,0}}^T \bar{Q}_k = \brac{k\cdot T\brac{k}} \oplus 0 \oplus 0 = 0$ for any nonzero $k\in\R^3$
		and since $\brac{v,0,0}\cdot Q_0 = \brac{v,0,0}\cdot \bar{Q}_0 = v\cdot T\brac{0} \oplus 0 \oplus 0 = 0$, we may deduce that $\im Q_k,\,\im \bar{Q}_k \subseteq V_k^\perp$
		for \emph{all} $k\in\Z^3$.
		Now observe that, since $J_{eq}^{1/2}$ and $R_{22}$ are invertible, we deduce that $\ker Q_k = \ker \bar{Q}_k = \brac{\ker T\brac{k}} \oplus 0 \oplus 0$.
		Therefore, since $\ker T\brac{k} = \Span\cbrac{k}$ when $k$ is nonzero and since $\ker T\brac{0} = \R^3$,
		we have that indeed $\ker Q_k = \ker \bar{Q}_k = V_k$ for \emph{all} $k\in\Z^3$, i.e. $Q_k$ and $\bar{Q}_k$ act on $\C^8/\,V_k$ for all $k\in\Z^3$.

		Finally observe that, since $R_{22}^2 = -I_2$, it follows that $Q_k\bar{Q}_k = \bar{Q}_k Q_k = {T\brac{k}}^2 \oplus I_3 \oplus I_2$,
		where ${T\brac{0}}^2 = 0$ and ${T\brac{k}}^2 = \frac{{\brac{ik\times}}^2}{\abs{k}^2} = \proj_{{\Span \cbrac{k}}^\perp}$ for $k\neq 0$.
		Note that we have used the $\varepsilon$-$\delta$ identity $\varepsilon_{aij}\varepsilon_{akl} = \delta_{ik}\delta_{jl}- \delta_{il}\delta_{jk}$
		to deduce that ${\brac{k\times}}^2 = -\abs{k}^2 \proj_{{\Span\cbrac{k}}^\perp}$.
		So indeed $Q_k \bar{Q}_k = \bar{Q}_k Q_k = \proj_{V_k^\perp}$.
	\end{proof}

	We now record how $M_k$ behaves under transformations of the form $k \mapsto -k$ and $k=(\bar{k}, k_3) \mapsto (\bar{H}\bar{k}, k_3)$ for $\bar{H}$ an orthogonal map.
	This comes in handy when constructing the maximally unstable solution in \fref{Section}{sec:max_unstable_sol}.

	\begin{lemma}[Equivariance and invariance of $M$]
	\label{lemma:equiv_inv_M}
	Let $H$ be a \emph{horizontal rotation}, i.e. $H\in\R^{3\times 3}$ such that $H = \bar{H} \oplus 1$ for some 2-by-2 orthogonal matrix $\bar{H}$.
		We call $\widetilde{H} \defeq H \oplus H \oplus \bar{H}$ the \emph{joint horizontal rotation} associated with $H$.
		\begin{enumerate}
			\item	$M$ is equivariant under horizontal rotations, i.e. for any $k\in\R^3$ and any horizontal rotation $H$, $M\brac{Hk} = \widetilde{H} M\brac{k} \widetilde{H}^T$ and
			\item	$M$ is even, i.e. for any $k\in\R^3$, $M\brac{-k} = M\brac{k}$.
		\end{enumerate}
	\end{lemma}

	\begin{proof}
		Note that $k\mapsto P_\parallel\brac{k}, P_\perp\brac{k}$ are both even and equivariant under horizontal rotations,
		i.e., for any horizontal rotation $H$, $P_\parallel\brac{Hk} = HP_\parallel\brac{k} H^T$ and similarly for $P_\perp$,
		whilst $k\mapsto\abs{k}$ is even and invariant under horizontal rotations.
		We can therefore write
		\begin{equation*}
			S\brac{k} = \begin{pmatrix}
				A\brac{k}	& B\brac{k}				& 0\\
				C\brac{k}	& -2\kappa J_{eq}\inv + D\brac{k}	& \phi I_{32}\\
				0		& \phi I_{23}				& 0
			\end{pmatrix}
		\end{equation*}
		for some $A,B,C,D$ which are equivariant under horizontal rotations and even.
		It follows immediately that $M$ is even.
		Now let $H$ be a horizontal rotation. Since $\bar{H} I_{23} H = I_{23}$, $H I_{32} \bar{H} = I_{32}$, and since $H$ commutes with $J_{eq}\inv$ one may readily compute that
		$S\brac{Hk} = \widetilde{H} S\brac{k} \widetilde{H}^T$.
		Finally, since two-dimensional rotations (i.e. elements of $O\brac{2}$) commute with one another, $A = \widetilde{H}A\widetilde{H}^T$
		and so indeed $M$ is equivariant under horizontal rotations.
	\end{proof}

	We now obtain some fairly crude bounds on the spectrum of $M_k$ in \fref{Lemmas}{lemma:bounds_spectrum_S}, \ref{lemma:bounds_real_part_eval_M}, and \ref{lemma:bounds_im_part_eval_M}.
	These bounds are nonetheless essential in the proofs of \fref{Propositions}{prop:max_unstable_evals} and \ref{prop:unif_bound_mat_exp}.
	As a first step in obtaining these bounds we identify the quadratic form associated with $S_k$, the symmetric part of $M_k$, in \fref{Lemma}{lemma:quad_form_S}.

	\begin{lemma}[Quadratic form associated with $S_k$]
	\label{lemma:quad_form_S}
		For any $y = \brac{v,\theta,b}\in\R^3\times\R^3\times\R^2 = \R^8$ and any $k\in\R^3$,
		\begin{equation*}
			S\brac{k} y \cdot y
			= -\mu \abs{k}^2 \abs{v_\perp}^2
			- 2\kappa {\vbrac{
				\half\abs{k}v_\perp - J_{eq}^{-1/2} \theta
			}}^2
			- \tilde{\alpha} \abs{k}^2 {\vbrac{
				{\brac{J_{eq}^{-1/2} \theta}}_\parallel
			}}^2
			- \tilde{\gamma} \abs{k}^2 {\vbrac{
				{\brac{J_{eq}^{-1/2} \theta}}_\perp
			}}^2
			+ 2 \phi \bar{\theta} \cdot b.
		\end{equation*}
		where, for any $w\in\R^3$, $w_\parallel \defeq \proj_k w$ and $w_\perp \defeq \brac{I-\proj_k}w $, and $\phi$ is as in \eqref{phi_c_d_def}.
	\end{lemma}

	\begin{proof}
		This follows immediately from the definition of $S$ in \eqref{eq:def_S}.
	\end{proof}

	\noindent
	We now use \fref{Lemma}{lemma:quad_form_S} to obtain upper bounds on the eigenvalues of $S$.

	\begin{lemma}[Spectral bounds on $S_k$]
	\label{lemma:bounds_spectrum_S}
		For any $k\in\R^3$, it holds that $\max\sigma\brac{S_k} \leqslant \min\brac{\phi,\,\frac{C_\sigma}{\abs{k}^2}}$,
		where $C_\sigma \defeq \frac{\phi^2\lambda}{\min\brac{\tilde{\alpha}, \tilde{\gamma}}}$ and $\phi$ is as in \eqref{phi_c_d_def}.
	\end{lemma}

	\begin{proof}
		Let $k\in\R^3$ and let $y=\brac{v,\theta,b}\in\R^3\times\R^3\times\R^2$. By \fref{Lemma}{lemma:quad_form_S}
		\begin{equation}
		\label{eq:spec_bounds_S_intermediate_upper_bound}
			S\brac{k}y\cdot y
			\leqslant - \tilde{\alpha}\abs{k}^2 {\vbrac{ {\brac{J_{eq}^{-1/2} \theta }}_\parallel }}^2
			- \tilde{\gamma}\abs{k}^2 {\vbrac{ {\brac{J_{eq}^{-1/2} \theta }}_\perp }}^2
			+ 2\phi\bar{\theta}\cdot b
		\end{equation}
		from which it follows that $S\brac{k}y\cdot y \leqslant \phi\brac{\abs{\bar{\theta}}^2 + \abs{b}^2}$ and hence that $\max\sigma\brac{S_k}\leqslant\phi$.
		Now observe that
		\begin{equation}
		\label{eq:spec_bounds_S_intermediate_upper_bound_2}
			-\tilde{\alpha} \abs{k}^2 {\vbrac{ {\brac{ J_{eq}^{-1/2} \theta }}_\parallel }}^2
			-\tilde{\gamma} \abs{k}^2 {\vbrac{ {\brac{ J_{eq}^{-1/2} \theta }}_\perp }}^2
			\leqslant -\min\brac{\tilde{\alpha},\tilde{\gamma}} \abs{k}^2 {\vbrac{ J_{eq}^{-1/2} \theta }}^2
			\leqslant - \frac{1}{\lambda} \min\brac{\tilde{\alpha},\tilde{\gamma}} \abs{k}^2 \abs{\bar{\theta}}^2.
		\end{equation}
		Combining \eqref{eq:spec_bounds_S_intermediate_upper_bound} and \eqref{eq:spec_bounds_S_intermediate_upper_bound_2} tells us that, for $k\neq 0$,
		\begin{equation*}
			S\brac{k}y\cdot y
			\leqslant -\frac{\phi^2 \abs{k}^2}{C_\sigma} \abs{\bar{\theta}}^2 + 2\phi\bar{\theta}\cdot b
			= -\frac{\phi^2 \abs{k}^2}{C_\sigma} {\vbrac{ \bar{\theta} - \frac{C_\sigma}{\phi\abs{k}^2}b }}^2 + \frac{C_\sigma}{\abs{k}^2} \abs{b}^2
			\leqslant \frac{C_\sigma}{\abs{k}^2} \abs{y}^2
		\end{equation*}
		from which we deduce that $\max\sigma\brac{S_k}\leqslant \frac{C_\sigma}{\abs{k}^2}$.
	\end{proof}

	\noindent
	The bounds on $S$ from \fref{Lemma}{lemma:bounds_spectrum_S} coupled with elementary considerations from linear algebra allow us to deduce bounds on the real parts of the eigenvalues of $M_k$.

	\begin{lemma}[Bounds on the real parts of eigenvalues of $M_k$]
	\label{lemma:bounds_real_part_eval_M}
	For any $k\in\R^3$, and with $\phi$ as in \eqref{phi_c_d_def}, it holds that $\max\re\sigma\brac{M_k}\leqslant\phi$.
	\end{lemma}

	\begin{proof}
		This follows immediately from \fref{Lemmas}{lemma:bounds_spectrum_S} and \ref{lemma:bounds_real_part_evals_using_sym_part_mat}.
	\end{proof}

	\noindent
	To conclude this batch of spectral estimates we obtain bounds on the imaginary parts of the eigenvalues of $M_k$ as a corollary of
	the Gershgorin disk theorem (\fref{Theorem}{thm:Gershgorin}).

	\begin{lemma}[Bounds on the imaginary parts of eigenvalues of $M_k$]
	\label{lemma:bounds_im_part_eval_M}
		For any $k\in\R^3$ it holds that $\max\abs{\imp\sigma\brac{M_k}}\leqslant\frac{\sqrt{7}\tor}{2\kappa}$.
	\end{lemma}

	\begin{proof}
		This follows from \fref{Corollary}{cor:bounds_im_part_evals_using_norm_antisym_part_mar} since
		$$\norm{A}{2}^2 = 2\brac{c^2+d^2} = {\brac{\frac{\tor}{2\kappa}}}^2 \brac{1 - 2\nu\brac{\nu - 2\lambda}} \leqslant {\brac{\frac{\tor}{2\kappa}}}^2.$$
	\end{proof}

	We now record some useful facts about the characteristic polynomial $p$ of $M_k$.
	Computing $p$ was done by using a computer algebra system, and we thus record $M_k$ in \fref{Appendix}{sec:M_k} in a form which can readily be used for computer-assisted algebraic manipulations.

	Upon computing $p$ we observe that it is a polynomial in $k$ of degree 10 and that it only depends on even powers of $\abs{\bar{k}}$ and $k_3$.
	Therefore we may write
	\begin{equation}
	\label{eq:p_sum_r_q}
		p\brac{x,k} = \sum_{q=0}^5 r_q \brac{x, \abs{\bar{k}}, k_3}
	\end{equation}
	where each $r_q$ is a polynomial in $\brac{x, \abs{\bar{k}}, k_3}$ which is \emph{homogeneous of degree $2q$ in $\brac{\abs{\bar{k}}, k_3}$}.
	In particular:
	\begin{equation}
	\label{eq:r_5_and_r_4}
		r_5 \brac{x, \abs{\bar{k}}, k_3} = C_0 x \brac{x^2 + d^2} \abs{k}^{10}\text{ and }
		r_4 \brac{x, \abs{\bar{k}}, k_3} = \abs{k}^6 \brac{t_1\brac{x} \abs{\bar{k}}^2 + t_2\brac{x} k_3^2}
	\end{equation}
	where
	\begin{equation}
	\label{eq:t_i}
		t_i \brac{x} = x^2 \brac{-C_{i,0} + C_{i,1} x + C_{i,2} \brac{x^2 + d^2}}
	\end{equation}
	and
	\begin{align*}
		C_0 &=
			{\brac{\mu + \kappa/2}}^2
			\brac{\alpha + 4\beta/3}
			{\brac{\beta + \gamma}}^2
			/(\nu\lambda^2)
		,\\
		C_{1,0} &=
			\brac{\alpha + 5\beta/3 + \gamma}
			\brac{\beta + \gamma}
			{\brac{\mu + \kappa/2}}^2
			\phi/(\nu\lambda)
		,\,
		C_{2,0} =
			2\brac{\alpha + 4\beta/3}
			\brac{\beta + \gamma}
			\brac{\mu + \kappa/2}
			\phi/(\nu\lambda)
		,\\
		C_{1,1} &= C_{2,1} =
			2 \kappa
			\brac{\mu + \kappa/2}
			\brac{\beta + \gamma}
			\brac{
				2\mu\brac{\alpha + 4\beta/3}
				+ \brac{\mu + \kappa/2}\brac{\beta + \gamma}
			}
			/(\nu\lambda^2)
		,\\
		C_{1,2} &=
			\brac{\mu + \kappa/2}
			\brac{\beta + \gamma}
			\brac{
				2\brac{\alpha + 4\beta/3}\brac{\beta + \gamma}
				+ \brac{\mu + \kappa/2} \brac{
					\brac{\alpha + 5\beta/3 + \gamma} \lambda
					+ \brac{\alpha + 4\beta/3} \nu
				}
			}
			/(\nu\lambda^2)
		\text{ and}\\
		C_{2,2} &=
			2\brac{\mu + \kappa/2}
			\brac{\beta + \gamma}
			\brac{
				\brac{\alpha + 4\beta/3}\brac{\beta + \gamma}
				+ \brac{\mu + \kappa/2} \brac{
					\brac{\alpha + 4\beta/3} \lambda
					+ \brac{\beta + \gamma} \nu/2
				}
			}
			/(\nu\lambda^2)
		.
	\end{align*}
	The exact dependence of these constants on the various physical parameters is not of concern here,
	since all that matters is that all these constants are strictly positive, i.e. $C_0, C_{i,j} > 0$ for all $i,j$.

	We now use Rouch\'{e}'s Theorem (c.f. \fref{Theorem}{thm:Rouche}) and our explicit expressions for the leading factors (with respect to $\abs{k}$) of the characteristic polynomial $p$ of $M_k$
	to control the number of eigenvalues remaining within bounded neighbourhoods of the origin as $\abs{k}$ becomes large.
	This is stated precisely in \fref{Lemma}{lemma:isolation_evals} below, which is another ingredient of the proof of \fref{Proposition}{prop:max_unstable_evals}.

	\begin{lemma}[Isolation of some eigenvalues of $M$ for large wavenumbers]
	\label{lemma:isolation_evals}
		For any $R > \frac{\tor}{2\kappa}$ there exist $K_I > 0$ such that for any $k\in\R^3$, if $\abs{k} > K_I$
		then there are precisely three eigenvalues of $M_k$ in an open ball of radius $R$ about the origin.
	\end{lemma}

	\begin{proof}
		Let $k\in\R^3$ be nonzero, let $p\brac{\cdot,k}$ denote the characteristic polynomial of $M_k$, and let us write $s\defeq p-r_5$ for $r_5$ as in \eqref{eq:r_5_and_r_4}.
		The key observations are that $r_5$ has precisely three roots in $B_R$ when $R>\frac{\tor}{2\kappa}$ and that $s$ is lower-order in $k$ than $r_5$.
		The result then follows from Rouch\'{e}'s Theorem since $r_5$ dominates $s$ for large $\abs{k}$.

		More precisely, let $R > d = \frac{\tor}{2\kappa}$ and let $\tilde{r}_5 \brac{x} \defeq C_0 x \brac{x^2+d^2}$ for $C_0$ as in \eqref{eq:r_5_and_r_4} such that
		$r_5\brac{x,k} = \tilde{r}_5\brac{x}\abs{k}^{10}$. Since $\tilde{r}_5$ is a polynomial whose roots are away from $\partial B_R$,
		since $s\brac{x,k}$ is a polynomial of degree 8 in $k$, and since $\partial B_R$ is compact, it follows that $C_r \defeq \inf_{\partial B_R} \abs{\tilde{r}_5} > 0$
		and that
		$
			C_s \defeq \sup\limits_{\substack{x\in\partial B_R\\ k\neq 0}} \frac{s\brac{x,k}}{\abs{k}^8} < \infty.
		$

		\vspace{-0.675cm}
		\noindent
		So pick $K_I\defeq\sqrt{\frac{C_s}{C_r}}$ and observe that for any $k\in\Z^3$, if $\abs{k} > K_I$ then, on $\partial B_R$,
		\begin{equation}
		\label{eq:isolation_Rouche_bound}
			\abs{r_5\brac{\cdot,k}} = \abs{\tilde{r}_5} \abs{k}^{10} > C_r \abs{k}^8 K_I^2 \geqslant \frac{C_r}{C_s} K_I^2 \abs{s\brac{\cdot,k}} = \abs{s\brac{\cdot,k}}.
		\end{equation}
		Since $r_5\brac{\cdot,k}$ has three roots in $B_R$, namely $0$ and $\pm\frac{\tor}{2\kappa}$,
		we may use \eqref{eq:isolation_Rouche_bound} to deduce from \fref{Theorem}{thm:Rouche} that $p\brac{\cdot,k}$ has three roots in $B_R$.
	\end{proof}

	In \fref{Proposition}{prop:trajectories_evals} below we use the Implicit Function Theorem
	to identify the trajectories of some unstable eigenvalues of $M_k$ when $\abs{k}$ is large.
	In particular we will see in the proof of \fref{Proposition}{prop:max_unstable_evals} that, combining this result with earlier results from this section,
	we may deduce that these eigenvalues are the \emph{most} unstable eigenvalues of $M_k$ for large $k$.
	Here we say that an eigenvalue is unstable when it has strictly positive real part.

	\begin{prop}[Trajectories of some eigenvalues of $M$ for large wavenumbers]
	\label{prop:trajectories_evals}
	There exists $K_T > 0$ and a function $z : \cbrac{k\in\R^3 : \abs{k} > K_T} \to \C$, which is continuously differentiable in the real sense
	(i.e. after identifying $\C$ with $\R^2$ in the canonical way), such that
	\begin{enumerate}
		\item	for every $k\in\R^3$, if $\abs{k} > K_T$ then
			\begin{enumerate}
				\item	$z\brac{k}$ and $\overline{z}\brac{k}$ are eigenvalues of $M\brac{k}$ and
				\item	$\re z\brac{k} > 0$, and
			\end{enumerate}
		\item	$z\brac{k} \to \frac{i\tor}{2\kappa}$ as $\abs{k}\to\infty$.
	\end{enumerate}
	\end{prop}

	\begin{proof}
		Recall that $d=\frac{\tor}{2\kappa}$ and let $p\brac{\cdot,k}$ denote the characteristic polynomial of $M_k$.
		We proceed in three steps: first we define $s$ to be essentially $\abs{\varepsilon}^5 p \brac{\,\cdot\,, \varepsilon^{-1/2}}$
		(such that the study of $s$ about zero is equivalent to the study of $p$ about infinity) and verify that we may apply the Implicit Function Theorem to $s$
		about $\brac{x, \varepsilon} \sim \brac{id, 0}$, second we deduce from explicit computations of $p$ (namely \eqref{eq:r_5_and_r_4}) that, for small nonzero $\varepsilon$,
		$s$ has two roots with strictly positive real parts, and third we write $k\sim\varepsilon^{-1/2}$ to turn our result from step 2 about $\varepsilon\sim 0$
		into a result about $k\sim\infty$ which allows us to conclude that, for large $\abs{k}$, $p$ has two roots with strictly positive real part.

		\textbf{Step 1:}
		Recall (from \eqref{eq:p_sum_r_q} and the preceding discussion) that $p$ only depends on $\abs{\bar{k}}$ and $k_3$,
		so we may write $p\brac{x,k} = \tilde{p}\brac{x,\abs{\bar{k}},k_3}$.
		Now define, for any $x\in\C$ and any $\varepsilon=\brac{\varepsilon_h,\varepsilon_v}\in\R^2_{>0}$,
		$s\brac{x,\varepsilon} \defeq \abs{\varepsilon}_1^5 \,\tilde{p} \brac{x,\frac{\brac{\sqrt{\varepsilon_h},\sqrt{\varepsilon_v}}}{\abs{\varepsilon}_1}}$,
		where $\abs{\,\cdot\,}_1$ denotes the $l^1$ norm.
		It follows from \eqref{eq:p_sum_r_q} that $s\brac{x,\varepsilon} = \sum_{q=0}^5 u_q\brac{x,\varepsilon}$
		for $u_{5-q}\brac{x,\varepsilon}\defeq\abs{\varepsilon}_1^5\, r_q\brac{x, \frac{\sqrt{\varepsilon_h}}{\abs{\varepsilon}_1}, \frac{\sqrt{\varepsilon_h}}{\abs{\varepsilon}_1}}$.
		Since the only dependence of $r_q$ on $k$ is through $\brac{\abs{\bar{k}}, k_3}$, i.e. since $r_q\brac{x,\abs{\bar{k}},k_3} = \tilde{r}_q\brac{x,\abs{\bar{k}}^2,k_3^2}$
		for some $\tilde{r}_q$, we may write $r_q\brac{x,\abs{\bar{k}},k_3} = C_q\brac{x}\bullet{\brac{\abs{\bar{k}}^2,k_3^2}}^{\otimes q}$ for some polynomial $C_q$.
		In particular, it follows that
		\begin{equation*}
			u_2\brac{x,\varepsilon} = C_3\brac{x}\bullet\frac{\varepsilon^{\otimes 3}}{\abs{\varepsilon}_1},\text{ }
			u_3\brac{x,\varepsilon} = \abs{\varepsilon}_1 \,C_2\brac{x} \bullet \varepsilon^{\otimes 2},\text{ }
			u_4\brac{x,\varepsilon} = \abs{\varepsilon}_1^3 \,C_1 \brac{x}, \text{ and }
			u_5\brac{x,\varepsilon} = \abs{\varepsilon}_1^5 \,C_0\brac{x}
		\end{equation*}
		such that, for $q\geqslant 2$, $u_q\brac{x,0}=0$ and both $\partial_x u_q \brac{x,0} = 0$ and $\nabla_\varepsilon u_q \brac{x,0} = 0$.
		Moreover we may compute, using \eqref{eq:r_5_and_r_4}, that
		\begin{equation}
		\label{eq:u_0_1}
			u_0\brac{x,\varepsilon} = C_0 x \brac{x^2+d^2} \eqdef u_0\brac{x}
			\text{ and }
			u_1\brac{x,\varepsilon} = \brac{t_1\brac{x}, t_2\brac{x}} \cdot \varepsilon \eqdef \bar{u}_1\brac{\varepsilon}.
		\end{equation}
		So finally, for $v \defeq s - \brac{u_0 + u_1} = \sum_{q=2}^5 u_q$, we have that $s\brac{x,\varepsilon} = u_0\brac{x} + \bar{u}_1\brac{x}\cdot\varepsilon + v\brac{x,\varepsilon}$
		where $v\brac{x,0}=0$ and both $\partial_x v\brac{x,0} = 0$ and $\nabla_\varepsilon v\brac{x,0}=0$.
		In particular, note that $s\brac{id,0} = u_0\brac{id} = 0$ and that $\partial_x s\brac{id,0} = u_0'\brac{id} = -2C_0 d^2 \neq 0$.

		\textbf{Step 2:}
		By step 1 we may apply \fref{Theorem}{thm:ift} to $s$ about $id$
		to deduce that there exists a number $\xi > 0$ and a function $w : B_{1,\xi}^+ \to \C$
		which is continuously differentiable in the real sense, where $B_{1,\xi}^+$ is the intersection of the first quadrant and the $l^1$-ball of radius $\xi$,
		i.e. $B_{1,\xi}^+ \defeq \setdef{\brac{\varepsilon_h,\varepsilon_v}}{\varepsilon_h,\varepsilon_v > 0 \text{ and }\varepsilon_h + \varepsilon_v < \xi}$,
		such that $w\brac{0} = id$, $s\brac{w\brac{\varepsilon},\varepsilon} = 0$ for every $\varepsilon\in B_{1,\xi}^+$,
		and $\nabla_\varepsilon w\brac{0} = \frac{-\nabla_\varepsilon s\brac{id,0}}{\partial_x s\brac{id,0}}$.
		Moreover we may compute from \eqref{eq:t_i} and \eqref{eq:u_0_1} that 
		$\nabla_\varepsilon w \brac{0} = \frac{1}{2C_0} \begin{pmatrix} C_{1,0} + i C_{1,1}d \\ C_{2,0} + i C_{2,1} d \end{pmatrix}$,
		such that $\re\nabla_\varepsilon w\brac{0}\in\R^2_{>0}$.
		It follows that there exists $0<\sigma<\xi$ such that $\re w\brac{\varepsilon} > 0$ for all $\varepsilon\in B_{1,\sigma}^+$.

		\textbf{Step 3:}
		Pick $K_T \defeq 1/\sqrt{\sigma}$ and define $z$ via, for every $k\in\R^3$ such that $\abs{k}>K_T$, $z\brac{k}\defeq w\brac{\varepsilon\brac{k}}$
		for $\varepsilon\brac{k}\defeq\frac{1}{\abs{k}^4} \brac{\abs{\bar{k}}^2,k_3^2}$.
		Note that $z$ is well-defined on $\cbrac{k\in\R^3:\abs{k}>K_T}$ since, for every $k\in\R^3$, $\abs{k}>K_T \iff \abs{\varepsilon\brac{k}} = 1 / \abs{k}^2 < \sigma$.
		Now observe that, for every $k\in\R^3$ such that $\abs{k}>K_T$,
		$
			\tilde{p}\brac{z\brac{k},k} = \frac{1}{\abs{\varepsilon}_1^5} s\brac{w\brac{\varepsilon\brac{k}},\varepsilon\brac{k}} = 0
		$,
		i.e. indeed $z\brac{k}$ is a root of $p\brac{\cdot,k}$ and hence an eigenvalue of $M_k$.
		Since $M_k$ is a matrix with real entries, we may deduce that $\bar{z}\brac{k}$ is also an eigenvalue of $M_k$.
		Moreover it follows from step 2 above that $\re z\brac{k} > 0$ for every $\abs{k}>K_T$.
		Finally, note that since $w\brac{0}=id$, since $w$ is continuous, and since $\varepsilon\brac{k}$ is continuous away from $k=0$, we may conclude that
		$z\brac{k} \to id$ as $k\to\infty$.
	\end{proof}

	We now have all the ingredients in hand to prove one of the two key results of this section, namely \fref{Proposition}{prop:max_unstable_evals}.
	This result tells us that there exists a \emph{most} unstable eigenvalue of $M_k$, i.e. an eigenvalue with largest strictly positive real part.

	\begin{prop}[Maximally unstable eigenvalues]
	\label{prop:max_unstable_evals}
		There exist $k_*\in\Z^3$ and $w_*\in\C$ with strictly positive real part such that
		\begin{enumerate}
			\item	$w_*$ is an eigenvalue of $M\brac{k_*}$ and
			\item	for every $k\in\Z^3$ and every eigenvalue $w$ of $M\brac{k}$, $\re w \leqslant \re w_*$.
		\end{enumerate}
		We define $\eta_* \defeq \re w_*$.
	\end{prop}

	\begin{proof}
		The key observations are that:
		(i) by combining \fref{Proposition}{prop:trajectories_evals} and \fref{Lemmas}{lemma:bounds_real_part_eval_M} and \ref{lemma:bounds_im_part_eval_M},
		we can show that for $\abs{k}$ large enough, the eigenvalues whose trajectory can be obtained via the implicit function theorem in \fref{Proposition}{prop:trajectories_evals}
		are the most unstable eigenvalues (i.e those with the largest real part) and that
		(ii) by \fref{Proposition}{prop:trajectories_evals} we know that $\re z\brac{k} \to 0$ as $\abs{k}\to\infty$.
		We prove the first observation in step 1 below, and in step 2 we use the first step and the second observation to conclude.

		\textbf{Step 1:}
		We show that there exists $K_* > 0$ such that, for every $\abs{k} > K_*$,
		$
			\re z\brac{k} = \max\limits_{w\in\sigma\brac{M\brac{k}}} \re w.
		$
		Pick $R > \phi^2 + 7d^2$ and note that since $R > d = \frac{\tor}{2\kappa}$ we may pick $K_I = K_I \brac{R}$ as in \fref{Lemma}{lemma:isolation_evals}.
		Let $K_* \defeq \max\brac{K_I, K_T}$ for $K_T$ as in \fref{Proposition}{prop:trajectories_evals},
		let $H$ denote the half-slab $\setdef{w\in\C}{\re z \leqslant \phi, \abs{\imp z} \leqslant \sqrt{7}d}$,
		and let $B_R \subseteq \C$ denote the open ball of radius $R$ about the origin.

		Let $k\in\Z^3$ such that $\abs{k} > K_*$.
		By \fref{Lemmas}{lemma:bounds_real_part_eval_M} and \ref{lemma:bounds_im_part_eval_M} we know that all the eigenvalues of $M\brac{k}$ are in $H$,
		and by \fref{Lemma}{lemma:isolation_evals} we know that exactly three eigenvalues of $M\brac{k}$ are in $B_R \cap H$.
		Moreover, by \fref{Proposition}{prop:trajectories_evals} we know that the three eigenvalues of $M\brac{k}$ in $B_R \cap H$ are precisely $0$
		(since $M\brac{k} \brac{k,0,0} = 0$), $z\brac{k}$, and $\bar{z\brac{k}}$, for $z$ as in \fref{Proposition}{prop:trajectories_evals}.

		In particular, since $R > \phi^2 + 7d^2$ such that no points in the half-slab $H$ have larger real parts than all points in $B_R \cap H$, it follows that
		indeed the eigenvalues of $M\brac{k}$ with largest real part are $z\brac{k}$ and $\bar{z\brac{k}}$.

		\begin{figure}[h!]
			\centering
			\includegraphics{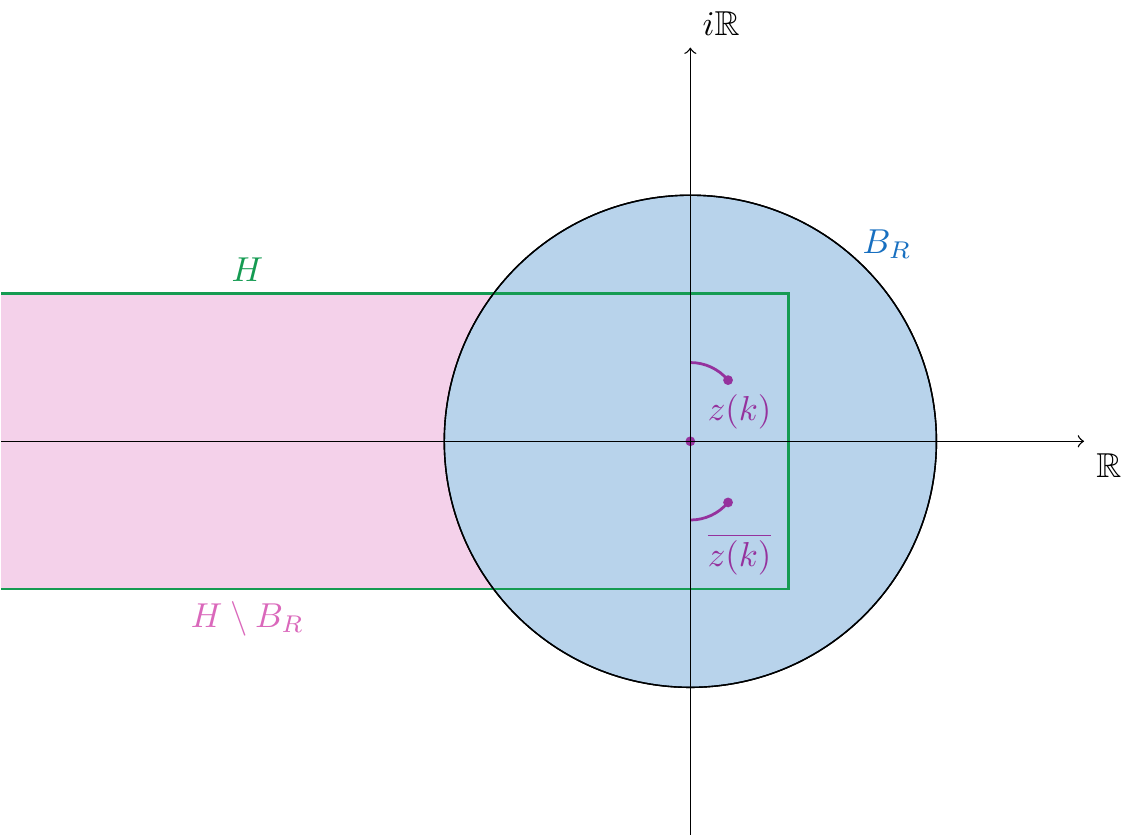}
			\caption{A pictorial summary of step 1 of the proof of \fref{Proposition}{prop:max_unstable_evals}.}
			\label{fig:diag-proof-max-unstable-mode}
		\end{figure}

		\textbf{Step 2:}
		We want to show that the supremum
		\begin{equation*}
			\sup_{k\in\Z^3} \;\max_{w\in\sigma\brac{M\brac{k}}} \re w
		\end{equation*}
		is strictly positive and attained. It is clearly strictly positive since for any $k\in\Z^3$ such that $\abs{k} > K_T$
		it follows from \fref{Proposition}{prop:trajectories_evals} that $z\brac{k}$ is an eigenvalue of $M\brac{k}$ with strictly positive real part.
		To see that this supremum is attained, we write for simplicity
		\begin{equation*}
			s\brac{E} \defeq \sup_{k\in E} \;\max_{w\in\sigma\brac{M\brac{k}}} \re w
		\end{equation*}
		for any $E\subseteq \Z^3$. We thus want to show that $s\brac{\Z^3}$ is attained.
		On one hand, by step 1, the supremum $s\brac{\cbrac{k\in\Z^3 : \abs{k} > K_*}}$
		is achieved. Indeed, we may pick the eigenvalue $z\brac{k_\text{crit}}$ of $M\brac{k_\text{crit}}$ corresponding to any $k_\text{crit}$ such that
		$\abs{k_\text{crit}}$ is equal to the smallest integer strictly larger than $K_*$ which can be written as a sum of squares of integers.
		On the other hand the supremum $s\brac{\cbrac{k\in\Z^3 : \abs{k} \leqslant K_*}}$
		is attained since it is taken over a finite set.
		Since $\Z^3$ is the union of $\cbrac{k\in\Z^3 : \abs{k} > K_*}$ and $\cbrac{k\in\Z^3 : \abs{k} \leqslant K_*}$
		we may conclude that the supremum $s\brac{\Z^3}$ is attained.
	\end{proof}

	We conclude this section with the second of its two key results: \fref{Proposition}{prop:unif_bound_mat_exp}.
	This result is essential in the construction of the semigroup associated with the linearized operator.
	This construction is performed in \fref{Section}{sec:semigroup} below.

	\begin{prop}[Uniform bound on the matrix exponentials]
	\label{prop:unif_bound_mat_exp}
		Let $\eta_*$ be as in \fref{Proposition}{prop:max_unstable_evals}.
		There exists $C_S > 0$ such that for every $k\in\Z^3$ and every $t>0$, $\abs{e^{t M_k}}\leqslant C_S \brac{1+t^8} e^{\eta_* t}$.
		As a consequence, for every $\varepsilon > 0$ there exists $C_S \brac{\varepsilon} > 0$ such that for every $k\in\Z^3$ and every $t>0$,
		$\abs{e^{t\hat{\B}_k}}\leqslant C_S \brac{\varepsilon}  e^{\brac{\eta_* + \varepsilon} t}$.
	\end{prop}

	\begin{proof}
		Naively, one may seek to use the bound from \fref{Corollary}{cor:bounds_mat_exp} to control $e^{t M_k}$.
		However, this bounds only holds up to a constant \emph{dependent on $k$}.
		To circumvent this issue, we observe that alternatively one may bound $e^{t M_k}$ using its symmetric part (as per \fref{Lemma}{lemma:bounds_real_part_evals_using_sym_part_mat}).
		Coupling this observation with the fact that we have an upper bound which decays as $\abs{k}^{-2}$ for the spectrum of $S_k$, namely \fref{Lemma}{lemma:bounds_spectrum_S},
		we see that for sufficiently large $\abs{k}$ the exponential $e^{tM_k}$ grows at most like $e^{\eta_* t}$.
		It thus suffices to use \fref{Corollary}{cor:bounds_mat_exp} for the finitely many modes with non-large $\abs{k}$, in which case the dependence of the constant on $k$ is harmless.

		More precisely: let $K_S \defeq \sqrt{\frac{C_\sigma}{\eta_*}}$ where $C_\sigma$ is as in \fref{Lemma}{lemma:bounds_spectrum_S},
		write $C\brac{k} \defeq C\brac{M_k}$ for $C\brac{M}$ as in \fref{Corollary}{cor:bounds_mat_exp}, and let
		$
			C_S \defeq \max\brac{1, \max\limits_{\abs{k} < K_S} C\brac{k}} > 0.
		$
		Then, for every $k\in\Z^3$,
		if $\abs{k} \geqslant K_S$ then
		$\frac{C_\sigma t}{\abs{k}^2} \leqslant \frac{C\sigma t}{K_S^2} = \eta_* t$
		and hence, by \fref{Lemmas}{lemma:bounds_real_part_evals_using_sym_part_mat} and \ref{lemma:bounds_spectrum_S},
		$
			\norm{e^{tM_k}}{\Leb\brac{l^2,\,l^2}}
			\leqslant e^{\frac{C_\sigma t}{\abs{k}^2}}
			\leqslant e^{\eta_* t},
		$
		and if $\abs{k} < K_S$ then by \fref{Corollary}{cor:bounds_mat_exp}, the choice of $C_S$, and \fref{Proposition}{prop:max_unstable_evals}
		\begin{equation*}
			\norm{e^{tM_k}}{\Leb\brac{l^2,\,l^2}}
			\leqslant C\brac{k} \brac{1 + t^8} e^{\brac{\max\re\sigma\brac{M_k}} t}
			\leqslant C_S \brac{1 + t^8} e^{\eta_* t}
		\end{equation*}
		from which the first part of the result follows.
		To obtain the second part we simply use the fact that polynomials of arbitrarily large degree can be controlled by exponentials of arbitrarily slow growth,
		i.e. the fact that for every $j\in\N$ and every $\varepsilon > 0$ there exists $C = C\brac{j,\varepsilon} > 0$ such that, for every $t\geqslant 0$, $1+t^j \leqslant Ce^{\varepsilon t}$.
	\end{proof}

\subsection{The semigroup}
\label{sec:semigroup}

	In this section we proceed in a standard fashion and use \fref{Proposition}{prop:unif_bound_mat_exp}
	to construct the semigroup associated with the linearized problem as recorded after Leray projection in \eqref{eq:PDE_compact_not_lin_semigroup}.

	\begin{prop}[Semigroup for the linearization]
	\label{prop:semigroup_lin_prob}
		Let $\eta_*$ be as in \fref{Proposition}{prop:max_unstable_evals}. For every $t\geqslant 0$ we define the operator $e^{t\B}$
		on $L^2 \brac{\T^3,\,\R^8}$ via the Fourier multiplier
		${\brac{e^{t\B}}}^\wedge \brac{k} \defeq e^{t\hat{\B}_k}$ for every $k\in\Z^3$ and we define $e^{t\Leb}$ as
		$
			e^{t\Leb} \defeq e^{t\B} \oplus e^{t\sbrac{\bar{\Omega}_{eq},\,\cdot\,}} \oplus 1
		$,
		i.e. for every $\brac{f, \JO, J_{33}} \in L^2 \brac{\T^3,\, \R^8} \times L^2 \brac{\T^3,\, \R^{2 \times 2}} \times L^2 \brac{\T^3,\, \R}$,
		$
			e^{t\Leb} \brac{ f , \JO , J_{33} }
			\defeq \brac{ e^{t\B} f , e^{t\sbrac{\bar{\Omega}_{eq},\,\cdot\,}} \JO , J_{33} }
		$.

		Then ${\brac{e^{t\Leb}}}_{t\geqslant 0}$ is a semigroup on $L^2$ and for every $\varepsilon > 0$ it is an $\brac{\eta_* + \varepsilon}$-contractive semigroup with domain
		$
		H^2 \brac{\T^3,\,\R^6} \times L^2 \brac{\T^3,\,\R^2} \times L^2\brac{\T^3,\,\R^{2\times 2}} \times L^2 \brac{\T^3,\,\R}
			\eqdef\mathfrak{D}
		$
		and generator $\Leb$.

		Moreover, for every $\varepsilon > 0$ there exists a constant $C_S \brac{\varepsilon} > 0$ such that,
		for every $p,q,r \geqslant 0$ and every $t\geqslant 0$, $e^{t\Leb}$ is a bounded operator on
		$
			H^{p,q,r} \defeq H^p \brac{\T^3, \R^8} \times H^q \brac{\T^3, \R^{2\times 2}} \times H^r \brac{\T^3, \R}
		$
		such that for any $\brac{f,\JO,J_{33}}\in H^{p,q,r}$,
		$
			\norm{
				e^{t\Leb} \brac{f, \JO, J_{33}}
			}{
				H^{p,q,r}
			}^2
			\leqslant C^2_S \brac{\varepsilon} e^{2\brac{\eta_* + \varepsilon}t} \norm{\brac{f,\JO,J_{33}}}{H^{p,q,r}}^2
		$, where
		\begin{equation*}
			\norm{\brac{f,\JO,J_{33}}}{H^{p,q,r}}^2 \defeq \norm{f}{H^p}^2 + \norm{\JO}{H^q}^2 + \norm{J_{33}}{H^r}^2.
		\end{equation*}
		Finally: the semigroup propagates incompressibility.
		More precisely: let
		\begin{equation*}
			X_0 = \brac{u_0, \omega_0, a_0, \JO_0, {\brac{J_{33}}}_0}
			\in L^2 \brac{\T^3,\, \R^3} \times L^2 \brac{\T^3,\, \R^3} \times L^2 \brac{\T^3,\, \R^2} \times L^2 \brac{\T^3,\, \R^{2 \times 2}} \times L^2 \brac{\T^3,\, \R}
		\end{equation*}
		and let
		$X\brac{t,\cdot} = \brac{u, \omega, a, \JO, J_{33}} \brac{t,\cdot} \defeq e^{t\Leb} X_0$ for all $t > 0$.
		If $u_0$ is incompressible, in a distributional sense, then $u\brac{t,\cdot}$ is incompressible for all time $t > 0$.
	\end{prop}

	\begin{proof}
		\textbf{Step 1:}
		We begin by constructing the semigroup $e^{t\B}$.
		Note that, in this proof, all matrix norms are norms in $\Leb\brac{l^2,\,l^2}$.
		To construct this semigroup we will use \fref{Proposition}{prop:const_semigroup_Fourier} and must therefore verify that
		(i) for every $\varepsilon > 0$ there exists $C_S\brac{\varepsilon} > 0$ such that for every $k\in\Z^3$ and every $t>0$,
		$\norm{e^{t\hat{\B}_k}}{} \leqslant C_S\brac{\varepsilon} e^{\brac{\eta_* + \varepsilon}t}$ and that
		(ii) there exists $C_D > 0$ such that for every $\brac{v,\theta,b}\in\R^3\times\R^3\times\R^2$,
		$\vbrac{\hat{\B}_k \brac{v,\theta,b}}\leqslant C_D\brac{\jap{k}{4}\brac{\abs{u}^2+\abs{\omega}^2} + \abs{a}^2}$.
		Note that (ii) follows immediately from the expression provided for $\hat{\B}$ in \eqref{eq:def_B}.
		To obtain (i) we note that it follows from \fref{Lemmas}{lemma:B_Q_act_on_V_k} and \ref{lemma:sim_mat_act_on_quot_spaces} that
		\begin{equation*}
			\hat{\B}^n_k
			= {\brac{ \bar{Q}_k M_k Q_k }}^n
			= \bar{Q}_k M^n_k Q_k
			\text{ for every } n \geqslant 1
		\end{equation*}
		whilst $\hat{B}_k^0 = \id = \proj_{V_k} + \proj_{V_k^\perp} = \proj_{V_k} + \bar{Q}_k M_k^0 Q_k$. Therefore
		\begin{equation}
		\label{eq:exp_B_sim_exp_M}
		e^{t\hat{\B}_k} = \proj_{V_k} +\, \bar{Q}_k e^{tM_k} Q_k
		\end{equation}
		where
		\begin{equation}
		\label{eq:bound_Qk}
			\frac{1}{2}\brac{ \norm{Q_k}{}^2 + \norm{\bar{Q}_k}{}^2}
			\leqslant \norm{\frac{ik\times}{\abs{k}}}{}^2 + \frac{1}{2} \brac{ \norm{J_{eq}^{1/2}}{}^2 + \norm{J_{eq}^{-1/2}}{}^2 } + \frac{1}{2} \brac{s + s\inv} \norm{R}{}^2
			\leqslant C_b
		\end{equation}
		for some $C_b > 0$ \emph{independent} of $k$.
		We may thus deduce from \eqref{eq:exp_B_sim_exp_M}, \eqref{eq:bound_Qk}, and \fref{Proposition}{prop:max_unstable_evals} that (i) holds.

		With (i) and (ii) in hand we apply \fref{Proposition}{prop:const_semigroup_Fourier} and obtain that $e^{t\B}$ is a semigroup on $L^2$ which is $\brac{\eta_*+\varepsilon}$-contractive
		on all $H^r$ spaces, for $r\geqslant 0$, with domain $H^2\brac{\T^3,\,\R^3\times\R^3} \times L^2\brac{\T^3,\,\R^3}$ and generator $\B$.

		\textbf{Step 2:}
		Now we construct the full semigroup $e^{t\Leb}$.
		First observe that, since $\sbrac{\bar{\Omega}_{eq},\,\cdot\,}$ is a finite-dimensional linear operator,
		${\brac{e^{t\sbrac{\bar{\Omega}_{eq},\,\cdot\,}}}}_{t\geqslant 0}$ is a semigroup on $\R^{2\times 2}$ and moreover
		\begin{equation}
		\label{eq:dom_gen_L_0}
			\text{the domain of }
			{\brac{ e^{t \sbrac{\bar{\Omega}_{eq},\,\cdot\,}} }}_{t\geqslant 0}
			\text{ is }
			\R^{2\times 2}
			\text{ and its generator is }
			\sbrac{\bar{\Omega}_{eq},\,\cdot\,}.
		\end{equation}
		Moreover, \fref{Lemma}{lemma:antisymmetry_commutator_on_space_symmetric_matrices} tells us that $\sbrac{\bar{\Omega}_{eq},\,\cdot\,}$ is antisymmetric,
		and thus it follows from \fref{Lemma}{lemma:bounds_mat_exp_using_sym_part}
		that ${\brac{ e^{t\sbrac{\bar{\Omega}_{eq},\,\cdot\,}} }}_{t\geqslant 0}$ is a contractive semigroup, i.e.
		\begin{equation}
		\label{eq:semigroup_L_0_contractive}
			\norm{e^{t\sbrac{\bar{\Omega}_{eq},\,\cdot\,}}}{\Leb\brac{l^2,l^2}} \leqslant 1.
		\end{equation}
		From \eqref{eq:semigroup_L_0_contractive} and step 1 it follows that
		$
			e^{t\Leb} = e^{t\calN} \oplus e^{t\sbrac{\bar{\Omega}_{eq},\,\cdot\,}} \oplus 1
		$
		is a direct sum of semigroups which are, for every $\varepsilon > 0$,
		$\brac{\eta_* + \varepsilon}$-contractive (since contractive semigroups are $\eta$-contractive for any $\eta >0$ and since $1 = e^0$ is the trivial semigroup, which is contractive),
		and is hence $\brac{\eta_* + \varepsilon}$-contractive itself.
		Moreover, it follows from the observation \eqref{eq:dom_gen_L_0} and step 1 that the domain and generator of $e^{t\Leb}$ are as claimed.
		Finally the $H^{p,q,r}$ estimates follow immediately from \eqref{eq:semigroup_L_0_contractive} and the $H^r$ estimates of step 1,
		upon observing that since, for each $t>0$, $e^{t\sbrac{\bar{\Omega}_{eq},\,\cdot\,}}$ is a linear operator \emph{independent of the spatial variable $x$},
		it commutes with partial derivatives and with the Fourier transform.

		\textbf{Step 3:}
		We now prove that incompressibility is propagated.
		Let us write $Y\brac{t,\cdot} \defeq \brac{ u, \omega, a }\brac{t,\cdot}$.
		The key observation is that, as a consequence of \fref{Lemma}{lemma:B_Q_act_on_V_k},
		$
			\pdt\brac{ \brac{k,0,0}\cdot \hat{Y}_k}
			= \brac{k,0,0}\cdot\hat{\B}_k \hat{Y}_k
			= 0
		$
		for every $k\in\Z^3$.
		In particular, if $\nabla\cdot_0 u = 0$ then indeed
		\begin{equation*}
			\brac{\nabla \cdot u}\brac{t,\cdot}
			= \sum_{k\in\Z^3} \brac{k,0,0} \cdot \hat{Y}_k \brac{t}
			= \sum_{k\in\Z^3} \brac{k,0,0} \cdot \hat{Y}_k \brac{0}
			= \nabla\cdot u_0
			= 0.
		\end{equation*}
	\end{proof}

\subsection{A maximally unstable solution}
\label{sec:max_unstable_sol}

	In this section we construct a maximally unstable solution of the linearized problem \eqref{eq:PDE_compact_not_lin_semigroup}.
	Recall that \eqref{eq:PDE_compact_not_lin_semigroup} is obtained from the linearized problem by Leray projection.
	In particular, since \eqref{eq:PDE_compact_not_lin_semigroup} is invariant under the transformation $u \mapsto u + C$ for any constant $C$,
	the component corresponding to $u$ in this maximally unstable solution will have average zero
	(this is as expected in light of the Galilean equivariance of the original system of equations, as discussed at the end of \fref{Section}{sec:micropolar}).
	Note that, just as \fref{Proposition}{prop:semigroup_lin_prob} is essentially a semigroup version of \fref{Proposition}{prop:unif_bound_mat_exp},
	\fref{Proposition}{prop:max_unstable_sol} below is essentially a semigroup version of \fref{Proposition}{prop:max_unstable_evals}.
	
	\begin{prop}[Maximally unstable solution]
	\label{prop:max_unstable_sol}
		Let $\eta_*$ be as in \fref{Proposition}{prop:max_unstable_evals}.
		There is a smooth function $Y : \cobrac{0,\infty} \times \T^3 \to \R^8$ such that $\pdt Y = \B Y$ and
		$
			\norm{Y\brac{t,\cdot}}{H^r\brac{\T^3, \R^8}} = e^{\eta_* t} \norm{Y\brac{0,\cdot}}{H^r\brac{\T^3, \R^8}}
		$
		for every $t\geqslant 0$ and every $r\geqslant 0$.
		Moreover, if we write $Y = \brac{u,\omega,a} \in \R^3 \times \R^3 \times \R^2$, then $\nabla\cdot u = 0$, and for every $t\geqslant 0$ and every $r\geqslant 0$
		\begin{align*}
			\norm{u\brac{t,\cdot}}{H^r\brac{\T^3,\, \R^3}} &= e^{\eta_* t} \norm{u\brac{0,\cdot}}{H^r\brac{\T^3,\, \R^3}},\\
			\norm{\omega\brac{t,\cdot}}{H^r\brac{\T^3,\, \R^3}} &= e^{\eta_* t} \norm{\omega\brac{0,\cdot}}{H^r\brac{\T^3,\, \R^3}},\text{ and }\\
			\norm{a\brac{t,\cdot}}{H^r\brac{\T^3,\, \R^3}} &= e^{\eta_* t} \norm{a\brac{0,\cdot}}{H^r\brac{\T^3,\, \R^3}}.
		\end{align*}
	\end{prop}

	\begin{proof}
		Let $k_*\in\Z^3$ and $w_*\in\C$ be as in \fref{Proposition}{prop:max_unstable_evals} and recall that $\eta_*\defeq\re w_*$.
		It follows from \fref{Lemma}{lemma:B_Q_act_on_V_k} and \fref{Lemma}{lemma:sim_mat_act_on_quot_spaces} that, for any $k\in\Z^3$, $\hat{\B}_k$ and $M_k$ are similar,
		so in particular $w_*$ is an eigenvalue of $\hat{\B}_k$ and thus there exists $v_*\in\C^8$ such that $\hat{\B}\brac{k_*} v_* = w_* v_*$.
		Now define, for every $t\geqslant 0$ and every $x\in\T^3$,
		$
			Y\brac{t,x}\defeq v_* e^{w_* t + i k_* \cdot x} + v_*^\dagger e^{w_*^\dagger t - i k_* \cdot x}
		$
		where, for any complex number $w$, we denote its complex conjugate by $w^\dagger$.
		For a complex matrix $A$ we will write, in this proof only, $A^\dagger$ to denote its entry-wise complex conjugate (and not its conjugate transpose).

		Observe that $Y^\dagger = Y$ and hence $Y$ is real-valued.
		Note that since $\hat{\B}_k = Q_k M_k \bar{Q}_k$ (which follows from \fref{Lemmas}{lemma:B_Q_act_on_V_k} and \ref{lemma:sim_mat_act_on_quot_spaces}),
		since $M_k$ has real entries and is even in $k$ (i.e. $M_{-k} = M_k$), and since ${Q\brac{k}}^\dagger = Q\brac{-k}$ and ${\bar{Q}\brac{k}}^\dagger = \bar{Q}\brac{-k}$,
		we obtain that ${\hat{B}\brac{k}}^\dagger = \hat{\B}\brac{-k}$ and hence $\brac{w_*^\dagger,\, v_*^\dagger}$ is an eigenvalue-eigenvector pair for $\hat{\B}\brac{-k_*}$.
		Therefore
		\begin{equation}
			\pdt Y
			= w_* v_* e^{w_* t + i k_* \cdot x} + w_*^\dagger v_*^\dagger e^{w_*^\dagger t - i k_* \cdot x}
			= \hat{\B}\brac{k_*} v_* e^{w_* t + i k_* \cdot x} + \hat{\B}\brac{-k_*} v_*^\dagger e^{w_*^\dagger t - i k_* \cdot x}
			= \B Y.
		\end{equation}

		Now we argue that $u \defeq \brac{Y_1,Y_2,Y_3}$ is divergence-free.
		Observe that if $k_* = 0$ then $Y$ is constant in the spatial variable $x\in\T^3$ and thus $u$ is constant and hence divergence-free.
		Now consider the case $k_*\neq 0$.
		Note that we have proved in \fref{Lemma}{lemma:B_Q_act_on_V_k} that, for all $k\in\Z^3$, $\im\hat{\B}_k \subseteq V_k^\perp$ and hence $\brac{k,0,0}\cdot v = 0$
		for any eigenvector $v$ of $\hat{\B}_k$.
		We may thus compute:
		\begin{equation}
			\nabla\cdot u
			= \sum_{k\in\Z^3} k\cdot \hat{u}\brac{k}
			= \brac{k_*,0,0} \cdot \hat{Y}\brac{k_*} + \brac{-k_*,0,0}\cdot\hat{Y}\brac{-k_*}
			= 0.
		\end{equation}

		Finally, observe that for any $j=1,\,\dots,\,8$, $Y_j \brac{t,x} = {(v_*)}_j\, e^{w_* t + i k_* \cdot x} + {(v_*^\dagger)}_j\, e^{w_*^\dagger t - i k_* \cdot x}$ and hence,
		proceeding as above yields
		\begin{equation*}
			\norm{Y_j\brac{t,\cdot}}{H^r}^2
			= \jap{k_*}{2r} \abs{ {(v_*)}_j }^2\, \abs{ e^{\re w_* t} }^2
			+ \jap{k_*}{2r} \abs{ {(v_*^\dagger)}_j }^2\, \abs{ e^{\re w_*^\dagger t} }^2
			= 2 \jap{k_*}{2r} \abs{v_*}\, e^{2 \eta_* t}
			= e^{2\eta_* t}\,\norm{Y_j\brac{0,\cdot}}{H^r}^2.
		\end{equation*}
		We can thus conclude that, for $u=\brac{Y_1, Y_2, Y_3}$, $\omega=\brac{Y_4, Y_5, Y_6}$, and $a=\brac{Y_7, Y_8}$,
		\begin{equation*}
			\norm{u\brac{t,\cdot}}{H^r}^2 = e^{2\eta_* t} \norm{u\brac{0,\cdot}}{H^r}^2,\text{ }
			\norm{\omega\brac{t,\cdot}}{H^r}^2 = e^{2\eta_* t} \norm{\omega\brac{0,\cdot}}{H^r}^2, \text{ and }
			\norm{a\brac{t,\cdot}}{H^r}^2 = e^{2\eta_* t} \norm{a\brac{0,\cdot}}{H^r}^2.
		\end{equation*}
	\end{proof}


\section{Nonlinear energy estimates}
\label{sec:nonlinear_estimates}

	In this section we perform the nonlinear energy estimates necessary to carry out the bootstrap instability argument in \fref{Section}{sec:bootstrap}.
	First we record the precise form of the nonlinearities and introduce, in \fref{Definitions}{def:m_and_its_domain} and \ref{def:small_energy_regime}, notation used in the remainder of the paper.
	In \fref{Section}{sec:est_nonlinearity} we obtain bounds on the nonlinearity in $L^2$.
	We record the energy-dissipation relations satisfied by solutions of \eqref{eq:stat_prob_lin_mom}--\eqref{eq:stat_prob_microinertia} and their derivatives in \fref{Section}{sec:ED_identities}.
	In \fref{Section}{sec:est_interactions} we estimate the interaction terms appearing in the relations obtained in the preceding section.
	Finally we use the results of \fref{Sections}{sec:ED_identities} and \ref{sec:est_interactions} in \fref{Section}{sec:chain_energy_ineq}
	to obtain a chain of energy inequalities from which we deduce the key bootstrap energy inequality.

	Writing the problem compactly using the same notation as that which was used in \eqref{eq:PDE_compact_not_lin} and defining $Z \defeq X-X_{eq}$ and $q\defeq p - p_{eq}$ we may write
	the original problem \eqref{eq:stat_prob_lin_mom}--\eqref{eq:stat_prob_microinertia} as
	\begin{equation}
	\label{eq:PDE_compact_not}
		\pdt DZ = \widetilde{\Leb} Z + \Lambda\brac{q} + N\brac{Z} \text{ subject to } \nabla\cdot u = 0.
	\end{equation}
	For simplicity we will abuse notation in this section and write the components of the perturbative unknown $Z$ as $Z = \brac{u ,\omega, J}$.
	This does conflict with the notation used in \fref{Section}{sec:ana_lin_prob} for $X$.
	However confusion may be avoided by noting that all the unknowns appearing in this section are \emph{perturbative}, i.e. $\brac{u, \omega, J}$ will always denote the components of $Z$.
	We also abuse notation and, in this section, write $p = q$.

	Using this notation we have that $N=\brac{N_1, N_2, N_3}$ for
	\begin{equation}
	\label{eq:record_nonlin_13}
		N_1\brac{Z} = -\brac{u\cdot\nabla}u,\, N_3\brac{Z} = \sbrac{\Omega,J} - \brac{u\cdot\nabla}J,
	\end{equation}
	and
	\begin{align}
		N_2\brac{Z} \hspace{-1.45cm}&&= -J_{eq} \brac{u\cdot\nabla}\omega - \brac{I+JJ_{eq}\inv}\inv \brac{
			\omega\times J\omega + \omega_{eq}\times J\omega + \omega\times J_{eq} \omega + \omega \times J \omega_{eq}
		}
		\nonumber
		\\
		&&\quad- JJ_{eq}\inv \brac{I + JJ_{eq}\inv}\inv \big(
			\kappa\nabla\times u - 2\kappa\omega + \brac{\tilde{\alpha} - \tilde{\gamma}} \nabla\brac{\nabla\cdot\omega} + \tilde{\gamma}\Delta\omega
		\nonumber
		\\
		&&\hspace{5cm}-\, \omega\times J_{eq}\omega_{eq} - \omega_{eq}\times J\omega_{eq} - \omega_{eq}\times J_{eq} \omega
		\big)
	\label{eq:record_nonlin_2}
	\end{align}

	Note that $Z$ being a solution of \eqref{eq:PDE_compact_not} is equivalent to $Z$ being a solution of
	\begin{equation}
	\label{eq:PDE_compact_not_semigroup}
		\pdt Z = \Leb Z + \Lambda\brac{p} + D\inv N\brac{Z} \text{ subject to } \nabla\cdot u = 0,
	\end{equation}
	for $\Leb$ as in \eqref{eq:PDE_compact_not_lin_semigroup}.
	The fact that both of these formulations are equivalent is very handy since \eqref{eq:PDE_compact_not} is particularly convenient for energy estimates
	whilst semigroup theory may be readily applied to \eqref{eq:PDE_compact_not_semigroup}.

	\begin{definition}
	\label{def:m_and_its_domain}
		Let $\mathfrak{B}\defeq\setdef{A\in\R^{n\times n}}{\norm{A}{\text{op}} <  1}$ and define	
		$m\brac{A} \defeq {\brac{I+A}}\inv$ for any $A\in\mathfrak{B}$.
		Note that $m$ is well-defined by \fref{Corollary}{cor:invertibility_pert_identity}.
	\end{definition}

	\begin{definition}[Small energy regime]
	\label{def:small_energy_regime}
		Since $n=3$ there exists $C_0 > 0$ such that
		$\norm{J}{\infty} \leqslant C_0 \norm{J}{H^4}$
		for every $J\in H^4 \brac{\T^3,\,\R^{3\times 3}}$. 
		We define $\delta_0 \defeq \min\brac{ \frac{1}{2},\frac{1}{2 C_0 \norm{J_{eq}\inv}{\infty}}}$.
	\end{definition}

\subsection{Estimating the nonlinearity}
\label{sec:est_nonlinearity}
	
	In this section we record some preliminary results in \fref{Lemmas}{lemma:small_energy_regime} and \ref{lemma:conseq_small_energy_regime}
	and then estimate the nonlinearity in $L^2$ in \fref{Proposition}{prop:est_nonlinearity}.

	First we record for convenience some elementary consequences of the Sobolev embeddings.
	In particular \fref{Lemma}{lemma:small_energy_regime} tells us that in the small energy regime $Z$, $\nabla Z$, and $\nabla^2 Z$ are $L^\infty$-multipliers,
	which simplifies many of the estimates below.
	It is precisely because the estimates are easier to perform when $\nabla^2 Z$ is in $L^\infty$ that we have chosen to close the estimates in $H^4$.

	\begin{lemma}
	\label{lemma:small_energy_regime}
		Let $Z\in H^4\brac{\T^3, \R^3 \times \R^{2\times 2} \times \R}$.
		\begin{enumerate}
			\item	There exists $C>0$ independent of $Z$ such that $\norm{Z}{L^\infty} + \norm{\nabla Z}{L^\infty} + \norm{\nabla^2 Z}{L^\infty} \leqslant C \norm{Z}{H^4}$.
			\item	For any polynomial $p$ with no zeroth-order term there exists $C\brac{p} > 0$ such that if $\norm{Z}{H^4} \leqslant 1$
				then $p\brac{\norm{Z}{H^4}}\leqslant C\brac{p} \norm{Z}{H^4}$.
		\end{enumerate}
	\end{lemma}

	\begin{proof}
		(1) follows from the Sobolev embedding $H^2\brac{\T^3}\hookrightarrow L^\infty\brac{\T^3}$ and (2) is immediate.
	\end{proof}

	The result below ensures, when combined with \fref{Corollary}{cor:invertibility_pert_identity},
	that the nonlinearities written in \eqref{eq:record_nonlin_13} and \eqref{eq:record_nonlin_2} are well-defined.
	Note that the only subtlety in ensuring that the nonlinearities are well-defined comes from the presence of ${\brac{ I + JJ_{eq\inv}}}\inv = m\brac{JJ_{eq}\inv}$.
	This terms owes its appearance to our choice to write \eqref{eq:stat_prob_ang_mom} in a form such that the left-hand side is $J_{eq}\pdt\omega$, and not $J\pdt\omega$.
	The former is more convenient since it makes it possible to use semigroup theory.

	\begin{lemma}
	\label{lemma:conseq_small_energy_regime}
	Let $\delta_0$ be as in the small energy regime (c.f. \fref{Definition}{def:small_energy_regime}). If $\norm{Z}{H^4}~\leqslant~\delta_0$ then $\norm{JJ_{eq}\inv}{\infty} \leqslant \frac{1}{2}$ and
	$\norm{m\brac{JJ_{eq}\inv}}{\infty} \leqslant 2$.
	\end{lemma}

	\begin{proof}
		If $\norm{Z}{H^4} \leqslant \delta_0$ then
		$\norm{JJ_{eq}\inv}{\infty} \leqslant \norm{J}{\infty} \norm{J_{eq}\inv}{\infty}
		\leqslant C_0 \norm{J}{H^4} \norm{J_{eq}\inv}{\infty}
		\leqslant C_0 \delta_0 \norm{J_{eq}\inv}{\infty}
		\leqslant \frac{1}{2}$
		and hence, by \fref{Corollary}{cor:invertibility_pert_identity}, $\norm{m\brac{JJ_{eq}\inv}}{\infty} \leqslant \frac{1}{1-\norm{JJ_{eq}\inv}{\infty}} \leqslant 2$.
	\end{proof}

	\noindent
	We now prove the main result of this section, namely the $L^2$ bound on the nonlinearity.

	\begin{prop}[Estimate of the nonlinearity]
	\label{prop:est_nonlinearity}
		Let $\delta_0$ be as in the small energy regime (c.f. \fref{Definition}{def:small_energy_regime}). There exists $C_N > 0$ such that if $\norm{Z}{H^4} \leqslant \delta_0$ then
		$\norm{N\brac{Z}}{L^2} \leqslant C_N \norm{Z}{H^2}^2$.
	\end{prop}

	\begin{proof}
		Recall that $N = \brac{N_1,N_2,N_3}$ is recorded in \eqref{eq:record_nonlin_13}--\eqref{eq:record_nonlin_2}.
		In particular, one immediately obtains that
		$
			\norm{N_1}{L^2} + \norm{N_3}{L^2} \lesssim \norm{Z}{L^2} \norm{Z}{H^1} \lesssim \norm{Z}{H^2}^2
		$.
		Dealing with $N_2$ is only slightly more delicate. 
		Considering $m\brac{JJ_{eq}\inv}$ as a fixed $L^\infty$ multiplier
		we see that all terms in $N_2$ are quadratic or cubic in $Z$
		(more precisely: the only cubic term is $-{\brac{I + JJ_{eq}\inv}}\inv\brac{\omega\times J\omega}$).
		We can thus use the generalized H\"{o}lder inequality as well as the Sobolev embeddings
		$H^1 \brac{\T^3} \hookrightarrow L^6 \brac{\T^3} \hookrightarrow L^p \brac{\T^3}$ for all $p\in\sbrac{1,6}$
		and $H^2 \brac{\T^3} \hookrightarrow L^\infty \brac{\T^3}$
		to obtain that
		$
			\norm{N_2}{L^2} \lesssim \norm{Z}{H^2}^2 + \norm{Z}{H^2}^3
			\lesssim \brac{1+\delta_0}\norm{Z}{H^2}^2
		$.
	\end{proof}

	\begin{remark}
	\label{rmk:Leray_proj_and_nonlinearity}
		The operator which must be estimated in the bootstrap instability argument is actually $\mathbb{P} N$ (and not merely $N$ as is done in \fref{Proposition}{prop:est_nonlinearity} above),
		where $\mathbb{P} = \mathbb{P}_L \oplus \id \oplus \id$
		for $\mathbb{P}_L$ denoting the Leray projector.
		However, since $\hat{\mathbb{P}}_L \brac{k}= \proj_{{\brac{\Span k}}^\perp} = I - \frac{k \otimes k}{\abs{k}^2}$ for every $k\in\Z^3$, i.e. since $\mathbb{P}_L$
		is a bounded Fourier multiplier, it follows that it is bounded on $L^2$.
	\end{remark}

\subsection{The energy-dissipation identities}
\label{sec:ED_identities}

	In this section we begin by recording the energy-dissipation relation and then remark on the coercivity of the dissipation.

	\begin{prop}[The energy-dissipation relation]
	\label{prop:ed_rel}
		If $Z$ solves \eqref{eq:PDE_compact_not} then for any multi-index $\alpha \in \N^3$
		\begin{equation*}
			\frac{1}{2} \Dt \norm{ \sqrt{D} \brac{\partial^\alpha Z}}{L^2}^2 + \ds\brac{ \partial^\alpha u, \partial^\alpha \omega}
			= B\brac{ \partial^\alpha \bar{\omega}, \partial^\alpha a} + \int_{\T^3} \partial^\alpha N\brac{Z} \cdot \partial^\alpha Z
		\end{equation*}
	\end{prop}
	where
	\begin{equation*}
		\ds\brac{u,\omega}\defeq
		\int_{\T^3}
			\frac{\mu}{2} \abs{\symgrad u}^2
			+ 2 \kappa{\vbrac{\half\nabla\times u - \omega}}^2
			+ \alpha\abs{\nabla\cdot\omega}^2
			+ \frac{\beta}{2} \abs{\symgrad^0 \omega}^2
			+ 2\gamma \abs{\nabla\times\omega}^2
	\end{equation*}
	and
	\begin{equation*}
		B\brac{\bar{\omega}, a}\defeq
		\brac{ 2\brac{\lambda-\nu} + {\brac{\frac{\tor}{2\kappa}}}^2 }
		\int_{\T^3} \bar{\omega}^\perp \cdot a.
	\end{equation*}

	\begin{proof}
		To compute the energy-dissipation relation we take a derivative $ \partial^\alpha $ of \eqref{eq:PDE_compact_not},
		multiply by $Z$, and integrate over the torus.
		Note that due to incompressibility $ \int_{\T^3} \partial^\alpha \Lambda\brac{p}\cdot \partial^\alpha Z = \int_{\T^3} -\brac{\nabla \partial^\alpha p}\cdot \partial^\alpha u = 0$.
		Now we compute $ \int_{\T^3} \widetilde{\Leb} Z \cdot Z$.
		Observe that for $T$ and $M$ as in \eqref{eq:constitutive_equations}, if we write $\tilde{T}$ for the trace-free part of $T$, i.e. $\tilde{T} = T+pI$, then we have that
		\begin{align}
			\int_{\T^3} \brac{ \brac{\mu + \kappa/2} \Delta u + \kappa\nabla\times\omega } \cdot u
			+ \int_{\T^3} \brac{
				\kappa\nabla\times u - 2\kappa\omega + \brac{\alpha + \beta/3 - \gamma } \nabla\brac{\nabla\cdot\omega} + \brac{\beta + \gamma} \Delta \omega
			}\cdot \omega
		\nonumber\\
			= \int_{\T^3} \brac{\nabla\cdot\tilde{T}}\cdot u + \brac{2\vc\tilde{T} + \nabla\cdot M}\cdot\omega
			= - \int_{\T^3} \tilde{T} : \brac{\nabla u - \Omega} + M:\nabla\omega
			= - \ds \brac{u,\omega}.
		\label{eq:deriv_dissipation_term}
		\end{align}
		Moreover, we may compute
		\begin{equation*}
			\omega_{eq} \times J \omega_{eq} = {\brac{\frac{\tor}{2\kappa}}}^2 \tilde{a}^\perp,\text{ }\\
			\omega_{eq} \times J_{eq} \omega = \frac{\lambda\tor}{2\kappa} \tilde{\omega}^\perp,\text{ and }\\
			\sbrac{\Omega,J_{eq}} = \brac{\lambda-\nu} \begin{pmatrix} 0 & 0 & -\omega_2 \\ 0 & 0 & \omega_1 \\ \omega_2 & -\omega_1 & 0 \end{pmatrix}
		\end{equation*}
		such that
		\begin{equation}
		\label{eq:deriv_unstable_term}
			\int_{\T^3} - \brac{
				\omega \times J_{eq} \omega_{eq}
				+ \omega_{eq} \times J \omega_{eq}
				+ \omega_{eq} \times J_{eq} \omega
			} \cdot \omega + \int_{\T^3} \brac{
				\sbrac{\Omega_{eq},J}
				+ \sbrac{\Omega, J_{eq}}
			} : J
			= B\brac{\bar{\omega},a}
		\end{equation}
		where we have used that $\sbrac{\Omega_{eq},J}:J = 0$ (c.f. \fref{Lemma}{lemma:antisymmetry_commutator_on_space_symmetric_matrices}).
		Combining \eqref{eq:deriv_dissipation_term} and \eqref{eq:deriv_unstable_term}, we obtain that $ \int_{\T^3} \widetilde{\Leb} Z \cdot Z = -\ds\brac{u,\omega} + B\brac{\bar{\omega},a}$,
		and hence we may conclude that
		\begin{align*}
			\frac{1}{2} \Dt \norm{\sqrt{D}\brac{ \partial^\alpha Z}}{L^2}^2
			= \int_{\T^3} \pdt\brac{D \partial^\alpha Z} \cdot \partial^\alpha Z
			= \int_{\T^3} \widetilde{\Leb} \,\partial^\alpha Z \cdot \partial^\alpha Z
			+ \int_{\T^3} \partial^\alpha N\brac{Z} \cdot \partial^\alpha Z
		\\
			= - \ds\brac{ \partial^\alpha u, \partial^\alpha \omega}
			+ B\brac{ \partial^\alpha \brac{\omega}, \partial^\alpha a}
			+ \int_{\T^3} \partial^\alpha N\brac{Z} \cdot \partial^\alpha Z.
		\end{align*}
	\end{proof}

	Besides the interaction term $ \int_{\T^3} \partial^\alpha N\brac{Z} \cdot \partial^\alpha Z$,
	the only term appearing in the energy-dissipation relation which does not have a sign is the term $B\brac{ \partial^\alpha\bar{\omega}, \partial^\alpha a}$.
	We refer to this term as the \emph{unstable term} since, as detailed in \fref{Section}{sec:heuristics_instability} the instability originates from $\bar{\omega}$ and $a$.
	In \fref{Lemma}{lemma:ibp_unstable_term} below we estimate this term
	in a manner which allows us to absorb a high-order contribution into the dissipation and leaves us with a lower-order term which is controlled by the energy.

	\begin{lemma}[Bounds on the unstable term]
	\label{lemma:ibp_unstable_term}
		For any $\sigma > 0$ there exists $C_\sigma > 0$ such that for any sufficiently regular $\brac{\omega, a}$ and any nonzero multi-index $\alpha$,
		\begin{equation*}
			\vbrac{B\brac{ \partial^\alpha \bar{\omega}, \partial^\alpha a}} \leqslant \sigma \norm{ \partial^{\alpha + 1} \bar{\omega}}{L^2}^2 + C_\sigma \norm{ \partial^{\alpha-1} a}{L^2}^2
		\end{equation*}
		where we write $\alpha\pm 1 \defeq \alpha\pm e_i$ for some $i$ such that $\alpha_i$ nonzero.
	\end{lemma}

	\begin{proof}
		This follows immediately from integrating by parts and applying an $\varepsilon$-Cauchy inequality:
		if we define $C\defeq 2\brac{\lambda-\nu} + {\brac{\frac{\tor}{2\kappa}}}^2$ then, for any $\varepsilon > 0$,
		\begin{equation*}
			\vbrac{B\brac{ \partial^\alpha \bar{\omega}, \partial^\alpha a}}
			= C \vbrac{ \int_{\T^3} \partial^{\alpha+1} \bar{\omega}^\perp \cdot \partial^{\alpha-1}a}
			\leqslant \varepsilon \norm{ \partial^{\alpha+1} \bar{\omega}^\perp}{L^2}^2 + \frac{C^2}{4\varepsilon} \norm{ \partial^{\alpha-1} a}{L^2}^2.
		\end{equation*}
	\end{proof}

	\noindent
	We now prove that the dissipation is coercive, since the velocity $u$ has average zero.

	\begin{lemma}[Coercivity of the dissipation over linear velocities of average zero]
	\label{lemma:coer_diss_vel_avg_zero}
		There exists a constant $C_D>0$ such that for every $\brac{u,\omega}\in H^1\brac{\T^3,\,\R^3 \times \R^3}$, if $\fint u = 0$ then
		$
			\ds\brac{u,\omega} \geqslant C_D\brac{ \norm{u}{H^1}^2 + \norm{\omega}{H^1}^2}.
		$
	\end{lemma}

	\begin{proof}
		Since $u$ has average zero, it follows from \fref{Propositions}{prop:Korn} and \ref{prop:Korn_Poincare} that
		\begin{equation}
			\label{eq:coerc_u}
			\norm{u}{H^1}^2 \lesssim \norm{\symgrad u}{L^2}^2 \lesssim D\brac{u,\omega}.
		\end{equation}
		To see that the dissipation also controls the $H^1$ norm of $\omega$ we observe that, by \eqref{eq:coerc_u},
		\begin{equation*}
			\norm{\omega}{L^2}^2
			\lesssim \int_{\T^3} {\vbrac{ \frac{1}{2} \nabla\times u - \omega }}^2 + \int_{\T^3} {\vbrac{ \frac{1}{2} \nabla\times u}}^2
			\lesssim D\brac{u,\omega} + \norm{u}{H^1}^2
			\lesssim D\brac{u,\omega}
		\end{equation*}
		whilst, by \fref{Lemma}{lemma:div_curl_identity}, $\norm{\nabla\omega}{L^2}^2 = \int_{\T^3} \abs{\nabla\cdot\omega}^2 + \int_{\T^3} \abs{\nabla\times\omega}^2 \lesssim D\brac{u,\omega}$,
		such that indeed $\norm{\omega}{H^1}^2~\lesssim~D\brac{u,\omega}$.
	\end{proof}

	Recall that, as detailed in \fref{Section}{sec:micropolar}, due to the Galilean equivariance of \eqref{eq:stat_prob_lin_mom}--\eqref{eq:stat_prob_microinertia} solutions of that system
	can be assumed without loss of generality to have an Eulerian velocity with average zero.
	Since $u_{eq} = 0$ it follows that we can assume that the perturbative velocity $u$ has average zero as well,
	and hence the coercivity result proven in \fref{Lemma}{lemma:coer_diss_vel_avg_zero} applies.

\subsection{Estimating the interactions}
\label{sec:est_interactions}

	In this section we introduce notation which makes it easier to write down the Fa\`{a} di Bruno formula for the chain rule,
	use this notation to record useful bounds on $m$ (defined in \fref{Definition}{def:m_and_its_domain}),
	and finally we estimate the interactions arising from the energy-dissipation relations satisfied by derivatives of solutions to \eqref{eq:stat_prob_lin_mom}--\eqref{eq:stat_prob_microinertia}
	in \fref{Proposition}{prop:est_interactions}.

	\begin{definition}[Integer partitions and derivatives]
	\label{def:int_part_deriv}
		Let $k\in\N$.
		\begin{itemize}
			\item	Let $i_1 \geqslant i_2 \geqslant \dots \geqslant i_l \geqslant 1$ be integers such that $k = i_1 + i_2 + \dots + i_l$.
				The sequence $\brac{ i_1, i_2, \dots, i_l}$ is called an \emph{integer partition} of $k$ and $l$ is referred to as the \emph{size} of that partition.
			\item	For $1\leqslant i \leqslant k$ we denote by $P_i \brac{k}$ the set of integer partitions of $k$ of size $i$,
				and by $P\brac{k}$ the set of integer partitions of $k$. In particular note that $P\brac{k} = \coprod_{i=1}^k P_i \brac{k}$.
			\item	Let $f:\R^n\to\R^m$ be $k$-times differentiable. For any $\pi = \brac{i_1,\,\dots,\,i_l} \in P\brac{k}$ (where possibly $i_p = i_q$ for $p\neq q$) we define
				$
					\nabla^\pi f \defeq \sym \brac{ \nabla^{i_1} f \otimes \dots \otimes \nabla^{i_l} f}
				$
				where for any tensor $T$ of rank $r$,
				$
					{\brac{ \sym T }}_{j_1 \,\dots\, j_r} \defeq \frac{1}{r!} \sum_{\sigma\in S_r} T_{j_{\sigma\brac{1}} \,\dots\, j_{\sigma\brac{r}}}.
				$
		\end{itemize}
	\end{definition}
	\begin{example} Examples of integer partitions and derivatives indexed by integer partitions are
		\begin{itemize}
			\item	$P_2 (4) = \cbrac{ \brac{3,1}, \brac{2,2} }$,
			\item	$P (4) = \cbrac{ \brac{4}, \brac{3,1}, \brac{2,2}, \brac{2,1,1}, \brac{1,1,1,1} }$, and
			\item	$\nabla^{\brac{2,1,1}} f = \sym \brac{\nabla^2 f \otimes \nabla f \otimes \nabla f}
					= \sym \brac{\nabla f \otimes \nabla^2 f \otimes \nabla f}
					= \sym \brac{ \nabla f \otimes \nabla f \otimes \nabla^2 f}
				$.
		\end{itemize}
	\end{example}
	\begin{remark}
		Derivatives indexed by integer partitions, denoted by $\nabla^\pi f$, are a convenient shorthand for terms appearing in the Fa\`{a} di Bruno formula for derivatives of compositions.
		Their key property which we will use in estimates is that, for any integer partition $\pi = \brac{i_1,\,\dots,\,i_l}$,
		$\vbrac{\nabla^\pi f} \leqslant \prod_{j=1}^l \vbrac{\nabla^{i_j} f}$. For example $\vbrac{\nabla^{\brac{2,1,1}} f} \leqslant \vbrac{\nabla^2 f} {\vbrac{\nabla f}}^2$.
	\end{remark}

	Having introduced notation for derivatives indexed by integer partitions we now use it to obtain bounds on derivatives of $m$ in \fref{Lemma}{lemma:bounds_on_nabla_k_m} below.

	\begin{lemma}[Bounds on derivatives of $m$]
	\label{lemma:bounds_on_nabla_k_m}
		The function $m$ from \fref{Definition}{def:m_and_its_domain} is smooth and moreover for every $k\in\N$ there exists $C_k > 0$ such that, for every $A\in \mathfrak{B}$,
		$
			\vbrac{\nabla^k m \brac{A}} \leqslant C_k \abs{m \brac{A}}^{k+1}
		$.
	\end{lemma}

	\begin{proof}
		First we observe that it suffices to show that, for $\partial_{ij} m \defeq \frac{\partial m\brac{A}}{\partial A_{ij}}$,
		\begin{equation}
		\label{eq:deriv_m_indices}
			\partial_{ij} m_{kl} = - m_{ki} m_{jl},
		\end{equation}
			To prove that \eqref{eq:deriv_m_indices} holds, note that for any smooth $A : \brac{-1,1} \to \mathfrak{B}$ (where $\mathfrak{B}$ is as in \fref{Definition}{def:m_and_its_domain}),
			$\Dt m\brac{A\brac{t}} = - m\brac{A\brac{t}} \brac{\Dt A\brac{t}} m\brac{A\brac{t}}$.
		Since we can pick $A$ such that $A\brac{0}$ and $\Dt A\brac{0}$ are \emph{arbitrarily} specified, it follows that for any $A_0\in\mathfrak{B}$ and any $V\in\R^{n \times n}$,
		$
			\nabla m \brac{A_0} V = - m \brac{A_0} V m \brac{A}
		$,
		i.e. indeed
		$
			\partial_{ij} m_{kl}
			= \nabla m_{kl} \brac{e_i \otimes e_j}
			= - {\brac{ m \brac{e_i \otimes e_j} m }}_{kl}
			= - m_{ki} m_{jl}
		$.
	\end{proof}

	\noindent
	We now use the bounds on $m$ we have just obtained to derive bounds on post-compositions with $m$.

	\begin{lemma}[Bounds on derivatives of post-compositions with $m$]
	\label{lemma:bounds_on_m_post_composed}
		Let $0 < \delta  < 1$ and
		consider $m$ from \fref{Definition}{def:m_and_its_domain}, which is smooth by \fref{Lemma}{lemma:bounds_on_nabla_k_m}.
		For every $k\in\N$ there exists $C_{k,\delta} > 0$ such that for every
		smooth $A: \T^n \to \R^{n \times n}$, if $\norm{A}{\infty} < \delta$ then, for every $x\in\T^n$,
		$
			\vbrac{\nabla^k \brac{m \brac{A}} \brac{x} } \leqslant C_{k,\delta} \sum_{\pi\in P\brac{k}}
			\vbrac{\nabla^\pi A\brac{x}}
		$,
		where $P\brac{k}$ and $\nabla^\pi$ are defined in \fref{Notation}{def:int_part_deriv}.
	\end{lemma}

	\begin{proof}
		Note that since $\norm{A}{\infty} < \delta < 1$ it follows from \fref{Corollary}{cor:invertibility_pert_identity} that $\norm{m\brac{A}}{\infty} < \frac{1}{1-\delta}$.
		Therefore, by \fref{Proposition}{prop:est_faa_di_bruno} and \fref{Lemma}{lemma:bounds_on_nabla_k_m},
		\begin{align*}
			\abs{\nabla^k \brac{m\brac{A}} \brac{x}}
			\leqslant C \sum_{i=1}^k \abs{\nabla^i m\brac{A\brac{x}}} \sum_{\pi\in P_i\brac{k}} \abs{\nabla^\pi A\brac{x}}
			\leqslant C \sum_{i=1}^k \abs{m\brac{A\brac{x}}}^{i+1} \sum_{\pi\in P_i\brac{k}} \abs{\nabla^\pi A\brac{x}}
			\\
			\leqslant C_{k,\delta} \sum_{\pi\in P\brac{k}} \abs{\nabla^\pi A\brac{x}}.
		\end{align*}
	\end{proof}

	\noindent
	Below we specialize \fref{Lemma}{lemma:bounds_on_m_post_composed} to the only case which matters for us, namely the case of $m\brac{JJ_{eq}\inv}$.

	\begin{cor}
	\label{cor:bounds_on_m_J_Jeq_inv}
		Let $\delta_0$ be as in the small energy regime (c.f. \fref{Definition}{def:small_energy_regime}).
		For every $k\in\N$ there exists $C_k > 0$ such that if $\norm{Z}{H^4}\leqslant \delta_0$ then
		$
			\vbrac{\nabla^k \brac{m \brac{JJ_{eq}\inv}} \brac{x} } \leqslant C_k \sum_{\pi\in P\brac{k}}
			\vbrac{\nabla^\pi J\brac{x}}
		$,
		for almost every $x\in\T^3$.
	\end{cor}

	\begin{proof}
		This follow immediately from combining \fref{Lemmas}{lemma:conseq_small_energy_regime} and \ref{lemma:bounds_on_m_post_composed}.
	\end{proof}

	\noindent
	Having obtained good estimates on terms involving $m$ which appear in the nonlinearity we are ready to estimate the interaction terms.

	\begin{prop}[Estimates of the interactions]
	\label{prop:est_interactions}
		Let $\delta_0$ be as in the small energy regime (c.f. \fref{Definition}{def:small_energy_regime}).
		For every $k = 0,1,2,3,4$ there exists $C_{I,\,k} > 0$ such that if $\norm{Z}{H^4} \leqslant \delta_0$ then
		\begin{equation*}
			\vbrac{\int_{\T^3} N\brac{Z}\cdot Z} \leqslant C_{I,\,0} \norm{Z}{H^4} \norm{Z}{L^2}^2
		\end{equation*}
		and
		\begin{equation*}
			\sum_{\abs{\alpha} = k} \vbrac{ \int_{\T^3} \partial^\alpha N\brac{Z} \cdot \partial^\alpha Z}
			\leqslant C_{I,\,k} \norm{Z}{H^4} \brac{
				\sum_{i=1}^k \norm{\nabla^i Z}{L^2}^2
				+ \norm{\nabla^{k+1}\brac{u,\omega}}{L^2}^2
			}.
		\end{equation*}
	\end{prop}

	\begin{proof}
		The nonlinearities are all of one of three types, and so we write $N = N_\RN{1} + N_\RN{2} + N_\RN{3}$ for
		\begin{align*}
			N_\RN{1} &\defeq -\brac{ \brac{u\cdot\nabla}u,\, J_{eq} \brac{u\cdot\nabla}\omega,\, \brac{u\cdot\nabla}J},
			\\
			N_\RN{2} &\defeq \Big(
				0,\, JJ_{eq}\inv m\brac{JJ_{eq}\inv} \brac{
					\omega \times J_{eq} \omega_{eq}
					+ \omega_{eq} \times J \omega_{eq}
					+ \omega_{eq} \times J_{eq} \omega
					+ 2 \kappa \omega
				}
			\\
			&\hspace{1.14cm}
				- m\brac{JJ_{eq}\inv} \brac{
					\omega \times J \omega
					+ \omega \times J \omega_{eq}
					+ \omega \times J_{eq} \omega
					+ \omega_{eq} \times J \omega
				}, \sbrac{\Omega,J}
			\Big), \text{ and }
			\\
			N_\RN{3} &\defeq \brac{
				0,\, -JJ_{eq}\inv m\brac{JJ_{eq}\inv} \brac{
					\kappa\nabla\times u
					+ \tilde{\alpha} \nabla\brac{\nabla\cdot\omega}
					+ \tilde{\gamma}\Delta\omega
				},\,0
			}.
		\end{align*}
		We first consider the case of $\alpha$ nonzero and so
		for $T\in\cbrac{\RN{1},\RN{2},\RN{3}}$ and $i=1,2,3,4$ we write $\calN_{T,i} \defeq \sum_{\abs{\alpha}=i} \int_{\T^3} \partial^\alpha N_T \brac{Z} \cdot \partial^\alpha Z$.

		Estimating nonlinearities of type $\RN{1}$ is fairly straightforward. We expand out $ \int_{\T^3} \partial^\alpha N_\RN{1} \brac{Z} \cdot \partial^\alpha Z$
		and use the generalized H\"{o}lder inequality, putting two factors in $L^2$ and putting the remaining factors in $L^\infty$ (thanks to \fref{Lemma}{lemma:small_energy_regime}).
		For example, writing for simplicity $N_\RN{1} \brac{Z} = \brac{u\cdot\nabla}Z$ and considering the case where $ \partial^\alpha = \partial_{ijkl}$,
		one of the terms that appears is $\int_{\T^3} \brac{ \partial_{ijk} u \cdot \nabla} \nabla_l Z \cdot \partial_{ijkl} Z$,
		and it can be estimated in the following way, which is typical of how nonlinear interactions of type $\RN{1}$ are handled:
		\begin{equation*}
			\vbrac{ \int_{\T^3}
				\brac{ \partial_{ijk} u \cdot \nabla} \nabla_l Z \cdot \partial_{ijkl} Z
			}
			\leqslant \norm{\nabla^3 u}{L^2} \norm{\nabla^2 Z}{\infty} \norm{\nabla^4 Z}{L^2}
			\lesssim \norm{Z}{H^4} \brac{ \norm{\nabla^3 Z}{L^2}^2 + \norm{\nabla^4 Z}{L^2}^2 }.
		\end{equation*}
		The only subtlety for these nonlinear terms is the fact that when $ \partial^\alpha $ hits $\nabla Z$ in $\brac{u\cdot\nabla}Z$,
		the interaction vanishes due to the incompressibility constraint. Indeed, for any multi-index $\alpha$,
		\begin{equation*}
			\int_{\T^3} \brac{u\cdot\nabla} \partial^\alpha Z \cdot \partial^\alpha Z
			= - \frac{1}{2} \int_{\T^3} \brac{\nabla\cdot u} \abs{ \partial^\alpha Z }^2
			= 0.
		\end{equation*}
		This cancellation is essential since we have no dissipative control of $J$ and hence we would not be able to control interactions involving $\nabla \partial^\alpha J$
		(which is a component of $ \nabla \partial^\alpha Z$).
		Estimating all the nonlinearities of type $\RN{1}$ in this manner we obtain:
		\begin{align*}
			&\abs{ \calN_{\RN{1},1} } \lesssim \norm{Z}{H^4} \norm{\nabla Z}{L^2}^2,
			&&\abs{ \calN_{\RN{1},3} } \lesssim \norm{Z}{H^4} \brac{
				\norm{\nabla^3 Z}{L^2}^2
				+ \norm{\nabla^2 Z}{L^2}^2
			},\text{ and }
			\\
			&\abs{ \calN_{\RN{1},2} } \lesssim \norm{Z}{H^4} \norm{\nabla^2 Z}{L^2}^2,
			&&\abs{ \calN_{\RN{1},4} } \lesssim \norm{Z}{H^4} \brac{
				\norm{\nabla^4 Z}{L^2}^2
				+ \norm{\nabla^3 Z}{L^2}^2
			}.
		\end{align*}

		To estimate nonlinearities of type $\RN{2}$ we proceed similarly, namely applying the generalized H\"{o}lder inequality with two factors in $L^2$ and the rest in $L^\infty$.
		In particular we use \fref{Lemma}{lemma:conseq_small_energy_regime} and \fref{Corollary}{cor:bounds_on_m_J_Jeq_inv} to control $m\brac{JJ_{eq}\inv}$ and its derivatives,
		as well as the second part of \fref{Lemma}{lemma:small_energy_regime} for the terms appearing when applying \fref{Corollary}{cor:bounds_on_m_J_Jeq_inv} which are cubic or higher-order.
		As an illustrative example let us write the nonlinearities of type $\RN{2}$ as $N_\RN{2} \brac{Z} = m\brac{J} b\brac{Z,Z}$ for some bilinear form $b$ and consider
		$ \int_{\T^3} \partial_{ijk} \brac{m\brac{JJ_{eq}\inv}} b\brac{ \partial_l Z, Z} \cdot \partial_{ijkl} Z$.
		This terms appears when $ \partial^\alpha = \partial_{ijkl}$ and can be estimated as follows:
		\begin{align*}
			\vbrac{\int_{\T^3} \partial_{ijk} \brac{m\brac{JJ_{eq}\inv}} b\brac{ \partial_l Z, Z} \cdot \partial_{ijkl} Z}
			&\lesssim \int_{\T^3} \brac{\abs{\nabla^3 J} + \abs{\nabla^2 J} \abs{\nabla J} + \abs{\nabla J}^3} \abs{\nabla Z} \abs{Z} \abs{\nabla^4 Z}
			\\
			&\hspace{-3.8cm}\lesssim \brac{
				\norm{\nabla^3 Z}{L^2}
				+ \norm{Z}{H^4} \norm{\nabla^2 Z}{L^2}
				+ \norm{Z}{H^4}^2 \norm{\nabla Z}{L^2}
			} \norm{Z}{H^4}^2 \norm{\nabla^4 Z}{L^2}
			\\
			&\hspace{-3.8cm}\lesssim \norm{Z}{H^4} \brac{
				\norm{\nabla^3 Z}{L^2}
				+ \norm{\nabla^2 Z}{L^2}
				+ \norm{\nabla Z}{L^2}
			} \norm{\nabla^4 Z}{L^2}
			\\
			&\hspace{-3.8cm}\lesssim \norm{Z}{H^4} \brac{
				\norm{\nabla^4 Z}{L^2}^2
				+ \norm{\nabla^3 Z}{L^2}^2
				+ \norm{\nabla^2 Z}{L^2}^2
				+ \norm{\nabla Z}{L^2}^2
			}.
		\end{align*}
		Estimating all terms of type $\RN{2}$ in this fashion yields, for $i=1,2,3,4$,
		$
			\abs{ \calN_{\RN{2},i} } \lesssim \norm{Z}{H^4} \sum_{j=1}^i \norm{\nabla^j Z}{L^2}^2.
		$

		Nonlinearities of type $\RN{3}$ are the most delicate to estimate due to the presence of $\nabla\times u$, $\nabla\brac{\nabla\cdot \omega}$, and $\Delta\omega$.
		The presence of these terms causes two difficulties
		\begin{enumerate}
			\item	when $ \partial^\alpha $ hits $\Delta\omega$ (or $\nabla\brac{\nabla\cdot\omega}$) we must integrate by parts since we do not have any control,
				even through the dissipation, on $\nabla^{\abs{\alpha}+2} \omega$, and
			\item	there are precisely two terms in which more than two derivatives of order three or above appear, terms for which we cannot simply use $L^2$ and $L^\infty$
				in the right-hand side of the generalized H\"{o}lder inequality. This is easily remedied by more carefully choosing the $L^p$ spaces used, which is done explicitly below.
		\end{enumerate}
		Let us write the nonlinearity schematically as $N_\RN{3} \brac{Z} = m\brac{JJ_{eq}\inv} b\brac{Z,\nabla^2\omega}$ for some bilinear form $b$.
		Here is how we handle (1) discussed above: for any multi-index $\alpha$
		\begin{align*}
			\calN_\alpha \defeq \int_{\T^3} m\brac{JJ_{eq}\inv} b \brac{Z,\Delta \partial^\alpha \omega} \cdot \partial^\alpha \omega
			= - \int_{\T^3} \partial_i \brac{m\brac{JJ_{eq}\inv}} b\brac{Z, \partial_i \partial^\alpha \omega} \cdot \partial^\alpha \omega
			\\
			- \int_{\T^3} m\brac{JJ_{eq}\inv} b\brac{\partial_i Z, \partial_i \partial^\alpha \omega} \cdot \partial^\alpha \omega
			- \int_{\T^3} m\brac{JJ_{eq}\inv} b\brac{Z, \partial_i \partial^\alpha \omega} \cdot \partial_i \partial^\alpha \omega
		\end{align*}
		and hence
		\begin{align*}
			\vbrac{\calN_\alpha}
			\lesssim \brac{
				\norm{\nabla Z}{\infty} \norm{Z}{\infty}
				+ \norm{\nabla Z}{\infty}
			} \norm{\nabla^{\abs{\alpha}+1}\omega}{L^2} \norm{\nabla^{\abs{\alpha}} \omega}{L^2}
			+ \norm{Z}{\infty} \norm{\nabla^{\abs{\alpha}+1}\omega}{L^2}
			\\
			\lesssim \norm{Z}{H^4} \brac{
				\norm{\nabla^{\abs{\alpha}+1} \omega}{L^2}^2
				+ \norm{\nabla^{\abs{\alpha}+1} \omega}{L^2}^2
			}.
		\end{align*}
		Now we show how to handle (2) discussed above. Both terms under consideration appear when $\abs{\alpha} = 4$, and so we write $ \partial^\alpha = \partial_{ijkl}$.
		Note that we will use \fref{Corollary}{cor:bounds_on_m_J_Jeq_inv} to bound $\abs{m\brac{JJ_{eq}\inv}}$ above by $\abs{\nabla^3 J} + \abs{\nabla^2 J}\abs{\nabla J} + \abs{\nabla J}^3$,
		but below we will only indicate how to deal with the first one amongst these three terms
		(since the last two can be taken care of by a generalized H\"{o}lder inequality using only $L^2$ and $L^\infty$).
		We have, using the fact that $H^1\brac{\T^3} \hookrightarrow L^4\brac{\T^3}$,
		\begin{align*}
			\vbrac{
				\int_{\T^3} m\brac{JJ_{eq}\inv} b\brac{\partial_{ijk} Z, \Delta\partial_l\omega} \cdot \partial_{ijkl}\omega
				+ \int_{\T^3} \partial_{ijk} \brac{m \brac{JJ_{eq}\inv}} b\brac{Z,\Delta\partial_l\omega} \cdot \partial_{ijkl}\omega
			}
			\\
			&\hspace{-12cm}\lesssim \int_{\T^3} \abs{\nabla^3 Z} \abs{\nabla^3 \omega} \abs{\nabla^4 \omega}
			+ \norm{Z}{\infty} \int_{\T^3} \abs{\nabla^3 J} \abs{\nabla^3 \omega} \abs{\nabla^4 \omega}
			+ \dots
			\\
			&\hspace{-12cm}\lesssim \int_{\T^3} \abs{\nabla^3 Z} \abs{\nabla^3 \omega} \abs{\nabla^4 \omega}
			\lesssim \norm{\nabla^3 Z}{H^1} \norm{\nabla^3 \omega}{H^1} \norm{\nabla^4\omega}{L^2}
			\\
			&\hspace{-12cm}\lesssim \norm{Z}{H^4} \brac{ \norm{\nabla^3\omega}{L^2} + \norm{\nabla^4\omega}{L^2} } \norm{\nabla^4 \omega}{L^2}
			\lesssim \norm{Z}{H^4} \brac{ \norm{\nabla^4\omega}{L^2} + \norm{\nabla^3\omega}{L^2} }.
		\end{align*}

		Estimating all nonlinearities of type $\RN{3}$ in this fashion yields, for $i=1,2,3,4$,
		\begin{equation*}
			\abs{ \calN_{\RN{3},i}} \lesssim \norm{Z}{H^4} \brac{
				\sum_{j=1}^i \norm{\nabla^j Z}{L^2}^2 + \norm{\nabla^{i+1} \brac{u,\omega}}{L^2}^2
			}.
		\end{equation*}

		Finally we consider the case $\alpha = 0$.
		Using the fact that $ \int_{\T^3} \brac{u\cdot\nabla}Z\cdot Z = 0$ and that $\sbrac{\Omega, J}:J=0$ (see \fref{Lemma}{lemma:antisymmetry_commutator_on_space_symmetric_matrices})
		we see that
		\begin{align*}
			\int_{\T^3} N\brac{Z}\cdot Z
			= - \int_{\T^3}
				m\brac{JJ_{eq}\inv} \brac{
					\omega \times J \omega
					+ \omega \times J \omega_{eq}
					+ \omega \times J_{eq} \omega
					+ \omega_{eq} \times J \omega
				} \cdot\omega
			\\
			+ \int_{\T^3}
				JJ_{eq}\inv m\brac{JJ_{eq}\inv} \brac{
					\omega \times J_{eq} \omega_{eq}
					+ \omega_{eq} \times J \omega_{eq}
					+ \omega_{eq} \times J_{eq} \omega
					+ 2 \kappa \omega
					- \kappa\nabla\times u
					- \tilde{\alpha} \nabla\brac{\nabla\cdot\omega}
					- \tilde{\gamma}\Delta\omega
				}\cdot\omega.
		\end{align*}
		It thus follows from \fref{Lemmas}{lemma:small_energy_regime} and \ref{lemma:conseq_small_energy_regime} that
		$
			\vbrac{ \int_{\T^3} N\brac{Z}\cdot Z} \lesssim \norm{Z}{H^4} \norm{Z}{L^2}^2.
		$
	\end{proof}

\subsection{The chain of energy inequalities}
\label{sec:chain_energy_ineq}

	We begin this section by combining the results of \fref{Sections}{sec:ED_identities} and \ref{sec:est_interactions} in order to obtain a chain of energy inequalities.

	\begin{prop}[Chain of energy inequalities]
	\label{prop:chain_energy_ineq}
		There exist $C_0, C_1, C_D > 0$ such that for every $0 < \varepsilon < 1$ there exists $0 < \delta\brac{\varepsilon} < 1$ such that
		if $\sup_{0 \leqslant t \leqslant T} \norm{Z\brac{t}}{H^4} \leqslant \delta\brac{\varepsilon}$
		and $Z$ solves \eqref{eq:PDE_compact_not} then
		\begin{equation*}
			\frac{1}{2} \Dt \norm{\sqrt{D}Z}{L^2}^2 + \ds\brac{u,\omega} \leqslant \varepsilon \norm{Z}{L^2}^2 + C_0 \brac{\norm{\bar{\omega}}{L^2}^2 + \norm{a}{L^2}^2}
		\end{equation*}
		and, for $k=1,2,3,4$,
		\begin{equation*}
			\frac{1}{2} \Dt \norm{\nabla^k \brac{\sqrt{D}Z}}{L^2}^2 + \frac{C_D}{2} \norm{\nabla^k \brac{u, \omega}}{H^1}^2
			\leqslant \varepsilon \norm{\nabla^k Z}{L^2}^2 + C_1 \sum_{i=0}^{k-1} \norm{\nabla^i Z}{L^2}^2.
		\end{equation*}
	\end{prop}

	\begin{proof}
		Let $\varepsilon > 0$, let $C_D$ and $C_{I,\,k}$ be as in \fref{Lemma}{lemma:coer_diss_vel_avg_zero} and \fref{Proposition}{prop:est_interactions} respectively,
		let $C_\sigma$ be as in \fref{Lemma}{lemma:ibp_unstable_term} for $\sigma\defeq\frac{C_D}{4}$,
		let $n_k \defeq \# \cbrac{\text{multi-index }\alpha : \abs{\alpha} = k}$, and pick
		$\delta \defeq \min\limits_{0\leqslant k \leqslant 4} \cbrac{ \delta_0, \frac{\varepsilon}{C_{I,k} n_k}, \frac{C_D}{4 C_{I,k}} }$.

		First we consider $k=0$. Observe that for $2C_0 \defeq 2\brac{\lambda-\nu} + {\brac{\frac{\tor}{2\kappa}}}^2$,
		\begin{equation}
		\label{eq:unstable_term_CS}
			B\brac{\bar{\omega},a} = 2C_0 \int_{\T^3} \bar{\omega}^\perp \cdot a \leqslant C_0 \brac{ \norm{\bar{\omega}}{L^2}^2 + \norm{a}{L^2}^2 }.
		\end{equation}
		By \fref{Propositions}{prop:ed_rel}, \ref{prop:est_interactions}, \eqref{eq:unstable_term_CS}, and the fact that $\delta \leqslant \frac{\varepsilon}{C_{I,0}}$
		we deduce the energy inequality for $k=0$.

		Now we consider $k=1,2,3,4$. For any nonzero multi-index $\alpha$ it follows from \fref{Propositions}{prop:ed_rel} and \ref{prop:est_interactions} and from
		\fref{Lemmas}{lemma:coer_diss_vel_avg_zero} and \ref{lemma:ibp_unstable_term} that
		\begin{align*}
			\frac{1}{2} \Dt \norm{\sqrt{D}\brac{ \partial^\alpha Z}}{L^2}^2
			+ C_D \norm{ \partial^\alpha \brac{u,\omega} }{H^1}^2
			\leqslant \brac{ \frac{C_D}{4} \norm{ \partial^{\alpha+1} \omega}{L^2}^2 + C_\sigma \norm{ \partial^{\alpha-1} a}{L^2}^2 }
			\\
			+ C_{I,k} \norm{Z}{H^4} \brac{ \sum_{i=1}^k \norm{\nabla^i Z}{L^2}^2 + \norm{\nabla^{k+1}\brac{u,\omega}}{L^2}^2}.
		\end{align*}
		Summing over $\abs{\alpha} = k$ and using that $\delta \leqslant \min \brac{\frac{C_D}{4 C_{I,k}}, \frac{\varepsilon}{C_{I,k}n_k}}$
		we observe that, after absorbing $\norm{\partial^{\alpha+1} \omega}{L^2}^2$ and $\norm{\nabla^{k+1}\brac{u,\omega}}{L^2}^2$ into the dissipation on the left-hand side,
		\begin{equation*}
			\frac{1}{2} \Dt\norm{\nabla^k \brac{\sqrt{D}Z}}{L^2}^2
			+ \frac{C_D}{2} \norm{\nabla^k\brac{u,\omega}}{H^1}^2
			\leqslant n_k C_\sigma \norm{\partial^{k-1} a}{L^2}^2 + \varepsilon \sum_{i=1}^k \norm{\nabla^i Z}{L^2}^2
		\end{equation*}
		from which the result follows upon taking $C_1 \defeq \max\brac{1, n_4 C_\sigma}$.
	\end{proof}

	\noindent
	We now record, in abstract form, a Gronwall-type lemma for chains of differential inequalities.

	\begin{lemma}[Chain of Gronwall inequalities]
	\label{lemma:chain_gronwall_ineq}
		Consider, for $k\geqslant -1$, $E_k : \cobrac{0,\infty} \to \cobrac{0,\infty}$.
		Suppose that there exists $C_{-1}, C > 0$, $0 < \theta \leqslant \theta_0 < \psi$, and $k_\text{max} \geqslant -1$ such that for every $t\geqslant 0$,
		$E_{-1}\brac{t} \leqslant C_{-1} e^{\psi t}$ and every $k\geqslant 0$,
		\begin{equation}
		\label{eq:chain_gronwall_diff_ineq}
			\Dt E_k\brac{t} \leqslant \theta E_k \brac{t} + C \sum_{i=-1}^{k-1} E_i \brac{t}.
		\end{equation}
		Then, for every $0 \leqslant k \leqslant k_\text{max}$, there exist $C_k > 0$ such that for every $t\geqslant 0$
		\begin{equation}
		\label{eq:chain_gronwall_induction_hyp}
			E_k\brac{t} \leqslant C_k \brac{C_{-1} + \sum_{i=0}^k E_i \brac{0}} e^{\psi t} \eqdef \widetilde{C}_k e^{\psi t}.
		\end{equation}
		Moreover: if \eqref{eq:chain_gronwall_diff_ineq} holds for every $k\geqslant -1$ then so does \eqref{eq:chain_gronwall_induction_hyp}.
	\end{lemma}

	\begin{proof}
		We induct on $k$, noting that the base case $k=-1$ holds by assumption. Now suppose that \eqref{eq:chain_gronwall_induction_hyp} holds for every $i=-1,\,\dots,\,k-1$.
		Then, by \eqref{eq:chain_gronwall_diff_ineq},
		$
			\Dt \brac{ E_k\brac{t} e^{-\theta t}} \leqslant C \sum_{i=-1}^{k-1} E_i\brac{t} e^{-\theta t}
		$
		and hence, integrating in time and using \eqref{eq:chain_gronwall_induction_hyp}, where $\widetilde{C}_{-1} \defeq C_{-1}$,
		\begin{align*}
			E_k\!\brac{t} \!\leqslant\! E_k\!\brac{0}e^{\theta t} + C \sum_{i=-1}^{k-1} e^{\theta t}\int_0^t \widetilde{C}_i e^{\brac{\psi-\theta}s} ds
			\!\leqslant\!\brac{ E_k\!\brac{0} + \frac{C}{\psi-\theta_0} \sum_{i=-1}^{k-1} \widetilde{C}_i} e^{\psi t}
			\!\leqslant\! C_k \brac{C_{-1} + \sum_{i=0}^{k} E_i\!\brac{0}} e^{\psi t}.
		\end{align*}
		for some $C_k > 0$.
	\end{proof}

	We conclude this section by applying \fref{Lemma}{lemma:chain_gronwall_ineq} to the chain of differential inequalities obtained in \fref{Proposition}{prop:chain_energy_ineq},
	which yields a bootstrap energy inequality.

	\begin{prop}[Bootstrap energy inequality]
	\label{prop:bootstrap_energy_ineq}
	There exists $0 < \delta_B < 1$ such that if $Z$ solves \eqref{eq:PDE_compact_not} and $\sup_{0 \leqslant t \leqslant T} \norm{Z\brac{t}}{H^4} \leqslant \delta_B$
		then for every $\psi > 0$ there exists $C\brac{\psi} > 0$ such that if there exists $C_{ins} > 0$ such that
		$E_{\text{ins}} \brac{t} \defeq \norm{\bar{\omega}\brac{t}}{L^2}^2 + \norm{a\brac{t}}{L^2}^2$ satisfies
		$E_{\text{ins}} \brac{t} \leqslant C_{\text{ins}} e^{\psi t}$ for all $t > 0$ then, for all $t\geqslant 0$,
		$$\norm{Z\brac{t}}{H^4}^2 \leqslant C\brac{\psi} \brac{\norm{Z\brac{0}}{H^4}^2 + C_{\text{ins}}} e^{\psi t}.$$
	\end{prop}

	\begin{proof}
		Let us define $E_{-1} \defeq E_{\text{ins}}$, $E_k\brac{t} \defeq \norm{\nabla^k \brac{\sqrt{D}Z}}{L^2}^2$ for every $t\geqslant 0$ and every $k\geqslant 0$,
		and $C\defeq \max \brac{C_0,C_1}$ for $C_0$ and $C_1$ as in \fref{Proposition}{prop:chain_energy_ineq}.
		Observe that $\abs{J_{eq}^{1/2}w}^2 \geqslant \nu/2 \abs{w}^2$ for any $w\in\R^3$ and hence
		$\norm{Z}{L^2}^2 \leqslant \max\brac{1,2/\nu}\norm{\sqrt{D}Z}{L^2}^2$.
		Let $\psi > 0$ and note that we may deduce from \fref{Proposition}{prop:chain_energy_ineq},
		picking $\varepsilon  = \frac{1}{2} \min\brac{1, \psi/2, \psi\nu/2}$, $\delta_B \defeq \delta\brac{\varepsilon}$, and neglecting the dissipation,
		that for $k=0,\,\dots,\,4$ and every $t\geqslant 0$
		\begin{equation}
		\label{eq:bootstrap_en_ineq_diff_ineq}
		\Dt E_k\brac{t} \leqslant \frac{\psi}{2} E_k\brac{t} + C\sum_{i=-1}^{k-1} E_i\brac{t}.
		\end{equation}
		Now suppose that, for every $t\geqslant 0$, $E_{-1}\brac{t} = E_{\text{ins}} \brac{t} \leqslant C_\text{ins} e^{\psi t} \eqdef C_{-1} e^{\psi t}$.
		Using \fref{Lemma}{lemma:chain_gronwall_ineq} we obtain that for $k=0,\,\dots,\,4$
		there exists $C_k > 0$ such that
		$
			E_k \brac{t} \leqslant C_k \brac{ C_{-1} + \sum_{i=0}^k E_i\brac{0}} e^{\psi t}
		$.
		Finally, summing over $k=0,\,\dots,\,4$ we obtain that
		\begin{align*}
			\norm{Z\brac{t}}{H^4}^2
			\leqslant \max\brac{1, 2/\nu} \sum_{k=0}^4 E_k\brac{t}
			&\leqslant \widetilde{C}\brac{\psi} \brac{C_{-1} + \norm{\sqrt{D}Z\brac{0}}{H^4}^2} e^{\psi t}
			\\
			&\leqslant \max\brac{1,\lambda,\nu} \widetilde{C}\brac{\psi} \brac{C_\text{ins} + \norm{Z\brac{0}}{H^4}^2} e^{\psi t}
		\end{align*}
		for some $\widetilde{C}\brac{\psi} > 0$, so we may simply pick $C\brac{\psi} \defeq \max\brac{1,\lambda,\nu} \widetilde{C}\brac{\psi}$.
	\end{proof}


\section{The bootstrap instability argument}
\label{sec:bootstrap}

	In this section we prove our main result using a Guo-Strauss bootstrapping argument.
	This technique was introduced by Guo and Strauss in \cite{guo_strauss_bgk}, inspired by \cite{guo_strauss_double_humped} and \cite{friedlander_strauss_vishik}.
	For a cleanly written and very readable form of the bootstrap instability argument we refer to Lemma 1.1 of \cite{guo_hallstrom_spirn}.

	For the purpose of the theorem below, we define what we mean by a strong solution of \eqref{eq:stat_prob_lin_mom}--\eqref{eq:stat_prob_microinertia}.
	\begin{definition}[Strong solutions]
	\label{def:strong_sol}
	For any $X_0 \in H^2 \brac{\T^3}$ and any $T > 0$
	we define a \emph{strong solution} of \eqref{eq:stat_prob_lin_mom}--\eqref{eq:stat_prob_microinertia} with initial condition $X_0$
	to be any function $X \in L^\infty \brac{ \sbrac{0,T},\, H^2\brac{\T^3}}$ with $\pdt X \in L^\infty \brac{ \sbrac{0,T}, \, L^2\brac{\T^3}}$
	for which \eqref{eq:stat_prob_lin_mom}--\eqref{eq:stat_prob_microinertia} is satisfied almost everywhere in $\brac{0,T}\times\T^3$ and such that $X(0) = X_0$.
	\end{definition}

	\begin{thm}[Bootstrap instability]
	\label{thm:bootstrap_instability}
		Let $\eta_*$ be as in \fref{Proposition}{prop:max_unstable_evals} and assume that \eqref{eq:positivity_requirement_on_viscosity_constants} holds.
		There exists $\theta,\delta > 0$ and $Z_0 \in L^2\brac{\T^3,\,\R^3\times\R^3\times\R^{3\times 3}}$ such that for all $0 < \iota < \delta$
		if we define $T_I \defeq \frac{1}{\eta_*}\log\frac{\theta}{\iota}$ then there exists a strong solution $X=\brac{u,\omega,J}\in L^\infty \brac{ \sbrac{0, T_I},\, H^4\brac{\T^3}}$
		of \eqref{eq:stat_prob_lin_mom}--\eqref{eq:stat_prob_microinertia} with pressure $p\in L^\infty \brac{ \sbrac{0, T_I},\, H^4\brac{\T^3}}$
		and initial condition $X\brac{0} = X_{eq} + \iota Z_0$ such that $\norm{X\brac{T_I}-X_{eq}}{L^2} > \frac{\theta}{2}$.
	\end{thm}
	\begin{proof}
		The crux of the argument is to compare three timescales: the instability timescale $T_I$, the linear-dominance timescale $T_L$, and the smallness timescale $T_S$.
		We will show that at times living in both the linear-dominance and the smallness timescale (i.e. times anterior to both $T_L$ and $T_S$) two key estimates hold,
		namely \eqref{eq:bootstrap_high_energy} and \eqref{eq:bootstrap_diff_nonlin_lin}.
		This will allow us, by way of contradiction, to show that the instability timescale is the shortest of the three.
		It will thus follow that instability occurs while the dynamics are dominated by the linearization and while we are in the small energy regime.

		We begin by recalling appropriate notation from previous results.
		Let $\brac{u_0,\omega_0,a_0} \eqdef Y$ be as in \fref{Proposition}{prop:max_unstable_sol} and note that without loss of generality we may assume that $\norm{Y}{L^2}=1$.
		Define $Z_0 \defeq \brac{u_0, \omega_0, J_0}$
		where $J_0 = \begin{pmatrix} 0_{2\times 2} & a_0 \\ a_0^T & 0 \end{pmatrix}$.
		Let $\delta_0$ be as in the small energy regime (c.f. \fref{Definition}{def:small_energy_regime}),
		let $C_S \defeq C_S\brac{\frac{\eta_*}{2}}$ as in \fref{Proposition}{prop:semigroup_lin_prob},
		let $C_N$ be as in \fref{Proposition}{prop:est_nonlinearity},
		let $\psi \defeq 2\eta_*$ such that $C_B \defeq C\brac{\psi}$ and $\delta_B$ are as in \fref{Proposition}{prop:chain_energy_ineq},
		and let $\delta_\text{lwp}$ be as in \fref{Theorem}{thm:lwp} with $\delta_\text{lwp}$ being chosen small enough so as to ensure that $L \leqslant \eta_*$.
		
		We can now define the appropriate small scales $\theta$ and $\delta$, which in turn will later allow us to precisely define the timescales.
		Let
		\begin{equation*}
			\theta = \frac{1}{2} \min\brac{
				\delta_0,
				\delta_B,
				\frac{1}{C},
				\norm{Z_0}{L^2} {\brac{\frac{\delta_\text{lwp}}{\norm{Z_0}{H^4}}}}^{\eta_*/L}
			},
		\end{equation*}
		\begin{equation*}
			\delta = \frac{1}{2} \min\brac{
				1,
				\delta_\text{lwp},
				\frac{\theta}{\norm{Z_0}{H^4}},
				{\brac{	C_B \brac{\norm{Z_0}{H^4}^2+4} \theta }}^{-1/2},
				\frac{1}{2C\theta}
			},
		\end{equation*}
		and let $0 < \iota < \delta$.

		By our local well-posedness result (see \fref{Section}{sec:lwp} and \fref{Corollary}{cor:lwp} in particular)
		there exists $T_E > 0$ and a unique strong solution
		$Z \in L^\infty \brac{ \sbrac{0, T_E};\, H^4\brac{\T^3}}$ of \eqref{eq:PDE_compact_not} with pressure $p\in L^\infty H^4$ and initial data $Z\brac{0} = \iota Z_0$.
		Note that our local existence result (\fref{Theorem}{thm:lwp}) tells us moreover that the solution $Z$ may be continued
		as long as $Z$ remains in an open $H^4$-ball of radius $\delta_\text{lwp}$. We may thus without loss of generality assume $T_E$ to be the \emph{maximal} time of existence
		of the solution in the sense that
		$T_E \defeq \sup\cbrac{T > 0:Z \text{ exists on } \sbrac{0,T} \text{ and } \sup_{0 \leqslant t < T} \norm{Z\brac{t}}{H^4} < \delta_\text{lwp}}$.
		Expanding out the definition of the notation in \eqref{eq:PDE_compact_not} we see that $X \defeq X_{eq} + Z$ is a strong solution of
		\eqref{eq:stat_prob_lin_mom}--\eqref{eq:stat_prob_microinertia} with initial condition $X\brac{0} = X_{eq} + \iota Z_0$.

		We may now define the timescales.
		We define
		$T_L\defeq \sup\cbrac{0<t<T_E:\norm{\bar{\omega}\brac{t}}{L^2} + \norm{a\brac{t}}{L^2} \leqslant 2\iota e^{\eta_* t}}$,
		$T_I \defeq \frac{1}{\eta_*} \log\frac{\theta}{\iota}$, and
		$T_S \defeq \sup\cbrac{0<t<T_E:\norm{Z\brac{t}}{H^4} < \theta}$.
		Note that $T_L \geqslant 0$ since $\norm{Z\brac{0}}{L^2} = \iota$, that $T_S \geqslant 0$ since $\norm{Z\brac{t}}{H^4} = \iota\norm{Z_0}{H^4} < \theta$,
		and that $T_S, T_L \leqslant T_E$

		\textbf{Step 1:}
		Since $\theta \leqslant \delta_B$ we deduce from \fref{Proposition}{prop:bootstrap_energy_ineq} that if $t\leqslant \min\brac{T_L,T_S}$ then
		\begin{equation}
		\label{eq:bootstrap_high_energy}
			\norm{Z\brac{t}}{H^4}^2 \leqslant C_B \iota^2 \brac{\norm{Z_0}{H^4}^2 + 4} e^{2\eta_* t}.
		\end{equation}

		Now we apply the Leray projector to eliminate the pressure and write \eqref{eq:PDE_compact_not}
		in the reduced form $\pdt Z = \Leb Z + \widetilde{N}\brac{Z}$, where $\widetilde{N}\defeq \mathbb{P}D\inv N$.
		More precisely we apply $\mathbb{P}$ and observe that $\bar{\mathbb{P}}\B = \B$ and hence $\mathbb{P}\Leb=\Leb$.
		Indeed this follows from the observation that on one hand, for $k\neq 0$,
		$
			\hat{\bar{\mathbb{P}}}\brac{k} = \brac{I - \frac{k\otimes k}{\abs{k}^2}} \oplus I_3 \oplus I_2 = I - \proj_{V_k}
		$
		and the fact that, since $\hat{\B}_k$ acts on $\C^8/\,V_k$, it follows that $\proj_{V_k} \circ \hat{\B}_k = \hat{\B}_k$,
		whilst on the other hand, for $k=0$, we have that $\hat{P}_L\brac{0} = I_3$ (since constant vector fields are divergence-free) and hence $\hat{\mathbb{P}}\brac{0}=\id$.

		We can thus apply the Duhamel formula to obtain
		\begin{equation*}
			Z\brac{t} - e^{t\Leb} Z\brac{0} = \int_0^t e^{\brac{t-s}\Leb} \widetilde{N}\brac{Z\brac{s}} ds
		\end{equation*}
		which can be estimated, when $t\leqslant \min\brac{T_L,T_S}$, using the fact that $\theta \leqslant \delta_0$, \fref{Proposition}{prop:semigroup_lin_prob},
		the fact that the Leray projector is bounded on $L^2$, the inequality $\norm{D\inv}{} \leqslant \sqrt{\max\brac{1,2/\nu}}$, and \fref{Proposition}{prop:est_nonlinearity} to yield,
		for $C \defeq \frac{2}{\eta_*} \max\brac{1,2/\nu} C_S C_N C_B \brac{ \norm{Z_0}{H^4}^2 + 4}$,
		\begin{equation}
		\label{eq:bootstrap_diff_nonlin_lin}
			\norm{Z\brac{t} - e^{t\Leb}\iota Z_0}{L^2} \leqslant C\iota^2 e^{2\eta_* t}.
		\end{equation}
		
		\textbf{Step 2:}
		Now we show that $T_I = \min\brac{T_I,T_L,T_S}$, using the key estimates \eqref{eq:bootstrap_high_energy} and \eqref{eq:bootstrap_diff_nonlin_lin}.
		First suppose for the sake of contradiction that $T_L = \min\brac{T_I,T_L,T_S}$.
		By definition of $T_L$,
		\begin{equation}
		\label{eq:cons_T_L_min}
			\norm{\bar{\omega}\brac{T_L}}{L^2} + \norm{a\brac{T_L}}{L^2} = 2\iota e^{\eta_* T_L}.
		\end{equation}
		Now note that \eqref{eq:bootstrap_diff_nonlin_lin} applies since $T_L \leqslant T_S$ and thus it follows from \fref{Proposition}{prop:max_unstable_sol} and the choice of $Z_0$ that
		$
			\norm{Z\brac{T_L}}{L^2}
			\leqslant \brac{1 + C\iota e^{\eta_* T_L}} \iota e^{\eta_* T_L}
			< 2 \iota e^{\eta_* T_L}
		$,
		where we have used that $T_L \leqslant T_I$ and hence $C\iota e^{\eta_* T_L}\leqslant C\iota e^{\eta_* T_I} = C\theta < 1$.
		This contradicts \eqref{eq:cons_T_L_min} and hence the linear-dominance timescale $T_L$ is not the smallest of the three timescales considered.

		Now suppose for the sake of contradiction that $T_S = \min\brac{T_I,T_L,T_S}$.
		By definition of $T_S$,
		\begin{equation}
		\label{eq:cons_T_S_min}
			\norm{Z\brac{T_S}}{H^4} = \theta.
		\end{equation}
		Since $T_S \leqslant T_L$ we may use \eqref{eq:bootstrap_high_energy} and since $T_S \leqslant T_I$ we have that $e^{2\eta_* T_S} \leqslant e^{2\eta_* T_I} = \theta$.
		Putting these two facts together tells us that $\norm{Z\brac{T_S}}{H^4}^2 \leqslant C_B \iota^2 \brac{\norm{Z_0}{H^4}^2+4} \theta^2 < \theta$ which contradicts \eqref{eq:cons_T_S_min}.
		Therefore the smallness timescale $T_S$ is not the smallest of the three timescales considered.
		We thus deduce that $T_I = \min\brac{T_I,T_L,T_S}$.

		\textbf{Step 3:}
		Finally we show that $\norm{Z\brac{T_I}}{L^2} > \frac{\theta}{2}$.
		Since $T_I$ is smaller than both $T_L$ and $T_S$ (and hence smaller than $T_E$) we may use \eqref{eq:bootstrap_high_energy} and \eqref{eq:bootstrap_diff_nonlin_lin},
		as well as \fref{Proposition}{prop:max_unstable_sol}, the choice of $Z_0$, and the fact that $\iota e^{\eta_* T_I} = \theta$ to see that
		$
			\norm{X\brac{T_I} - X_{eq}}{L^2}
			= \norm{Z\brac{T_I}}{L^2}
			\geqslant \iota e^{\eta_* T_I} - C\iota^2 e^{2\eta_* T_I}
			= \theta\brac{1 - C\iota\theta}
			> \frac{1}{2}.
		$
	\end{proof}

\appendix
\section{Local well-posedness}
\label{sec:lwp}

	In this section we prove the local well-posedness of \eqref{eq:stat_prob_lin_mom}--\eqref{eq:stat_prob_microinertia}.
	This is done in two steps: we prove local existence in the small energy regime in \fref{Theorem}{thm:lwp}
	and we prove uniqueness within a broader class of solutions in \fref{Theorem}{thm:uniqueness}.
	Notably, this uniqueness result makes no smallness assumption and only requires that the unknowns belong to appropriate Sobolev spaces.

	A key step on the way to our local existence result is to prove that the nonlinearity is sufficiently regular.
	We do this below in \fref{Lemma}{lemma:N_is_analytic} where we prove that the nonlinearity is analytic.

	\begin{lemma}[Analyticity of the nonlinearity]
	\label{lemma:N_is_analytic}
		Let $0 < \delta \leqslant \delta_0$ for $\delta_0$ as in the \hyperref[def:small_energy_regime]{small energy regime}.
		For every $s > \frac{3}{2}$, $N:H^{s+2} \cap H^4_{\delta_0} \to H^s$ is analytic (as a mapping from $H^{s+2}$ to $H^s$).
		Moreover the Lipschitz constant of $N$ on $H^{s+2} \cap H^4_{\delta_0} \to H^s$ approaches zero as $\delta\downarrow 0$.
	\end{lemma}
	\begin{proof}
		The two key observations are that (i) we may write $N\brac{Z} = P\brac{ m\brac{JJ_{eq}\inv}, Z, \nabla Z, \nabla^2 Z}$ for some polynomial $P$
		and that (ii) $m$ is analytic (recall that $m$ is defined in \fref{Definition}{def:m_and_its_domain}).
		Indeed $m$ can be written as a geometric series, namely $m\brac{A} = \sum_{i=0}^\infty {\brac{-1}}^i A^i$ for every $A\in\mathfrak{B}$,
		where $\mathfrak{B}$ is defined in \fref{Definition}{def:m_and_its_domain}.

		Using \fref{Lemma}{lemma:post_comp_analyt_is_analyt}, the fact that $H^s$ is a continuous algebra when $s > \frac{3}{2}$, and the fact that polynomials are analytic,
		it follows that we may write $N = F\brac{\mathcal{J}^2 Z}$ for some function $F: \text{dom}\, F \subseteq H^s \to H^s$ which is analytic on its domain (i.e. where it is well-defined),
		where $\mathcal{J}^2 Z \defeq \brac{Z, \nabla Z, \nabla^2 Z}$. The last observation we need is that $\mathcal{J}^2 \brac{H^s \cap H^4_{\delta_0}} \subseteq \text{dom}\, F$.
		This holds since, if $Z=\brac{u,\omega,J} \in H^{s+2} \cap H^4_{\delta_0}$ for $\delta_0$ as in the \hyperref[def:small_energy_regime]{small energy regime},
		then by \fref{Lemma}{lemma:conseq_small_energy_regime} we know that $J\mapsto m\brac{JJ_{eq}\inv}$ is well-defined, and hence analytic.
		Since $\mathcal{J}$ is a bounded linear map from $H^{s+2}$ to $H^s$ it is also analytic, and so we may conclude that $N : H^{s+2} \cap H^4_{\delta_0} \to H^s$ is analytic
		as a map from $H^{s+2}$ to $H^s$.

		Finally, note that the polynomial $P$ above is at least quadratic in $\brac{Z, \nabla Z, \nabla^2 Z}$ and that therefore $DN\brac{0} = 0$.
		In particular it follows that the Lipschitz constant of $N$ on balls of vanishingly small radii approaches zero, as claimed.
	\end{proof}
	\begin{remark}
		See \cite{whittlesey} for a brief and clean summary of basic results regarding analytic functions between Banach spaces.
	\end{remark}

	\noindent
	With \fref{Lemma}{lemma:N_is_analytic} in hand we may now prove our local existence result.

	\begin{thm}[Local existence and continuous dependence on the data]
	\label{thm:lwp}
		There are universal constants $\rho, \delta_\text{lwp}, C > 0$
		such that for any $Z_0 = \brac{u_0, \omega_0, J_0} \in H^4$ with $\nabla\cdot u_0 = 0$, $\fint_{\T^3} u_0 = 0$, and $\norm{Z_0}{H^4} < \delta_\text{lwp}$,
		there exists a time of existence $T_\text{lwp} > 0$,
		there exists $Z=\brac{u,\omega,J}\in L^\infty H^4$ with $\brac{u,\omega}\in L^2 H^5$, $\pdt Z\in L^\infty H^2 \cap L^2 H^3$, and $\pdt J \in L^\infty H^3$,
		and there exists $p \in L^\infty H^4 \cap L^2 H^5$ with average zero such that $u$ is divergence-free and has average zero, $\brac{u, p, \omega, J}$ solves
		\begin{equation}
		\label{eq:thm_lwp_PDE}
			\pdt DZ = \widetilde{\Leb}Z + \Lambda\brac{p} + N\brac{Z} \text{ a.e. in } \brac{0,T_\text{lwp}} \text{ and }
			Z\brac{0} = Z_0 \text{ in } H^{4-\frac{1}{4}},
		\end{equation}
		and the estimates
		\begin{equation}
		\label{eq:thm_lwp_estimates}
			\norm{Z}{L^\infty H^4} + \norm{\brac{u, \omega}}{L^2 H^5} + \norm{\pdt Z}{L^\infty H^2 \cap L^2 H^3} + \norm{\pdt J}{L^\infty H^3} \leqslant C \norm{Z_0}{H^4}
		\end{equation}
		and
		\begin{equation}
		\label{eq:thm_lwp_estimate_pressure}
			\norm{p}{L^\infty H^4 \cap L^2 H^5} \leqslant C \norm{u}{L^\infty H^4 \cap L^2 H^5}^2.
		\end{equation}
		hold. Moreover we have the lower bound $T_\text{lwp} \geqslant \frac{1}{\rho}\log\frac{\delta_\text{lwp}}{\norm{Z_0}{H^4}}$.
	\end{thm}
	\begin{proof}
		We proceed via a standard Galerkin scheme and thus omit the fine details of the proof here.
		A key point is that everything we need to know about the nonlinearity for the purpose of this local well-posedness result is obtained in \fref{Lemma}{lemma:N_is_analytic}.

		We now proceed in five steps.
		In Step 1 we eliminate the pressure via Leray projection,
		in Step 2 we prove local well-posedness for a sequence of appropriate approximate problems,
		in Step 3 we obtain uniform bounds on these approximate solutions,
		in Step 4 we pass to the limit via a compactness argument,
		and in Step 5 we reconstruct the pressure.

		First we recall some notation from earlier results which is required to define the smallness parameter $\delta_\text{lwp}$.
		Let $\delta_0$ be as in the \hyperref[def:small_energy_regime]{small energy regime},
		let $\delta = \delta\brac{ \frac{1}{2} }$ be as in \fref{Proposition}{prop:chain_energy_ineq},
		and define $C_2 \defeq \max\brac{1, \lambda, \nu} \max\brac{1, 2/\nu}$.
		Then take $\delta_\text{lwp} \defeq \frac{1}{3} \min\brac{\delta_0 / C_2, \delta}$.

		\textbf{Step 1:} Leray projection eliminating the pressure.

			Recall that we denote the Leray projector by $\PP_L$ and that we write $\PP = \PP_L \oplus I_3 \oplus I_{3 \times 3}$.
			Upon applying $\PP$ to \eqref{eq:thm_lwp_PDE} we thus see that
			(noting that $\PP Z = Z$ since $\nabla\cdot u=0$ and that $\PP$ and $\widetilde{\Leb}$ commute since they are both Fourier multipliers):
			$\pdt DZ = \widetilde{\Leb}Z + \PP N\brac{Z}$.

		\textbf{Step 2:} Local well-posedness of a sequence of approximate problems.
		
		Let $V_n \defeq \setdef{
			Z\in L^2\brac{\T^3; \R^3\times\R^3\times\R^{3 \times 3}}
		}{
			\hat{Z}\brac{k} = 0 \text{ if } \abs{k} > n \text{ and } \nabla\cdot u = 0
		}$, let $\U_n \defeq V_n \cap H^4_{\delta_0/2}$ where $H^\alpha_R$ denotes the open ball around zero of radius $R$ in $H^\alpha$, and
		let $\PP_n$ be the orthogonal projection onto $V_n$ defined by $\hat{\PP}_n \brac{k} = \mathbbm{1} \brac{\abs{k}\leqslant n}$.

		We approximate the system obtained after Leray projection in Step 1 by
		\begin{equation}
		\label{eq:approx_problem}
			\pdt D Z_n = \widetilde{\Leb} Z_n + \PP_n\PP N\brac{Z_n} \text{ and }
			Z_n\brac{0} = \PP_n Z_0.
		\end{equation}
		In order to use standard finite-dimensional ODE theory we write \eqref{eq:approx_problem} as
		\begin{equation}
		\label{eq:approx_problem_std_ODE_form}
			\pdt Z_n = F_n \brac{Z_n} \text{ and } Z_n\brac{0} = \PP_n Z_0
		\end{equation}
		for $F_n = D\inv \brac{\widetilde{\Leb} + \PP_n\PP N}$.
		It follows from \fref{Lemma}{lemma:N_is_analytic} that $F_n$ is analytic from $H^4_{\delta_0}$ to $H^2$,
		and since $\U_n$ is a subset of $H^4_{\delta_0}$ and $\PP\circ\PP$ maps onto $V_n$ we deduce that $F_n$ maps $\U_n$ to $V_n$.

		We may now apply standard ODE theory, which tells us that if we pick an initial condition  $Z_0 =\brac{u_0, \omega_0, J_0} \in H^4$
		which satisfies $\nabla\cdot u_0 = $, $\fint_{\T^3} u_0 = 0$, and $\norm{Z_0}{H^4} < \delta_\text{lwp}$ then there exists a maximal time of existence $T_n > 0$,
		a unique $Z_n \in C^\infty\brac{\cobrac{0,T_n}; \U_n}$ solving \eqref{eq:approx_problem_std_ODE_form},
		and the following blow-up criterion holds: for any $T>0$ if $\sup_{0\leqslant t \leqslant T} \norm{Z_n\brac{t}}{H^4} < \frac{\delta_0}{2}$ then $T\leqslant T_n$.

		\textbf{Step 3:} Uniform bounds on the approximate solutions.
		
		To obtain uniform bounds it suffices to apply \fref{Proposition}{prop:chain_energy_ineq} to the approximate solutions $Z_n$.
		Since \fref{Proposition}{prop:chain_energy_ineq} is only applicable in a small energy regime we must first ensure that $\norm{Z_n}{H^4}$ remains sufficiently small.
		We defined $\tilde{T}_n$ to this effect below.

		Let $\delta_u = \frac{1}{3} \min\brac{\delta_0, \delta}$, and
		let $\tilde{T}_n = \sup\setdef{t>0}{\norm{Z_n}{H^4}\leqslant \delta_u}$.
		Note that $\tilde{T}_n \geqslant T_n$ by the blow-up criterion from Step 1.
		We may now apply a time-integrated version of \fref{Proposition}{prop:chain_energy_ineq} (with $\varepsilon = \frac{1}{2}$) to obtain
		\begin{equation}
		\label{eq:unif_est_approx_ODE_k_0}
			\frac{1}{2} \norm{\sqrt{D}Z_n \brac{t}}{L^2}^2 - \frac{1}{2} \norm{\sqrt{D}Z_n \brac{0}}{L^2}^2
			+ \int_0^t \mathcal{D}\brac{u_n,\omega_n}\brac{s} ds
			\leqslant \int_0^t \brac{ \frac{1}{2} + C_0 } \norm{Z_n \brac{s}}{L^2}^2 ds
		\end{equation}
		and, for $k=1,2,3,4$,
		\begin{align}
			\frac{1}{2} \norm{\nabla^k \brac{\sqrt{D}Z_n \brac{t}}}{L^2}^2 - \frac{1}{2} \norm{\nabla^k \brac{\sqrt{D}Z_n \brac{0}}}{L^2}^2
			+ \int_0^t \frac{C_D}{2} \norm{\nabla^k \brac{u_n,\omega_n}\brac{s}}{H^1}^2 ds
			\nonumber
			\\
			\leqslant \int_0^t \max\brac{ \frac{1}{2}, C_1} \sum_{i=0}^{k} \norm{\nabla^i Z_n \brac{s}}{L^2}^2 ds.
		\label{eq:unif_est_approx_ODE_k_geq_1}
		\end{align}
		where $C_0$, $C_1$, and $C_D$ are as in \fref{Proposition}{prop:chain_energy_ineq}.
		Note that \fref{Proposition}{prop:chain_energy_ineq} as stated applies to solutions of $\pdt DZ = \widetilde{\Leb}Z + N\brac{Z} + \Lambda\brac{p}$
		whereas $Z_n$ satisfies $\pdt DZ_n = \widetilde{\Leb} Z_n + \PP_n \PP_L N \brac{Z_n}$.
		Nonetheless, \fref{Proposition}{prop:chain_energy_ineq} applies to $Z_n$ as well since this theorem relies solely on energy estimates, and in particular,
		since $\int_{\T^3} \Lambda\brac{p}\cdot Z = 0$ when $\nabla\cdot u=0$ and $\int_{\T^3} \PP_n \PP N\brac{Z_n} \cdot Z_n = \int_{\T^3} N\brac{Z_n}\cdot Z_n$
		since $Z_n$ belongs to the image of the projection $\PP_n \circ \PP$, it follows that the estimate obtained for $Z$ in \fref{Proposition}{prop:chain_energy_ineq} also holds for $Z_n$.

		Summing \eqref{eq:unif_est_approx_ODE_k_0} and \eqref{eq:unif_est_approx_ODE_k_geq_1} and using the integral form of the Gronwall inequality tells us that, for any $0 < t < \tilde{T}_n$,
		\begin{equation}
		\label{eq:unif_est_post_Gronwall}
			\norm{Z_n\brac{t}}{H^4}^2 + \int_0^t \norm{\brac{u_n,\omega_n}}{H^5}^2 \leqslant C_2 e^{\rho t} \norm{Z_0}{H^4}^2
		\end{equation}
		where $\rho\defeq 2\brac{1+C_0+C_1}\max\brac{1,2/\nu}$.
		In particular we deduce from the blow-up criterion that if we denote by $T_\text{lwp}$ the infimum of $T_n$ over $n$ then
		$T_\text{lwp} \geqslant \frac{1}{\rho} \log \frac{\delta_\text{lwp}}{\norm{Z_0}{H^4}}$.
		In other words we have a uniform lower bound on the time of existence of the approximate solutions.

		Now we obtain bounds on the time derivative $\pdt Z_n$, which are required for the compactness argument in Step 4.
		Note first that \eqref{eq:unif_est_post_Gronwall} tells us that, for $C_4 = C_2 e^{\rho T_\text{lwp}}$,
		\begin{equation}
		\label{eq:unif_est_Z}
			\sup_n \brac{
				\norm{\brac{u_n, \omega_n, J_n}}{L^\infty H^4}^2 + \norm{\brac{u_n, \omega_n}}{L^2 H^5}^2
			}
			\leqslant C_4 \norm{Z_0}{H^4}^2
		\end{equation}
		where $L^p H^s$ denote $L^p \brac{ \sbrac{0, T_\text{lwp}}; H^s}$.
		Using \fref{Lemma}{lemma:N_is_analytic} and the boundedness of $\widetilde{\Leb}$, $\PP_n$, and $\PP$ we deduce from \eqref{eq:unif_est_Z} that, for some $C_5 > 0$,
		\begin{equation}
		\label{eq:unif_est_pdt_Z}
			\sup_n \brac{
				\norm{\pdt \brac{u_n, \omega_n, J_n}}{L^\infty H^2}^2
				+ \norm{\pdt\brac{u_n,\omega_n}}{L^2 H^3}^2
			} \leqslant C_5 \norm{Z_0}{H^4}^2.
		\end{equation}
		Finally we improve this bound on $\pdt Z_n$ by paying closer attention to the structure of the PDE \eqref{eq:approx_problem}.
		Specifically: since $\widetilde{\Leb}_3$ and $N_3$ lose fewer derivatives than $\widetilde{\Leb}$ and $N$ do, we obtain an improved estimate for $\pdt J_n$:
		\begin{equation}
		\label{eq:unif_est_pdt_J}
			\sup_n \norm{\pdt J_n}{L^\infty H^3}^2 \leqslant C_4 \norm{Z_0}{H^4}^2.
		\end{equation}

		\textbf{Step 4:} Passing to the limit by compactness.
		
		By applying Banach-Alaoglu (i.e. the weak-$\ast$ compactness of bounded sets) to the bounds provided by \eqref{eq:unif_est_Z}, \eqref{eq:unif_est_pdt_Z}, and \eqref{eq:unif_est_pdt_J}
		we obtain a subsequence of $\brac{Z_n}$, which for simplicity we do not relabel, such that
		\begin{equation}
		\label{eq:limits_by_compactness_Z}
				Z_n \overset{\ast}{\rightharpoonup} Z \text{ in } L^\infty H^4,
				\brac{u_n, \omega_n} \rightharpoonup \brac{u, \omega} \text{ in } L^2 H^5,
		\end{equation}
		\begin{equation}
		\label{eq:limits_by_compactness_pdt_Z}
				\pdt Z_n \overset{\ast}{\rightharpoonup} \pdt Z \text{ in } L^\infty H^2,
				\pdt Z_n \rightharpoonup \pdt Z \text{ in } L^2 H^3, \text{ and }
				\pdt J_n \overset{\ast}{\rightharpoonup} \pdt J \text{ in } L^\infty H^3
		\end{equation}
		for some $Z = \brac{u, \omega, J} \in L^\infty H^4$ with $\brac{u, \omega} \in L^2 H^5$,
		$\pdt Z \in L^\infty H^2 \cap L^2 H^3$, and $\pdt J \in L^\infty H^3$.
		Moreover, it follows from Aubin-Lions-Simon that, passing to another subsequence which we do not relabel,
		\begin{equation}
			Z_n \rightarrow Z \text{ in } C^0 H^{4-\frac{1}{4}}
		\end{equation}
		and that $Z\in C^0 H^{4-\frac{1}{4}}$.

		We now pass to the limit. 
		It follows immediately from \eqref{eq:limits_by_compactness_Z} and \eqref{eq:limits_by_compactness_pdt_Z} that
		\begin{equation}
		\label{eq:pass_limit_A}
			\pdt D Z_n \overset{\ast}{\rightharpoonup} \pdt D Z
			\text{ and }
			\widetilde{\Leb} Z_n \overset{\ast}{\rightharpoonup} \widetilde{\Leb} Z
			\text{ in } L^\infty H^2.
		\end{equation}
		To pass to the limit in the nonlinearity we write
		\begin{equation*}
			\PP_n \PP N \brac{Z_n} - \PP N \brac{Z}
			= \PP_n \PP \brac{ N\brac{Z_n} - N\brac{Z} }
			+ \brac{\PP_n - I} \PP N \brac{Z}
			\defeq A + B.
		\end{equation*}
		Passing to the limit in $B$ is immediate: by weak-$\ast$ lower semi-continuity of the $L^\infty H^4$ norm
		we know that $\sup_{0\leqslant t \leqslant T_0} \norm{Z\brac{t}}{H^4} \leqslant \frac{\delta_0}{2} < \delta_0$ such that $N\brac{Z}$ is a well-defined element of $L^\infty H^2$.
		In particular, since $\norm{\brac{I - \PP_n}f}{H^s} \to 0$ for all $s\geqslant 0$ and all $f\in H^s$, it follows that
		\begin{equation}
		\label{eq:pass_limit_B}
			\norm{B}{L^\infty H^2}
			= \norm{ \brac{I - \PP_n} \PP N\brac{Z} }{L^\infty H^2}
			\to 0.
		\end{equation}
		Passing to the limit in $A$ relies on the analyticity of the nonlinearity obtained in \fref{Lemma}{lemma:N_is_analytic}: since $Z_n \to Z$ in $C^0 H^{4-\frac{1}{4}}$ and
		since, as observed above, both the sequence $\brac{Z_n}$ and its limit $Z$ lie in $H^4_{\delta_0/2}$, it follows from \fref{Lemma}{lemma:N_is_analytic}
		(since $2-\frac{1}{4} > \frac{3}{2}$) that $N\brac{Z_n} \to N\brac{Z}$ in $C^0 H^{4-\frac{1}{4}}$. So finally:
		\begin{equation}
		\label{eq:pass_limit_C}
			\norm{A}{L^\infty H^{2-\frac{1}{4}}}
			= \norm{ \PP_n \PP \brac{ N\brac{Z_n} - N\brac{Z} } }{L^\infty H^{2-\frac{1}{4}}}
			\leqslant \norm{ N\brac{Z_n} - N\brac{Z} }{L^\infty H^{2-\frac{1}{4}}}
			\to 0.
		\end{equation}
		We conclude from \eqref{eq:approx_problem}, \eqref{eq:pass_limit_A}, \eqref{eq:pass_limit_B}, and \eqref{eq:pass_limit_C}
		that $Z$ is a strong solution of $\pdt DZ = \widetilde{\Leb}Z + \PP N\brac{Z}$.
		As a consequence we deduce that the conditions $\nabla\cdot u = 0$ and $\fint_{\T^3} u = 0$ are propagated in time, i.e. they hold for every $0 \leqslant t < T_\text{lwp}$.

		Finally we deduce from \eqref{eq:unif_est_Z}, \eqref{eq:unif_est_pdt_Z}, and \eqref{eq:unif_est_pdt_J} and the weak and weak-$\ast$ lower semi-continuity of the appropriate norms that,
		for some $C > 0$,
		\begin{equation}
		\label{eq:lwp_est}
			\norm{\brac{u, \omega, J}}{L^\infty H^4}
			+ \norm{\brac{u, \omega}}{L^2 H^5}
			+ \norm{\pdt\brac{u, \omega, J}}{L^\infty H^2 \cap L^2 H^3}
			+ \norm{\pdt J}{L^\infty H^3}
			\leqslant C \norm{Z_0}{H^4}.
		\end{equation}

		\textbf{Step 5:} Reconstructing the pressure.

		The key observation is that since $\PP = \PP_L \oplus I_3 \oplus I_{3\times 3}$ we may reconstruct $p$ via $I - \PP_L$,
		where $I - \PP_L = \nabla\Delta\inv\nabla\cdot$ as per \fref{Lemma}{lemma:formula_complement_Leray_projector}.
		More precisely: let $p \defeq \Delta\inv\brac{\nabla\cdot N_1\brac{Z}}$ and note that $p$ thus defined has average zero.
		Then, by \fref{Lemma}{lemma:formula_complement_Leray_projector}, $\nabla p = \brac{I - \PP_L}N_1\brac{Z}$ and hence $\Lambda\brac{p} = -\brac{I-\PP}N\brac{Z}$
		such that \eqref{eq:thm_lwp_PDE} holds.
		Finally, since $N_1\brac{Z} = - \brac{u\cdot\nabla}u$ and since $H^s$ is an algebra for $s > 3/2$ we have that, for $s=3$ or $4$,
		\begin{equation*}
			\norm{p}{H^s}
			\lesssim \norm{N_1\brac{Z}}{H^{s-1}}
			= \norm{\brac{u\cdot\nabla}u}{H^{s-1}}
			\lesssim \norm{u}{H^{s-1}} \norm{u}{H^s}.
		\end{equation*}
		Combining these estimates with \eqref{eq:lwp_est} yields \eqref{eq:thm_lwp_estimate_pressure}.
	\end{proof}

	\begin{remark}
		It may appear somewhat odd that the initial condition $Z\brac{0} = Z_0$ of \eqref{eq:thm_lwp_PDE} holds in $H^{4-\frac{1}{4}}$ and not in $H^4$ as one might expect.
		This is due to the loss of spatial regularity incurred when applying the Aubin-Lions-Simon lemma to obtain strong convergence of the approximate solutions in $C^0 H^{4-\frac{1}{4}}$.
		In particular, note that the only thing which is special about $\frac{1}{4}$ is that it sits squarely between $0$ and $\frac{1}{2}$ and that we use that
		$\brac{4 - \frac{1}{4}} - 2 > \frac{3}{2}$ when we leverage \fref{Lemma}{lemma:N_is_analytic} to pass to the limit in the nonlinearity in Step 4 of the proof of \fref{Theorem}{thm:lwp}.
		This means that we can actually show that $Z\brac{0} = Z_0$ in $H^{4-\varepsilon}$ for any $0 < \varepsilon < \frac{1}{2}$, since then $4 - \varepsilon < 4$ such that Aubin-Lions-Simon
		applies and $\brac{4 - \varepsilon} - 2 > \frac{3}{2}$ such that we may still use \fref{Lemma}{lemma:N_is_analytic}.
	\end{remark}

	We now state and prove our uniqueness result.
	Note that the only assumptions made are boundedness of appropriate Sobolev norms of the solutions.
	No smallness assumptions are made here.
	\begin{thm}[Uniqueness]
	\label{thm:uniqueness}
		Suppose that, for $i = 1, 2$, $\brac{u_i, p_i, \omega_i, J_i}$ are strong solutions of
		\begin{equation*}
			\left\{
			\begin{aligned}
				&\pdt u_i + \brac{u_i \cdot \nabla} u_i = \brac{\nabla\cdot T} \brac{u_i, p_i, \omega_i},\\
				&\nabla\cdot u_i = 0,\\
				&J_i \brac{\pdt \omega_i + \brac{u_i\cdot\nabla} \omega_i} + \omega_i \wedge J_i \omega_i
					= 2\vc T\brac{u_i, p_i, \omega_i} + \brac{\nabla\cdot M}\brac{\omega_i} + \tau e_3, \text{ and }\\
				&\pdt J_i + \brac{u_i \cdot\nabla}J_i = \sbrac{\Omega_i, J_i}
			\end{aligned}
			\right.
		\end{equation*}
		on some common time interval $\brac{0,T}$ which agree initially, i.e. which agree at time $t=0$.
		If $J_1$ is uniformly positive-definite, $p_i, \pdt\brac{u_i, \omega_i, J_i} \in L_T^2 L^2$, $\brac{u_i, \omega_i, J_i}, \nabla\brac{u_i, \omega_i, J_i} \in L_T^\infty L^\infty$,
		and $\pdt J_1, \pdt\omega_2 \in L_T^\infty L^\infty$, then these solutions coincide on $\brac{0,T}$.
	\end{thm}
	\begin{proof}
		This follows from simple energy estimates for the equations satisfied by the difference of the two solutions.
		The difference $\brac{u, p, \omega, J} = \brac{u_1 - u_2, p_1 - p_2, \omega_1 - \omega_2, J_1 - J_2}$ satisfies
		\begin{subnumcases}{}
			\brac{\pdt + u_1\cdot\nabla} u = \brac{\nabla\cdot T}\brac{u, p, \omega} + f,
			\label{eq:lemma_eq_sat_diff_1}\\
			\nabla\cdot u_1 = 0,
			\label{eq:lemma_eq_sat_diff_2}\\
			\brac{J_1 \brac{\pdt + u_1\cdot\nabla} + \omega_1 \wedge J_1} \omega = 2\vc T\brac{u, p, \omega} + \brac{\nabla\cdot M}\brac{\omega} + g,
			\label{eq:lemma_eq_sat_diff_3}\\
			\brac{\pdt + u_1\cdot\nabla} J_1 = \sbrac{\Omega_1, J_1}, \text{ and }
			\label{eq:lemma_eq_sat_diff_4}\\
			\brac{\pdt + u_1\cdot\nabla} J = \sbrac{\Omega, J} + h
			\label{eq:lemma_eq_sat_diff_5}
		\end{subnumcases}
		for
		\begin{equation*}
			\left\{
			\begin{aligned}
				&f = -\brac{u\cdot\nabla}u_2\\
				&g = -J\pdt\omega_2 - J_1\brac{u\cdot\nabla}\omega_2 - J\brac{u_2\cdot\nabla}\omega_2 - \omega_1\wedge J\omega_2 - \omega\wedge J_2\omega_2, \text{ and }\\
				&h = -\brac{u\cdot\nabla}J_2 + \sbrac{\Omega, J_2}.
			\end{aligned}
			\right.
		\end{equation*}
		We can thus multiply \eqref{eq:lemma_eq_sat_diff_1}, \eqref{eq:lemma_eq_sat_diff_3}, and \eqref{eq:lemma_eq_sat_diff_5} by $u$, $\omega$, and $J$ respectively to see that,
		for every $0 < t < T$,
		\begin{align*}
			&\int_{\T^3} \frac{1}{2} \abs{u}^2 + \frac{1}{2} J_1 \omega\cdot\omega + \frac{1}{2} \abs{J}^2 \bigg\vert_{s=t}
			- \int_{\T^3} \frac{1}{2} \abs{u}^2 + \frac{1}{2} J_1 \omega\cdot\omega + \frac{1}{2} \abs{J}^2 \bigg\vert_{s=0}
			\\
			+ &\int_0^t \int_{\T^3} 
				\frac{\mu}{2} \abs{\symgrad u}^2
				+ 2 \kappa {\vbrac{\half\nabla\times u - \omega}}^2
				+ \alpha\abs{\nabla\cdot\omega}^2
				+ \frac{\beta}{2} \abs{\symgrad^0 \omega}^2
				+ 2\gamma \abs{\nabla\times\omega}^2
			= \int_0^t \int_{\T^3} f\cdot u + g\cdot\omega + h:J.
		\end{align*}
		We can write this energy-dissipation-interaction relation more succintly as $\mathcal{E}(t)-\mathcal{E}(0) + \int_0^t \mathcal{D} = \int_0^t \mathcal{I}$
		for $\mathcal{I} = \int_{\T^3} f\cdot u + g\cdot\omega + h:J$. It follows from straightforward application of the H\"{o}lder and Cauchy-Schwartz inequalities that the interactions
		are controlled by the energy, i.e. $\abs{\mathcal{I}} \leqslant C\mathcal{E}$ for some constant $C > 0$.
		Note that since the two solutions agree initially we have that $\mathcal{E} (0) = 0$.
		Therefore the integral version of Gronwall's inequality tells us that $\mathcal{E} (t) = 0$ for all $0 < t < T$.
		Since $J_1$ is uniformly positive definite we deduce that $\brac{u,\omega,J} = 0$.
		Finally, since $-\Delta p = \nabla u_1 : {\nabla u}^T + \nabla u : {\nabla u_2}^T = 0$, we conclude that indeed the two solutions coincide.
	\end{proof}

	\noindent
	Putting \fref{Theorem}{thm:lwp} and \fref{Theorem}{thm:uniqueness} together yields our local well-posedness result, stated below.
	\begin{cor}[Local well-posedness]
	\label{cor:lwp}
		The solution obtained in \fref{Theorem}{thm:lwp} is unique.
	\end{cor}
	\begin{proof}
		This is immediate since the assumptions of \fref{Theorem}{thm:lwp} ensure that \fref{Theorem}{thm:uniqueness} applies.
	\end{proof}

\section{Auxiliary results}
\label{sec:auxiliary}
	Here we record auxiliary results which are used throughout the main body of the paper.
	Whilst these results are typically either elementary lemmas or well-known theorems, they are of interest since they are applicable beyond the scope of this paper.

	\begin{lemma}[Lower bound on the real part of complex square roots]
	\label{lemma:low_bound_re_part_complex_sq_root}
		Let $x,y\in\R$ with $y\neq 0$ and let $\alpha > 0$. We follow the convention according to which the square root of a complex number with non-trivial imaginary part
		is chosen to have a strictly positive real part. Then $\re\sqrt{x+iy} > \alpha$ if and only if $x > \alpha^2 - \frac{y^2}{4\alpha^2}$.
	\end{lemma}
	\begin{proof}
		Let us write $\sqrt{x+iy} = u+iv$ for some $u>0$ and $v\in\R$, such that $x = u^2 - v^2$ and $y = 2uv$.
		What we wish to prove can then be written as $u > \alpha$ if and only if $u^2 - v^2 > \alpha^2 - \frac{u^2 v^2}{\alpha^2}$.
		The latter inequality can be rearranged as $u^2 - \alpha^2 > -\frac{v^2}{\alpha^2} \brac{u^2 - \alpha^2}$.
		This can be simplified, using the fact that $u + \alpha > 0$, to $\brac{u - \alpha} \brac{1 + \frac{v^2}{\alpha^2}} > 0$.
		This is indeed equivalent to $u > \alpha$ so we are done.
	\end{proof}

	\begin{lemma}[Similarity of matrices acting on quotient spaces]
	\label{lemma:sim_mat_act_on_quot_spaces}
		Let $V$ be a subspace of $\C^n$ and let $A$, $G$, and $H$ be complex $n$-by-$n$ matrices which act on $\C^n /\, V$ (c.f. \fref{Definition}{def:lin_map_act_on_quot_space}) such that
		$GH = HG = \proj_{V^\perp}$.
		Then (1) $B \defeq G A H$ acts on $\C^n /\, V$, (2) $A = H B G$, and (3) $A$ and $B$ are similar.
	\end{lemma}

	\begin{proof}
		First we show that $B$ acts on $\C^n /\, V$.
		We know that $\im B \subseteq \im G \subseteq V^\perp$ and that $V = \ker H \subseteq \ker B$,
		so it is enough to show that $\ker B \subseteq V$.
		Let $x\in\ker B$.
		Since $Hx\in V^\perp$, it suffices to show that $Hx\in V$ as then $Hx=0$, i.e. $x\in\ker H = V$.
		The key observation is that since $\im A \subseteq V^\perp$ and since $G$ and $H$ are inverses on $V^\perp$, we obtain that $A=HGA$.
		It follows that $AHx = HGAHx = HBx = 0$, i.e. $Hx\in\ker A = V$, and hence (1) holds.
		
		Now observe that in order to prove that $A = HBG$ it is enough to show that $HGA=A$, which was done above, and that $AHG=A$, which we do now.
		Pick any $x\in\C^n$ and write $x=x_\parallel + x_\perp$ for $x_\parallel\in V$ and $X_\perp\in V^\perp$.
		Since $\ker G = \ker A = V$ and since $HG=\proj_{V^\perp}$ it follows that $AHGx = AHGx_\perp = Ax_\perp = Ax$, i.e. indeed $AHG = A$.
		
		Finally we show that $A$ and $B$ are similar by explicitly finding an appropriate change-of-basis matrix.
		Let $P$ be the orthogonal projection onto $V$, i.e. $\ker P = V^\perp$ and $P\vert_V = \id\vert_V$.
		Observe that, since $\ker B = V = \im P$ and since $\im B \subseteq V^\perp = \ker P$, we may deduce that $BP = PB = 0$. Therefore
		\begin{equation}
		\label{eq:witness_3}
			\brac{H+P}B\brac{G+P}=A.
		\end{equation}
		We will now show that $G+P$ and $H+P$ are invertible and $(G+P)^{-1} = H+P$, from which it follows that \eqref{eq:witness_3} witnesses (3).
		Let $x\in\ker\brac{G+P}$ and let us write $x=x_\parallel + x_\perp$ as above.
		Then $0=\brac{G+P}x = Gx_\perp + x_\parallel$ with $Gx_\perp \in V^\perp$ and $x_\parallel\in V$, and hence we must have $Px_\perp = 0$ and $x_\parallel = 0$.
		In particular, since $\ker G = V$, we know that $x_\perp$ belongs to both $V$ and $V^\perp$ and hence $x_\perp = 0$, such that $x=0$.
		This shows that $G+P$ has trivial kernel and is thus invertible.
		We may deduce in exactly the same way that $H+P$ is invertible.   To conclude we simply compute $(H + P)(G+P) = H G + H P + PG + P^2 = HG + P = I.$
	\end{proof}

	\begin{lemma}[Bounds on the real parts of the eigenvalues of a matrix using the spectrum of its symmetric part]
	\label{lemma:bounds_real_part_evals_using_sym_part_mat}
		Let $S$ and $A$ be symmetric and antisymmetric real $n$-by-$n$ matrices respectively. It then holds that $\min\sigma\brac{S} \leqslant \re\sigma\brac{S+A} \leqslant \max\sigma\brac{S}$.
	\end{lemma}

	\begin{proof}
		Let us denote by $\lambda_+$ and $\lambda_-$ the maximal and minimal eigenvalues of $S$, respectively,
		let us define $M = S+A$, and let $a+ib$, $a,b\in\R$, be an eigenvalue of $M$ with eigenvector $x+iy$, $x,y\in\R^n$.
		Then, since $M\brac{x+iy} = \brac{a+ib}\brac{x+iy}$ it follows that $Mx = ax - by$ and $My = bx + ay$.
		In particular
		$
			Sx\cdot x + Sy\cdot y
			= Mx\cdot x + My\cdot y
			= a \brac{\abs{x}^2 + \abs{y}^2}
		$
		where
		$
			Sx\cdot x + Sy\cdot y \leqslant \lambda_+ \brac{\abs{x}^2 + \abs{y}^2},
		$
		and therefore $a \leqslant \lambda_+$.
		We may obtain in exactly the same way that $a \geqslant \lambda_-$, and hence indeed $\lambda_- \leqslant \re\sigma\brac{S+A} \leqslant \lambda_+$.
	\end{proof}

	\begin{thm}[Gershgorin disk theorem]
	\label{thm:Gershgorin}
		Let $A$ be a complex $n$-by-$n$ matrix and let $R_i \defeq \sum_{j\neq i} \abs{A_{ij}}$ for $i=1,\,\dots,\,n$.
		Every eigenvalue of $A$ lies in one of the closed disks $\overline{B\brac{A_{ii},R_i}}$, where $i=1,\,\dots,\,n$.
		These disks are called the \emph{Gershgorin disks} of A.
	\end{thm}

	\begin{proof}
		Let $v$ be an eigenvector of $A$ with eigenvalue $\lambda$.
		Without loss of generality (otherwise we may divide $v$ by $\pm\norm{v}{\infty}$): $v_i = 1$ for some index $i$ and $\abs{v_j} \leqslant 1$ for all indices $j$ different from $i$.
		Now observe that
		\begin{equation*}
			{\brac{Av}}_i = \lambda v_i
			\quad\Leftrightarrow\quad
			A_{ii} v_i + \sum_{j\neq i} A_{ij} v_j = \lambda v_i
			\quad\Leftrightarrow\quad
			\lambda - A_{ii} = \sum_{j\neq i} A_{ij} v_j
		\end{equation*}
		and thus
		$
			\abs{\lambda - A_{ii}}
			\leqslant \sum_{j\neq i} \abs{A_{ij}} \abs{v_j}
			\leqslant \sum_{j\neq i} \abs{A_{ij}}
			 = R_i
		 $
		i.e. indeed $\lambda$ lies in $\overline{B\brac{A_{ii},R_i}}$, which is one of the Gershgorin disks of $A$.
	\end{proof}

	\begin{cor}[Bounds on the imaginary parts of the eigenvalues of a matrix using the Frobenius norm of its antisymmetric part]
	\label{cor:bounds_im_part_evals_using_norm_antisym_part_mar}
		Let $S$ and $A$ be symmetric and antisymmetric real $n$-by-$n$ matrices respectively. Then $\abs{\imp\sigma\brac{S+A}} \leqslant \sqrt{n-1}\norm{A}{2}$,
		where $\norm{A}{2} \defeq \sqrt{A:A}$ is the \emph{Frobenius norm} of $A$.
	\end{cor}
	
	\begin{proof}
		Since $S$ is symmetric, there exists an orthogonal matrix $Q$ and a diagonal matrix $D$ such that $QSQ^T = D$.
		Therefore $Q\brac{S+A}Q^T = D+QAQ^T$. In particular, for $\tilde{A} \defeq QAQ^T$, we know that $S+A$ and $D+\tilde{A}$ have the same spectrum.
		Writing $D = \diag\brac{\lambda_1,\dots,\lambda_n}$ where the $\lambda_i$'s are the eigenvalues of $S$,
		we may apply \fref{Theorem}{thm:Gershgorin} to deduce that the eigenvalues of $D+\tilde{A}$
		lie within closed disks
		centered at $\lambda_i$ (since $\tilde{A}$ is antisymmetric and hence all its diagonal entries are equal to zero) and
		with corresponding radii
		$
			R_i~=~\sum_{j\neq i} \abs{\tilde{A}_{ij}}~\leqslant~\sqrt{n-1} \, \normns{\tilde{A}}{2}.
		$
		The result then follows from the observation that the eigenvalues $\lambda_i$ of the symmetric matrix $S$ are real
		and the fact that $\normns{\tilde{A}}{2}^2 = QAQ^T:QAQ^T = Q^TQAQ^TQ:A = \norm{A}{2}^2$.
\end{proof}

	\begin{lemma}[Bounds on matrix exponentials using the symmetric part]
		\label{lemma:bounds_mat_exp_using_sym_part}
		Let $M$ be a real $n$-by-$n$ matrix, let $S\defeq \frac{1}{2}\brac{M+M^T}$ denote its symmetric part, and let $\sigma$ denote the largest eigenvalue of $S$.
		Then, for every $t > 0$, $\norm{e^{tM}}{\Leb\brac{l^2,l^2}} \leqslant e^{\sigma t}$.
	\end{lemma}

	\begin{proof}
		This follows from a simple Gronwall inequality upon noticing that, for any $x\in\R^n$, $Mx\cdot x = Sx\cdot x$.
		More precisely: pick any $x_0\in\R^n$ and define $x\brac{t}\defeq e^{tM} x_0$ for every $t\geqslant 0$.
		Observe that $\Dt x\brac{t} = Mx\brac{t}$ and hence $\Dt\norm{x\brac{t}}{2}^2 = 2Sx\brac{t}\cdot x\brac{t} \leqslant 2\sigma\norm{x\brac{t}}{2}^2$.
		Since $x\brac{0} =  x_0$, applying Gronwall's inequality yields that, for every $t\geqslant 0$, $\norm{e^{tM}x_0}{2}^2 = \norm{x\brac{t}}{2}^2 \leqslant e^{2\sigma t}\norm{x_0}{2}^2$,
		from which the result follows.
	\end{proof}

	\begin{lemma}[Bounds on matrix exponentials for Jordan canonical forms]
	\label{lemma:bounds_mat_exp_Jordan_form}
		For any matrix norm $\abs{\,\cdot\,}$ there exists a constant $C_n > 0$ such that for every complex $n$-by-$n$ matrix $M$ in Jordan canonical form,
		if $\eta\defeq\max\re\sigma\brac{M}$ then, for every $t\geqslant 0$, $\abs{e^{tM}} \leqslant C_n \brac{1+t^n} e^{\eta t}$.
	\end{lemma}

	\begin{proof}
		Since $M$ is in Jordan canonical form it can be written as
		$
			M = J_{a_1} \brac{\lambda_1} \oplus \dots \oplus J_{a_k} \brac{\lambda_k}
		$
		where the $\lambda_i$'s are eigenvalues of $M$ and $J_{a} \brac{\lambda} = \lambda I_a + N_a$
		for ${\brac{N_a}}_{ij} = 1$ if $j=i+1$ and ${\brac{N_a}}_{ij} = 0$ otherwise,
		Note that, since $N_a$ is an $a$-by-$a$ matrix whose only non-zero entries are those immediately above the diagonal, it is nilpotent of order $a$.
		In particular, note that since the identity commutes with all matrices, it follows that
		$
			e^{J_a\brac{\lambda}} = e^{\lambda} e^{N_a},
		$
		and recall that for any nilpotent matrix $N$ of order $q$ its matrix exponential is given by a finite sum, i.e.
		$
			e^N
			= \sum_{j=0}^{q-1} \frac{1}{j!} N^j.
		$
		We can thus compute the matrix exponential of $M$ to be
		$
			e^{tM}
			= e^{\lambda_1 t} e^{t N_{a_1}} \oplus e^{\lambda_k t} e^{t N_{a_k}}
		$
		which can be estimated by
		$
			\abs{e^{tM}}
			\leqslant \sum_{i=1}^k e^{\brac{\re\lambda_i} t} \vbrac{\sum_{j=0}^{a_i} \frac{1}{j!} {\brac{tN_{a_i}}}^j}
			\lesssim e^{\eta t} \brac{1 + t^n}
		$
		where have used that polynomials of degree $q$ in a real variable $x$ can be bounded above (up to a constant) by $1 + x^q$,
		and where the constants up to which the inequalities above hold only depends on $n$ and the choice of the matrix norm.
	\end{proof}

	\begin{cor}[Bounds on matrix exponentials]
	\label{cor:bounds_mat_exp}
		Let $M$ be a real $n$-by-$n$ matrix and let $\eta\!\defeq\!\max\re\sigma\brac{M}$.
		For any matrix norm $\abs{\,\cdot\,}$ there exists a constant $C = C\brac{M} > 0$ such that, for every $t\in\R$, it holds that $\abs{e^{tM}} \leqslant C \brac{1 + t^n} e^{\eta t}$.
	\end{cor}

	\begin{proof}
		This follows from \fref{Lemma}{lemma:bounds_mat_exp_Jordan_form} since every matrix $M$ is similar to a matrix in Jordan canonical form.
		The constant obtained depends on $M$ since the norm of the matrices used to conjugate $M$ to put it in Jordan canonical form depend on $M$.
	\end{proof}

	\begin{prop}[Construction of a semigroup via matrix exponentials as Fourier multipliers]
	\label{prop:const_semigroup_Fourier}
		Let $M~:~\Z^n \to \R^{l\times l}$ be a family of matrices for which there exists $\eta \in \R$ and $C_F > 0$ such that, for every $k\in\Z^n$ and every $t > 0$,
		\begin{equation}
		\label{eq:uniform_bound_mat_exp}
			\norm{ e^{tM\brac{k}} }{\Leb\brac{l^2,\,l^2}} \leqslant C_F e^{\eta t}.
		\end{equation}
		For any $t\geqslant 0$ the operator $e^{t\Leb}$ defined by the multiplier
		$
			{\brac{e^{t\Leb}}}^\wedge \brac{k} \defeq e^{tM\brac{k}}
		$
		is a bounded operator on $L^2\brac{\T^n; \R^l}$ such that ${\brac{e^{t\Leb}}}_{t\geqslant 0}$ defines an $\eta$-contractive semigroup,
		i.e.
		\begin{enumerate}
			\item	$e^{0 \Leb}$ is the identity,
			\item	for every $t,s\geqslant 0$, $e^{t\Leb} e^{s\Leb} = e^{s\Leb} e^{t\Leb} = e^{\brac{t+s}\Leb}$,
			\item	for every $f\in L^2\brac{\T^n; \R^l}$, $t\to e^{t\Leb} f$ is a continuous map from $\cobrac{0,\infty}$ to $L^2 \brac{\T^n; \R^l}$, and
			\item	for every $r\geqslant 0$,
				$
					\norm{e^{t\Leb}}{\Leb\brac{H^r \brac{\T^n; \R^l}; H^r\brac{\T^n; \R^l}}} \leqslant C_F e^{\eta t}
				$.
		\end{enumerate}
		Moreover, let us write $v = \brac{v_1, \dots, v_p} \in \R^{q_1} \times \dots \times \R^{q_p}$, where $q_1 + \dots + q_p = l$,
		and suppose that there exists $\alpha_1, \dots, \alpha_p \in \N$ and $C_D > 0$ such that for every $k\in\Z^n$ and every $v\in\R^l$,
		\begin{equation}
		\label{eq:semigroup_domain_condition}
			{\vbrac{ M\brac{k}v }}^2 \leqslant C_D \sum_{i=1}^p \jap{k}{2\alpha_i} \abs{v_i}^2.
		\end{equation}
		Then
		\begin{enumerate}
			\setcounter{enumi}{4}
			\item	the domain of the semigroup ${\brac{ e^{t\Leb} }}_{t\geqslant 0}$ is $H^{\alpha_1} \brac{\T^n, \R^{q_1}} \times\dots\times H^{\alpha_p} \brac{\T^n, \R^{q_p}}$ and
			\item	its generator is the linear differential operator $\mathcal{L}$ with symbol $M$, i.e. $\widehat{\Leb} \brac{k} \defeq M\brac{k}$.
		\end{enumerate}
	\end{prop}

	\begin{proof}
		The boundedness of $e^{t\Leb}$ and (4) follow directly from \eqref{eq:uniform_bound_mat_exp}.
		(1) and (2) follow from the fact that, for any matrix $M$, ${\brac{e^{tM}}}_{t\geqslant 0}$ is a representation of the semigroup $\brac{\R_{\geqslant 0},+}$,
		i.e. $e^{0M} = I$ and $e^{tM}e^{sM}=e^{sM}e^{tM}=e^{\brac{t+s}M}$.
		To prove that (3) holds it suffices to show that $t\mapsto e^{t\Leb}f$ is continuous at $t=0$.
		This is immediate since
		\begin{equation*}
			\norm{e^{t\Leb}f-f}{L^2}
			\leqslant \sum_{\abs{k}\leqslant K} \abs{ \brac{e^{tM_k} -I } \hat{f}(k) }^2
			+ {\brac{e^{\eta t} + 1}}^2 \underbrace{\sum_{\abs{k} > K} \abs{\hat{f}(k)}^2}_{\eqdef R_f (K)}
		\end{equation*}
		where $R_f (K) \to 0$ as $K\to\infty$ since $f\in L^2$, and hence, since for any fixed $K$ the collection
		${\cbrac{ t\mapsto e^{tM_k} }}_{\abs{k}\leqslant K}$ is as finite collection of continuous maps,
		we indeed obtain that $e^{t\Leb}f \to f$ in $L^2$ as $t\to 0$.
 
		Finally, to prove (5) and (6) we proceed as we did for (3).
		First we note that, by the mean-value theorem, for every $k\in\Z^n$ and every $t>0$,
		$\frac{e^{tM_k}-I}{t} - M_k = \int_0^1 \brac{ e^{stM_k} - I } M_k ds$.
		Therefore, for any $f\in L^2$ and any $0<t<\delta$, if we write $\hat{f} = \brac{\hat{f}_1,\,\dots,\,\hat{f}_p} \in \R^{q_1}\times\cdots\times\R^{q_p}$ then
		\begin{align*}
			\norm{\frac{e^{t\Leb}f-f}{t}-\Leb f}{L^2}^2
			&\leqslant \sum_{k\in\Z^l} \norm{\int_0^1 \brac{e^{stM_k} - I} ds}{\Leb\brac{l^2,\,l^2}}^2 {\vbrac{M_k \hat{f}\brac{k}}}^2
			\\
			&\leqslant C\brac{K,f} \sum_{\abs{k}<K} \norm{\int_0^1 \brac{e^{stM_k} - I} ds}{\Leb\brac{l^2,\,l^2}}^2
			+ C\brac{\eta,\delta} \underbrace{\sum_{\abs{k}>K} \sum_{i=1}^p \jap{k}{2\alpha_i} {\vbrac{\hat{f}_i\brac{k}}}^2}_{\eqdef H_f\brac{K}}
		\end{align*}
		In particular, if $f\in H^{\alpha_1}\times\cdots\times H^{\alpha_p}$ then $H_f\brac{K}\to 0$ as $K\to\infty$ and thus,
		since, for any fixed $K$, ${\cbrac{ t\mapsto e^{tM_k} }}_{\abs{k}\leqslant K}$ is as finite collection of continuous maps,
		we may conclude that indeed $\frac{e^{t\Leb}f-f}{t}\to\Leb f$ in $L^2$ as $t\to 0$.
	\end{proof}

	\begin{thm}[Rouch\'{e}]
	\label{thm:Rouche}
		Let $\Omega \subseteq \C$ be a connected open set whose boundary is a simple curve and let $f$ and $g$ be holomorphic in $\Omega$.
		If $\abs{f-g} < \abs{f}$ on $\partial\Omega$ then $f$ and $g$ have the same number of zeros in $\Omega$.
	\end{thm}

	\begin{proof}
		See Chapter 4 of \cite{ahlfors}.
	\end{proof}

	\begin{thm}[Implicit Function Theorem for mixed real-complex functions]
	\label{thm:ift}
		Let $f : \mathcal{O} \subseteq \C\times\R^m\to\C$, where $\mathcal{O}$ is open,
		be continuously differentiable in the real sense (i.e. after identifying $\C$ with $\R^2$ in the canonical way) is continuously differentiable.
		Let $\brac{z_0, v_0}\in\mathcal{O}$ and let us write $f = f\brac{z,v}$ for $z\in\C$ and $v\in\R^m$.
		If
		(1)~$f\brac{z_0, v_0} = 0$ and
		(2)~$\partial_z f\brac{z_0, v_0} \neq 0$
		then there exist open sets $\mathcal{U}\subseteq \C\times\R^m$ and $W\subseteq\R^m$ and a function $g : W\to\C$ which is continuously differentiable in the real sense such that
		(1)~$\brac{z_0, v_0}\in\mathcal{U}, v_0\in W$,
		(2)~$g\brac{v_0} = z_0$,
		(3)~$\brac{g\brac{v},v}\in\mathcal{U}$ for every $v\in W$,
		(4)~$f\brac{g\brac{v},v} = 0$ for every $v\in W$ and
		\begin{equation*}
			\nabla_v g \brac{v_0} = \frac{-\nabla_v f \brac{z_0, v_0}}{\partial_z f\brac{z_0, v_0}}.
		\end{equation*}
		Moreover, if $f$ is more regular, in the real sense, then so is $g$.
	\end{thm}

	\begin{proof}
		See Chapter 9 of \cite{rudin}.
	\end{proof}

	\begin{lemma}[Coercivity implies invertibility and bounds on the inverse]
	\label{lemma:coer_implies_invertible_and_bounds}
		Let $B$ be a real $n$-by-$n$ matrix. If $B$ is coercive, i.e. if there exists $C_0 > 0$ such that for every $x\in\R^n$, $\abs{Bx} \geqslant C_0 \abs{x}$,
		then $B$ is invertible and $\norm{B\inv}{\text{op}} \leqslant \frac{1}{C_0}$.
	\end{lemma}

	\begin{proof}
		Observe that since $B$ is coercive, it has trivial kernel, and is hence invertible. To obtain the bound on the operator norm of $B\inv$ simply observe that for every $y\in\R^n$,
		$\abs{y} = \abs{BB\inv y} \geqslant C_0 \abs{B\inv y}$.
	\end{proof}

	\begin{cor}[Invertibility and bounds for perturbations of the identity]
	\label{cor:invertibility_pert_identity}
		Let $B$ be a real $n$-by-$n$ matrix. If $\norm{B}{\text{op}} < 1$
		then $I+B$ is invertible and $\norm{ {\brac{ I+B }}\inv }{\text{op}} \leqslant \frac{1}{1-\norm{B}{\text{op}}}$.
	\end{cor}

	\begin{proof}
		The key observation is that $I+B$ is coercive with coercivity constant $1-\norm{B}{\text{op}}$. The result then follows from \fref{Lemma}{lemma:coer_implies_invertible_and_bounds}.
	\end{proof}

	\begin{lemma}
	\label{lemma:antisymmetry_commutator_on_space_symmetric_matrices}
		Let $A$ and $N$ be real $n$-by-$n$ matrices such that $N$ is normal, i.e. $NN^T = N^TN$. Then $\sbrac{A,N}:N=0$.
	\end{lemma}

	\begin{proof}
		This follows from a direct computation: $NA:N = A:N^TN = A:NN^T = AN:N$ and hence $\sbrac{A,N}:N = AN:N-NA:N = 0$.
	\end{proof}

	\begin{prop}[Korn inequality]
	\label{prop:Korn}
		There exists $C_K>0$ such that for every $u\in H^1\brac{\T^3,\,\R^3}$, \\$\norm{\nabla u}{L^2} \leqslant C_K \brac{ \norm{u}{L^2} + \norm{\symgrad u}{L^2}}$.
	\end{prop}

	\begin{proof}
		See Lemma IV.7.6 in \cite{boyer_fabrie}.
	\end{proof}

	\begin{prop}[Korn-Poincar\'{e} inequality]
	\label{prop:Korn_Poincare}
		There exists $C_{KP}>0$ such that for every $u\in H^1\brac{\T^3,\,\R^3}$, $\norm{u}{L^2} \leqslant C_{KP} \brac{ \vbrac{\fint u} + \norm{\symgrad u}{L^2}}$.
	\end{prop}

	\begin{proof}
		This is a consequence of \fref{Proposition}{prop:Korn} -- see for example Lemma IV.7.7 in \cite{boyer_fabrie} -- noting that $\nabla\times u$ has average zero on the torus.
	\end{proof}

	\begin{lemma}[A div-curl identity on the torus]
	\label{lemma:div_curl_identity}
		For any $v\in H^1\brac{\T^3,\,\R^3}$, it holds that $\norm{\nabla v}{L^2}^2 = \norm{\nabla\cdot v}{L^2}^2 + \norm{\nabla\times v}{L^2}^2$.
	\end{lemma}

	\begin{proof}
		The key observation is that for any $w\in\R^3$ and any nonzero $k\in\Z^3$, $w\mapsto \frac{k\times w}{\abs{k}}$ is an isometry on $\Span_k^\perp$,
		and hence $\abs{w}^2 = \abs{ \proj_k w }^2 + \abs{ \proj_{k^\perp} w }^2 = \frac{\abs{k\cdot w}^2}{\abs{k}^2} + \frac{\abs{k\times w}^2}{\abs{k}^2}$.
		Combining this observation with Parseval's identity allows us to conclude:
		\begin{equation*}
			\norm{\nabla v}{L^2}^2
			= \sum_{k\in\Z^3} \abs{k \otimes \hat{v}\brac{k}}^2
			= \sum_{k\in\Z^3\setminus\cbrac{0}} \abs{k}^2 \abs{ \hat{v}\brac{k}}^2
			= \sum_{k\in\Z^3} \abs{k \cdot \hat{v}\brac{k}}^2
			+ \sum_{k\in\Z^3} \abs{k \times \hat{v}\brac{k}}^2
			= \norm{\nabla\cdot v}{L^2}^2 + \norm{\nabla\times v}{L^2}^2.
		\end{equation*}
	\end{proof}

	\begin{prop}[Estimates from the Fa\`{a} di Bruno formula]
	\label{prop:est_faa_di_bruno}
		Let $\U\subseteq\R^n$ and $\mathcal{V}\subseteq\R^p$ be open and let $g:\U\to\mathcal{V}$ and $F:\mathcal{V}\to\R^q$ be $k$-times differentiable.
		There exists a constant $C = C\brac{n,p,q,k} > 0$ which does not depend on $F$ or $g$ such that, for every $x\in\U$,
		\begin{equation*}
			\vbrac{ \nabla^k \brac{F \circ g} \brac{x} }
			\leqslant C \sum_{i=1}^k \vbrac{ \nabla^i F \brac{g\brac{x}} }
			\sum_{\pi\in P_i\brac{k}} \vbrac{\nabla^\pi g \brac{x}}.
		\end{equation*}
	\end{prop}

	\begin{proof}
		This estimate follows immediately from the Fa\`{a} di Bruno formula, which was first proven in \cite{arbogast} and can be found in a rather clean form in \cite{m_hardy}.
	\end{proof}

	\begin{lemma}[Post-compositions by analytic functions are analytic]
	\label{lemma:post_comp_analyt_is_analyt}
		Suppose that $F:\R^k \to \R^l$ is analytic about zero and let $s>\frac{n}{2}$.
		There exists $\delta > 0$ such that $F^* : H^s_\delta \brac{\T^n; \R^k} \to H^s \brac{\T^n; \R^l}$,
		defined by $F^*\brac{G} = F\circ G$ for every $G\in H^s_\delta$, is analytic.
	\end{lemma}
	\begin{proof}
		Let $\delta = \frac{R}{C_s}$ where $R$ is the radius of convergence of $F$ about zero and $C_s$ is the constant from the continuous embedding $H^s \cdot H^s \hookrightarrow H^s$
		and suppose that $F\brac{x} = \sum_{i=0}^{\infty} F_i \bullet X^{\otimes i}$ for every $x\in B\brac{0,R}$, for some fixed tensorial coefficients $F_i$.
		Then indeed, for every $G\in H^s_\delta$, $F^*\brac{G} = \sum_{i=0}^\infty F_i \bullet G^{\otimes i}$ with
		\begin{equation*}
			\sum_{i=0}^{\infty} \abs{F_i}\, \norm{G^{\otimes i}}{H^s}
			\leqslant \sum_{i=0}^\infty \abs{F_i} \, C_s^i \, \norm{G}{H^s}^i
			\leqslant \sum_{i=0}^\infty \abs{F_i} \,R^i
			< \infty.
		\end{equation*}
	\end{proof}

	\begin{lemma}[Formula for the Leray projector and its complement]
	\label{lemma:formula_complement_Leray_projector}
		Let $\PP_L$ denote the Leray projector on the torus.
		Then $\PP_L = -\nabla\times\Delta\inv\,\nabla\times$ and  $I - \PP_L = \nabla\Delta\inv\nabla\cdot$.
	\end{lemma}
	\begin{proof}
		This is immediate since $\hat{\PP}_L\brac{0} = I$ and $\hat{\PP}_L\brac{k} = I - \frac{k\otimes k}{\abs{k}^2}$ if $k\neq 0$
		and since $k\times k\times \cdot = \abs{k}^2 - k\otimes k$.
	\end{proof}

\section{Derivation of the perturbative energy-dissipation relation}
\label{sec:deriv_relative_energy_dissipation_relation}
	In this section we derive the energy-dissipation relation \eqref{eq:relative_energy_dissipation_relation},
	which is satisfied by solutions of \eqref{eq:stat_prob_lin_mom}--\eqref{eq:stat_prob_microinertia}.
	First recall that the Cauchy stress tensor $T$ and the couple stress tensor $M$ are defined in \eqref{eq:constitutive_equations}.
	We will write $T_{eq} = -\kappa\Omega_{eq}$ for the equilibrium version of the stress tensor.
	For simplicity we will also write $D_t \defeq \pdt + u \cdot\nabla$ for the advective derivative.
	The conservation of linear momentum \eqref{eq:stat_prob_lin_mom} can then be written as $D_t u = \nabla\cdot T$ such that multiplying by $u$ yields
	\begin{equation}
	\label{eq:full_nonlinear_pert_ED_lemma_ED_involving_u}
		\Dt \int_{\T^3} \frac{1}{2} \abs{u}^2
		= \int_{\T^3} D_t \brac{ \frac{1}{2} \abs{u}^2 }
		= \int_{\T^3} D_t u \cdot u
		= \int_{\T^3} \brac{\nabla\cdot T} \cdot u
		= - \int_{\T^3} T:\nabla u.
	\end{equation}
	Similarly, the conservation of angular momentum \eqref{eq:stat_prob_ang_mom} can be written as $$JD_t\omega + \sbrac{\Omega, J}\omega = 2 \vc\brac{T-T_{eq}} + \nabla\cdot M$$
	and hence multiplying by $\omega-\omega_{eq}$ yields
	\begin{equation}
	\label{eq:full_nonlinear_pert_ED_lemma_ED_involving_J_Dt_omega}
		J D_t \omega \cdot \brac{\omega - \omega_{eq}}
		+ \sbrac{\Omega, J} \omega \cdot \brac{\omega - \omega_{eq}}
		= 2\vc \brac{T - T_{eq}} \cdot \brac{\omega - \omega_{eq}}
		+ \brac{\nabla\cdot M} \cdot \brac{\omega - \omega_{eq}}.
	\end{equation}
	The right-hand side of \eqref{eq:full_nonlinear_pert_ED_lemma_ED_involving_J_Dt_omega} is dealt with in the usual way:
	\begin{equation}
	\label{eq:full_nonlinear_pert_ED_lemma_A}
		\int_{\T^3} 2\vc\brac{T-T_{eq}} \cdot \brac{\omega - \omega_{eq}}
			+ \brac{\nabla\cdot M} \cdot \brac{\omega - \omega_{eq}}
		= \int_{\T^3} \brac{T-T_{eq}} : \brac{\Omega - \Omega_{eq}}
			- M : \nabla\brac{\omega - \omega_{eq}}.
	\end{equation}
	Dealing with the left-hand side of \eqref{eq:full_nonlinear_pert_ED_lemma_ED_involving_J_Dt_omega} requires further rearranging.
	Using the fact that the conservation of micro-inertia \eqref{eq:stat_prob_microinertia} can be written as $D_t J = \sbrac{\Omega, J}$ and
	adding and subtracting $ \frac{1}{2} D_t J \brac{\omega-\omega_{eq}}\cdot\brac{\omega-\omega_{eq}}$ yields
	\begin{align}
		J D_t \omega \cdot \brac{\omega - \omega_{eq}} + \sbrac{\Omega, J} \omega \cdot \brac{\omega - \omega_{eq}}
		= D_t \brac{ \frac{1}{2} J\brac{\omega - \omega_{eq}} \cdot \brac{\omega - \omega_{eq}} }
		+ \frac{1}{2} D_t J \brac{\omega + \omega_{eq}} \cdot \brac{\omega - \omega_{eq}}.
		\label{eq:full_nonlinear_pert_ED_lemma_B}
	\end{align}
	The key observation that allows us to conclude is the identity
	$
		\sbrac{\Omega, J}\brac{\omega + v}\cdot\brac{\omega - v} = -\sbrac{\Omega, J}v\cdot v
	$
	for every $v\in\R^3$.
	Combining this identity with $D_t J = \sbrac{\Omega, J}$ tells us that
	\begin{equation}
	\label{eq:full_nonlinear_pert_ED_lemma_C}
		\frac{1}{2} D_t J \brac{\omega + \omega_{eq}} \cdot \brac{\omega - \omega_{eq}}
		= - \frac{1}{2} \brac{ D_t J } \omega_{eq} \cdot \omega_{eq}
		= - D_t \brac{ \frac{1}{2} J \omega_{eq} \cdot\omega_{eq} }.
	\end{equation}
	Finally: combining \eqref{eq:full_nonlinear_pert_ED_lemma_A}, \eqref{eq:full_nonlinear_pert_ED_lemma_B}, and \eqref{eq:full_nonlinear_pert_ED_lemma_C} yields
	\begin{equation*}
		\Dt\brac{
			\int_{\T^3}
				J\brac{\omega-\omega_{eq}}\cdot\brac{\omega-\omega_{eq}}
				- \frac{1}{2} J \omega_{eq} \cdot \omega_{eq}
		} = \int_{\T^3}
			\brac{T-T_{eq}} : \brac{\Omega - \Omega_{eq}}
			- M : \nabla\brac{\omega - \omega_{eq}}.
	\end{equation*}
	Adding this equation to \eqref{eq:full_nonlinear_pert_ED_lemma_ED_involving_u} yields the energy-dissipation relation \eqref{eq:relative_energy_dissipation_relation}.

\section{The 8-by-8 matrix $M$ in all its glory}
\label{sec:M_k}

	In this section we record the matrix $M_k$ in an explicit form.
	Recall that $M_k$ is introduced in \fref{Section}{sec:spec_ana}, and is written there in a compact form well-suited to the analysis of its spectrum.
	However, in order to compute the characteristic polynomial of $M$, we employed the assistance of a symbolic algebra package, and this thus requires providing an explicit form of the matrix $M_k$.
	$M_k$ can be written in block form as
	\begin{equation*}
		M_k = \begin{pmatrix}
			A		& B	& 0_{3\times 2}\\
			B^T		& C	& D\\
			0_{2\times 3}	& E	& F
		\end{pmatrix}
	\end{equation*}
	where
	\begin{align*}
		A =& -\brac{\mu + \kappa/2} \brac{\abs{k}^2 I_3 - k\otimes k}
		= \brac{\mu + \kappa/2} \begin{pmatrix}
			- k_2^2 - k_3^2		& k_1 k_2		& k_1 k_3\\
			k_1 k_2			& - k_1^2 - k_3^2	& k_2 k_3\\
			k_1 k_3			& k_2 k_3		& - k_2^3 - k_3^2\\
		\end{pmatrix},\\
		B =&\, \frac{\kappa}{\abs{k}} \brac{ \abs{k}^2 I_3 - k\otimes k} \diag\brac{ \lambda^{-1/2}, \lambda^{-1/2}, \nu^{-1/2}}
		\\
		=&\, \frac{\kappa}{\sqrt{k_1^2 + k_2^2 + k_3^2}} \begin{pmatrix}
			\brac{k_2^2 + k_3^2} / \sqrt{\lambda}	& - k_1 k_2 / \sqrt{\lambda}			& - k_1 k_3 / \sqrt{\nu}\\
			- k_1 k_2 / \sqrt{\lambda}		& \brac{k_1^2 + k_3^2} / \sqrt{\lambda}		& - k_2 k_3 / \sqrt{\nu}\\
			- k_1 k_3 / \sqrt{\lambda}		& - k_2 k_3 / \sqrt{\lambda}			& \brac{k_1^2 + k_2^2} / \sqrt{\nu}\\
		\end{pmatrix}\\
		D =& \frac{\tau}{2\kappa} \sqrt{1 - \frac{\nu}{\lambda}} \begin{pmatrix} 1&0 \\ 0&1 \\ 0&0 \end{pmatrix},\,
		E = \begin{pmatrix} 1&0&0 \\ 0&1&0 \end{pmatrix},\,
		F = \frac{\tau}{2\kappa} \begin{pmatrix} 0&-1 \\ 1&0 \end{pmatrix}
	\end{align*}
	and
	\begin{align*}
		C =& -\diag\brac{\lambda^{-1/2}, \lambda^{-1/2}, \nu^{-1/2}} \brac{
			2 \kappa I_3
			+ \brac{\alpha + \beta/3 - \gamma} k\otimes k
			+ \brac{\beta + \gamma} \abs{k}^2 I_3
		} \diag\brac{\lambda^{-1/2}, \lambda^{-1/2}, \nu^{-1/2}}
		\\
		&- \brac{1 - \frac{\nu}{\lambda}} \frac{\tau}{2\kappa} \brac{e_2 \otimes e_1 - e_1 \otimes e_2}
	\end{align*}
	such that
	\begin{align*}
		C_{11} &= - \lambda\inv \brac{ 2\kappa + \brac{\alpha + 4\beta / 3} k_1^2 + \brac{\beta + \gamma} \brac{k_2^2 + k_3^2}},
		&&C_{12} = - \lambda\inv \brac{\alpha + \beta/3 - \gamma} k_1 k_2 + \frac{\tau}{2\kappa} \brac{1 - \frac{\nu}{\lambda}},\\
		C_{22} &= - \lambda\inv \brac{ 2\kappa + \brac{\alpha + 4\beta / 3} k_2^2 + \brac{\beta + \gamma} \brac{k_1^2 + k_3^2}},
		&&C_{21} = - \lambda\inv \brac{\alpha + \beta/3 - \gamma} k_1 k_2 - \frac{\tau}{2\kappa} \brac{1 - \frac{\nu}{\lambda}},\\
		C_{33} &= -\nu\inv \brac{ 2\kappa + \brac{\alpha + 4\beta/3} k_3^2 + \brac{\beta + \gamma} \brac{k_1^2 + k_2^2}},
		&&C_{13} = C_{31} = - \lambda^{-1/2} \nu^{-1/2} \brac{\alpha + \beta/3 - \gamma} k_1 k_3,\\
		&&&C_{23} = C_{32} = -\lambda^{-1/2} \nu^{-1/2} \brac{\alpha + \beta/3 - \gamma} k_2 k_3.
	\end{align*}

\begingroup
\renewcommand{\addcontentsline}[3]{} 
	\bibliographystyle{alpha-bis}
	\bibliography{main}

\newcommand{\etalchar}[1]{$^{#1}$}
\begin{thebibliography}{GBRT13}

\bibitem[Ahl78]{ahlfors}
L.~V. Ahlfors.
\newblock {\em Complex analysis}.
\newblock McGraw-Hill Book Co., New York, third edition, 1978.
\newblock An introduction to the theory of analytic functions of one complex
  variable, International Series in Pure and Applied Mathematics.

\bibitem[AK71]{allen_kline_lubrication}
S.~Allen and K.~Kline.
\newblock {Lubrication Theory for Micropolar Fluids}.
\newblock {\em {Journal of Applied Mechanics}}, {38}({3}):{646--\&}, {1971}.

\bibitem[Arb00]{arbogast}
L.~F.~A. Arbogast.
\newblock Du calculs des d\'{e}rivations.
\newblock {\em Levrault, Strasbourg}, 1800.

\bibitem[AS74]{ahmadi_shahinpoor}
G.~Ahmadi and M.~Shahinpoor.
\newblock Universal stability of magneto-micropolar fluid motions.
\newblock {\em Internat. J. Engrg. Sci.}, 12:657--663, 1974.

\bibitem[BBR{\etalchar{+}}08]{beg_al_blood_flow}
O.~A. Beg, R.~Bhargava, S.~Rawat, K.~Halim, and H.~S. Takhar.
\newblock {Computational modeling of biomagnetic micropolar blood flow and heat
  transfer in a two-dimensional non-Darcian porous medium}.
\newblock {\em {Meccanica}}, {43}({4}):{391--410}, {AUG} {2008}.

\bibitem[BF13]{boyer_fabrie}
F.~Boyer and P.~Fabrie.
\newblock {\em {Mathematical tools for the study of the incompressible
  {N}avier-{S}tokes equations and related models}}, volume 183 of {\em {Applied
  Mathematical Sciences}}.
\newblock Springer, New York, 2013.

\bibitem[B{\L}96]{bayada_lukaszewicz_lubrication}
G.~Bayada and G.~{\L}ukaszewicz.
\newblock On micropolar fluids in the theory of lubrication. {R}igorous
  derivation of an analogue of the {R}eynolds equation.
\newblock {\em Internat. J. Engrg. Sci.}, 34(13):1477--1490, 1996.

\bibitem[BSAV17]{banerjee_souslov_et_at_chiral_active_fluids}
D.~Banerjee, A.~Souslov, A.~G. Abanov, and V.~Vitelli.
\newblock {Odd viscosity in chiral active fluids}.
\newblock {\em {Nature Communications}}, {8}, {Nov 17} {2017}.

\bibitem[CCD07]{chen_chen_dong}
J.~Chen, Z.-M. Chen, and B.-Q. Dong.
\newblock Uniform attractors of non-homogeneous micropolar fluid flows in
  non-smooth domains.
\newblock {\em Nonlinearity}, 20(7):1619--1635, 2007.

\bibitem[CM12]{chen_miao}
Q.~Chen and C.~Miao.
\newblock Global well-posedness for the micropolar fluid system in critical
  {B}esov spaces.
\newblock {\em J. Differential Equations}, 252(3):2698--2724, 2012.

\bibitem[DC09]{dong_chen}
B.-Q. Dong and Z.-M. Chen.
\newblock Asymptotic profiles of solutions to the 2{D} viscous incompressible
  micropolar fluid flows.
\newblock {\em Discrete Contin. Dyn. Syst.}, 23(3):765--784, 2009.

\bibitem[DZ10]{dong_zhang}
B.-Q. Dong and Z.~Zhang.
\newblock Global regularity of the 2{D} micropolar fluid flows with zero
  angular viscosity.
\newblock {\em J. Differential Equations}, 249(1):200--213, 2010.

\bibitem[Eri66]{eringen_first}
A.~C. Eringen.
\newblock {Theory of Micropolar Fluids}.
\newblock {\em {Journal of Mathematics and Mechanics}}, {16}({1}):{1--\&},
  {1966}.

\bibitem[Eri99]{erigen_vol_1}
A.~C. Eringen.
\newblock {\em Microcontinuum field theories. {I}. {F}oundations and solids}.
\newblock Springer-Verlag, New York, 1999.

\bibitem[Eri01]{erigen_vol_2}
A.~C. Eringen.
\newblock {\em Microcontinuum field theories. {II}. {F}luent media}.
\newblock Springer-Verlag, New York, 2001.

\bibitem[FP13]{ferreira_precioso}
L.~C.~F. Ferreira and J.~C. Precioso.
\newblock Existence of solutions for the 3{D}-micropolar fluid system with
  initial data in {B}esov-{M}orrey spaces.
\newblock {\em Z. Angew. Math. Phys.}, 64(6):1699--1710, 2013.

\bibitem[FSV97]{friedlander_strauss_vishik}
S.~Friedlander, W.~Strauss, and M.~Vishik.
\newblock Nonlinear instability in an ideal fluid.
\newblock {\em Ann. Inst. H. Poincar\'{e} Anal. Non Lin\'{e}aire},
  14(2):187--209, 1997.

\bibitem[FVR07]{ferreira_villamizar_roa}
L.~C.~F. Ferreira and E.~J. Villamizar-Roa.
\newblock Micropolar fluid system in a space of distributions and large time
  behavior.
\newblock {\em J. Math. Anal. Appl.}, 332(2):1425--1445, 2007.

\bibitem[GBRT13]{gay_balmaz_ratiu_tronci_liquid_crystals}
F.~Gay-Balmaz, T.~S. Ratiu, and C.~Tronci.
\newblock Equivalent theories of liquid crystal dynamics.
\newblock {\em Arch. Ration. Mech. Anal.}, 210(3):773--811, 2013.

\bibitem[GHS07]{guo_hallstrom_spirn}
Y.~Guo, C.~Hallstrom, and D.~Spirn.
\newblock Dynamics near unstable, interfacial fluids.
\newblock {\em Comm. Math. Phys.}, 270(3):635--689, 2007.

\bibitem[GR77]{galdi_rionero}
G.~P. Galdi and S.~Rionero.
\newblock A note on the existence and uniqueness of solutions of the micropolar
  fluid equations.
\newblock {\em Internat. J. Engrg. Sci.}, 15(2):105--108, 1977.

\bibitem[GS95a]{guo_strauss_bgk}
Y.~Guo and W.~A. Strauss.
\newblock Instability of periodic {BGK} equilibria.
\newblock {\em Comm. Pure Appl. Math.}, 48(8):861--894, 1995.

\bibitem[GS95b]{guo_strauss_double_humped}
Y.~Guo and W.~A. Strauss.
\newblock Nonlinear instability of double-humped equilibria.
\newblock {\em Ann. Inst. H. Poincar\'{e} Anal. Non Lin\'{e}aire},
  12(3):339--352, 1995.

\bibitem[Har06]{m_hardy}
M.~Hardy.
\newblock Combinatorics of partial derivatives.
\newblock {\em Electron. J. Combin.}, 13(1):Research Paper 1, 13, 2006.

\bibitem[KL{\L}19]{kalita_langa_lukaszewicz}
P.~Kalita, J.~A. Langa, and G.~{\L}ukaszewicz.
\newblock Micropolar meets {N}ewtonian. {T}he {R}ayleigh-{B}\'{e}nard problem.
\newblock {\em Phys. D}, 392:57--80, 2019.

\bibitem[LR04]{lhuillier_rey_liquid_crystals}
D.~Lhuillier and A.~Rey.
\newblock {Nematic liquid crystals and ordered micropolar fluids}.
\newblock {\em {Journal of Non-Newtonian Fluid Mechanics}},
  {120}({1-3}):{169--174}, {JUL 1} {2004}.
\newblock {3rd International Workshop on Nonequilibrium Thermodynamics and
  Complex Fluids, Princeton, NJ, AUG 14-17, 2003}.

\bibitem[Lu89]{lukaszewicz_89}
G.~\L~ukaszewicz.
\newblock On the existence, uniqueness and asymptotic properties for solutions
  of flows of asymmetric fluids.
\newblock {\em Rend. Accad. Naz. Sci. XL Mem. Mat. (5)}, 13(1):105--120, 1989.

\bibitem[{\L}uk90]{lukaszewicz_90}
G.~{\L}ukaszewicz.
\newblock On nonstationary flows of incompressible asymmetric fluids.
\newblock {\em Math. Methods Appl. Sci.}, 13(3):219--232, 1990.

\bibitem[{\L}uk99]{lukaszewicz_book}
G.~{\L}ukaszewicz.
\newblock {\em Micropolar fluids}.
\newblock Modeling and Simulation in Science, Engineering and Technology.
  Birkh\"{a}user Boston, Inc., Boston, MA, 1999.
\newblock Theory and applications.

\bibitem[{\L}uk01]{lukaszewiscz_01}
G.~{\L}ukaszewicz.
\newblock Long time behavior of 2{D} micropolar fluid flows.
\newblock {\em Math. Comput. Modelling}, 34(5-6):487--509, 2001.

\bibitem[LuT09]{lukaszewicz_tarasinska}
G.~\L~ukaszewicz and A.~Tarasi\'{n}ska.
\newblock On {$H^1$}-pullback attractors for nonautonomous micropolar fluid
  equations in a bounded domain.
\newblock {\em Nonlinear Anal.}, 71(3-4):782--788, 2009.

\bibitem[LZ16]{liu_zhang}
Q.~Liu and P.~Zhang.
\newblock Optimal time decay of the compressible micropolar fluids.
\newblock {\em J. Differential Equations}, 260(10):7634--7661, 2016.

\bibitem[Mau85]{maurya_biological_fluids}
R.~P. Maurya.
\newblock {Peripheral-Layer Viscosity and Microstructural Effects on the
  Capillary-Tissue Fluid Exchange}.
\newblock {\em {Journal of Mathematical Analysis and Applications}},
  {110}({1}):{59--73}, {1985}.

\bibitem[MK08]{mekheimer_kot_blood_flow}
K.~S. Mekheimer and M.~A.~E. Kot.
\newblock {The micropolar fluid model for blood flow through a tapered artery
  with a stenosis}.
\newblock {\em {Acta Mechanica Sinica}}, {24}({6}):{637--644}, {DEC} {2008}.

\bibitem[NS12]{rajasekhar_nicodemus_sharma_lubrication}
E.~R. Nicodemus and S.~C. Sharma.
\newblock {Performance Characteristics of Micropolar Lubricated
  Membrane-Compensated Worn Hybrid Journal Bearings}.
\newblock {\em {Tribology Transactions}}, {55}({1}):{59--70}, {2012}.

\bibitem[NST16]{nochetto_salgado_tomas_ferrohydrodynamics}
R.~H. Nochetto, A.~J. Salgado, and I.~Tomas.
\newblock The equations of ferrohydrodynamics: modeling and numerical methods.
\newblock {\em Math. Models Methods Appl. Sci.}, 26(13):2393--2449, 2016.

\bibitem[Ram85]{ramkissoon_air_bubble_blood_flow}
H.~Ramkissoon.
\newblock Flow of a micropolar fluid past a {N}ewtonian fluid sphere.
\newblock {\em Z. Angew. Math. Mech.}, 65(12):635--637, 1985.

\bibitem[RM97]{rojas_medar_marko}
M.~A. Rojas-Medar.
\newblock Magneto-micropolar fluid motion: existence and uniqueness of strong
  solution.
\newblock {\em Math. Nachr.}, 188:301--319, 1997.

\bibitem[Rud76]{rudin}
W.~Rudin.
\newblock {\em Principles of mathematical analysis}.
\newblock McGraw-Hill Book Co., New York-Auckland-D\"{u}sseldorf, third
  edition, 1976.
\newblock International Series in Pure and Applied Mathematics.

\bibitem[SSP82]{sinha_singh_prasad_human_joints}
P.~Sinha, C.~Singh, and K.~PrasaD.
\newblock {Lubrication of Human Joints - A Microcontinuum Approach}.
\newblock {\em {WEAR}}, {80}({2}):{159--181}, {1982}.

\bibitem[Tar06]{tarasinska}
A.~Tarasi\'{n}ska.
\newblock Global attractor for heat convection problem in a micropolar fluid.
\newblock {\em Math. Methods Appl. Sci.}, 29(11):1215--1236, 2006.

\bibitem[Whi65]{whittlesey}
E.~F. Whittlesey.
\newblock Analytic functions in {B}anach spaces.
\newblock {\em Proc. Amer. Math. Soc.}, 16:1077--1083, 1965.

\bibitem[Yua10]{yuan}
B.~Yuan.
\newblock On regularity criteria for weak solutions to the micropolar fluid
  equations in {L}orentz space.
\newblock {\em Proc. Amer. Math. Soc.}, 138(6):2025--2036, 2010.

\end{thebibliography}
\endgroup
\end{document}